\documentclass[leqno,10pt,oneside]{amsart}

\usepackage[latin1]{inputenc}
\usepackage[english]{babel}

\usepackage{mathtools}  
\usepackage{amssymb, amsthm, mathrsfs}
\usepackage{braket}     
\usepackage{esint}      
\usepackage{xfrac, nicefrac}

\usepackage{graphicx}
\usepackage{color}

\usepackage{enumitem}   
\usepackage{orcidlink} 
\usepackage[toc]{appendix}
\usepackage{csquotes}   
\usepackage{hyperref}   

\allowdisplaybreaks
\numberwithin{equation}{section}

\newcommand{\rn}{\mathbb{R}^n}
\newcommand{\R}{\mathbb{R}} \newcommand{\N}{\mathbb{N}} 
\newcommand{\dx}{\,dx} \newcommand{\dy}{\,dy}
\newcommand{\tail}{\,\texttt{Tail}}
\newcommand{\stackleq}[1]{\stackrel{\eqref{#1}}{\leq}}
\newcommand{\vs}{\medskip}

\def\A{\mathcal A}  \def\H{\mathcal H}
 \def\M{\mathcal M}
\def\e{\varepsilon} \def\s{\sigma} 
\def\vphi{\varphi} 
\def\l{\lambda} \def\g{\gamma} \def\k{\kappa}
\def\omegau{\omega_{\mathsf{u}}}
\def\omegax{\omega_{\mathsf{x}}}

\def\a{\alpha} \def\b{\beta} 
\def\vrho{\varrho} \def\d{\delta}

\def\mint_#1{\mathchoice%
          {\mathop{\kern 0.2em\vrule width 0.6em height 0.69678ex depth -0.58065ex
                  \kern -0.8em \intop}\nolimits_{\kern -0.4em#1}}%
          {\mathop{\kern 0.1em\vrule width 0.5em height 0.69678ex depth -0.60387ex
                  \kern -0.6em \intop}\nolimits_{#1}}%
          {\mathop{\kern 0.1em\vrule width 0.5em height 0.69678ex
              depth -0.60387ex
                  \kern -0.6em \intop}\nolimits_{#1}}%
          {\mathop{\kern 0.1em\vrule width 0.5em height 0.69678ex depth -0.60387ex
                  \kern -0.6em \intop}\nolimits_{#1}}}
                  
\newcommand{\aveint}[2]{\mathchoice
  {\mathop{\kern 0.2em\vrule width 0.6em height 0.7ex depth -0.6ex \kern -0.8em \intop}\nolimits_{\kern -0.45em#1}^{#2}}
  {\mathop{\kern 0.1em\vrule width 0.5em height 0.7ex depth -0.6ex \kern -0.6em \intop}\nolimits_{#1}^{#2}}
  {\mathop{\kern 0.1em\vrule width 0.5em height 0.7ex depth -0.6ex \kern -0.6em \intop}\nolimits_{#1}^{#2}}
  {\mathop{\kern 0.1em\vrule width 0.5em height 0.7ex depth -0.6ex \kern -0.6em \intop}\nolimits_{#1}^{#2}}}

\makeatletter
\DeclareFontFamily{OMX}{MnSymbolE}{}
\DeclareSymbolFont{MnLargeSymbols}{OMX}{MnSymbolE}{m}{n}
\SetSymbolFont{MnLargeSymbols}{bold}{OMX}{MnSymbolE}{b}{n}
\DeclareFontShape{OMX}{MnSymbolE}{m}{n}{
    <-6>  MnSymbolE5
   <6-7>  MnSymbolE6
   <7-8>  MnSymbolE7
   <8-9>  MnSymbolE8
   <9-10> MnSymbolE9
  <10-12> MnSymbolE10
  <12->   MnSymbolE12
}{}
\DeclareFontShape{OMX}{MnSymbolE}{b}{n}{
    <-6>  MnSymbolE-Bold5
   <6-7>  MnSymbolE-Bold6
   <7-8>  MnSymbolE-Bold7
   <8-9>  MnSymbolE-Bold8
   <9-10> MnSymbolE-Bold9
  <10-12> MnSymbolE-Bold10
  <12->   MnSymbolE-Bold12
}{}

\let\llangle\@undefined
\let\rrangle\@undefined
\DeclareMathDelimiter{\llangle}{\mathopen}%
                     {MnLargeSymbols}{'164}{MnLargeSymbols}{'164}
\DeclareMathDelimiter{\rrangle}{\mathclose}%
                     {MnLargeSymbols}{'171}{MnLargeSymbols}{'171}
\makeatother

\newcommand{\linethrough}{\mathpalette\@thickbar}
\makeatletter
\newcommand{\@thickbar}[2]{{#1\mkern0mu\vbox{
  \sbox\z@{$#1#2\mkern-0.5mu$}%
  \dimen@=\dimexpr\ht\tw@-\ht\z@+2\p@\relax
  \hrule\@height0.5\p@
  \vskip\dimen@ \box\z@}}}
\makeatother

\newcommand{\data}{\textnormal{\texttt{data}}}

\newcommand{\smallzero}{{\scalebox{.8}{$\scriptscriptstyle 0$}}}

\makeatletter
\newcommand{\pushright}[1]{\ifmeasuring@#1\else\omit\hfill$\displaystyle#1$\fi\ignorespaces}
\newcommand{\pushleft}[1]{\ifmeasuring@#1\else\omit$\displaystyle#1$\hfill\fi\ignorespaces}
\makeatother

\newtheorem{theorem}{Theorem}[section]
\newtheorem*{theorem*}{Theorem}
\newtheorem{lemma}[theorem]{Lemma}
\newtheorem{proposition}[theorem]{Proposition}
\newtheorem*{proposition*}{Proposition}

\newtheorem{remark}[theorem]{Remark}
\newtheorem*{remark*}{Remark}

\newcommand{\todo}[1]{\text{\colorbox{yellow}{#1}}}

\title{Partial regularity for vectorial local-nonlocal problems}

\begin{document}

\begin{abstract}
    We establish sharp interior regularity for solutions to vectorial, nonlinear local-nonlocal systems.
\end{abstract}

\author[Antonini]{Carlo Alberto Antonini \orcidlink{0000-0002-7663-1090}}  \address{Carlo Alberto Antonini \\ 
Dipartimento di Matematica ``Federico Enriques'', 
Universit\`a degli studi di Milano, 
Viale Cesare Saldini 50, 20133,
Milan,
Italy\\ ORCID ID: 0000-0002-7663-1090}
\email{\url{carlo.antonini@unimi.it}\\ \url{antonini@altamatematica.it}}

\author[Baroni]{Paolo Baroni \orcidlink{0000-0002-2683-4025}}
\address{Paolo Baroni \\ Dipartimento di Scienze Matematiche, Fisiche e Informatiche, Universit\`a di Parma, Parco Area delle Scienze, 53/A, 43124 Parma, Italy\\ ORCID ID: 0000-0002-2683-4025}
\email{\url{paolo.baroni@unipr.it}}

\subjclass[2020]{Primary 35B65; Secondary 35R11, 35J47, 35M10}

\maketitle


\section{Introduction}

Operators that couple a local and a nonlocal diffusion term naturally arise in the modeling of processes combining continuous and jump-type behavior~\cite{garroni}, and in recent years they have been studied with increasing attention, both for their analytical structure and for this modeling relevance. The prototype of this class is the case where a local Laplacian diffusion is coupled with a fractional Laplacian diffusion
\[
-\Delta u+(-\Delta)^s u .
\] 

For this class of operators, several aspects of the elliptic and parabolic theory have by now been thoroughly investigated, including Green function estimates, boundary Harnack principles, maximum principles, and regularity properties, both in the elliptic and in the parabolic setting; since the literature on the linear theory is by now vast, we only refer  to~\cite{CKSV10,CKSV12,BDVV22,BDVV23b,das,SVWZ25} and the references therein for a detailed account. More recently, attention has turned to nonlinear counterparts: first to models preserving the linear local structure while allowing more general stochastic nonlocal terms~\cite{BM21}, and then, more substantially, to mixed quasilinear operators modeled on
\[
-\Delta_p u + (-\Delta_q)^s u, \qquad p,q\in(1,\infty), \ s\in(0,1);
\]
a number of existence and qualitative results have been obtained, covering both variational and nonvariational formulations; see, for instance,~\cite{AC25,apaza,BP23,CSYZ,dS20,ding,GK22,GL}.  The regularity theory for these operators is shaped by the balance between the local and nonlocal components, and this balance in turn depends on the relation between $p$, $q$ and $s$; the present paper adds to this general picture by showing how this interplay manifests itself for vector-valued solutions with quadratic growth.  In the scalar case, this balance has already been made precise: De Filippis \& Mingione indeed established in~\cite{dmmixed} the interior H\"older continuity of $Du$ via perturbation methods under the structural condition $p>sq$, as this assumption ensures that at small scales the local operator $-\Delta_p u$ dominates the nonlocal one $(-\Delta_q)^s u$,  so that the regularizing effect of the equation is essentially governed by its local part. Consequently, solutions inherit the same type of regularity of solutions to purely local operators, namely the $C^{1,\alpha}$-regularity for some non-explicit $\alpha\in (0,1)$, which does not exceed the optimal H\"older exponent of the gradient of $p$-harmonic functions \cite{A26,BBDS26,DB83}. The gradient H\"older regularity in the complementary regime $p\leq sq$ was recently treated by Biswas and Topp~\cite{BT26} through viscosity methods. 

\vs

We initiate, in this paper, the study of quasilinear vectorial mixed operators by treating the quadratic growth case, namely $p=q=2$, allowing for a class of local and nonlocal terms more general than the model operators, see \eqref{eq1}--\eqref{eqNP} and the related assumptions \eqref{phi:gr} and \eqref{sym:phi}. Although this is only the first step toward the more general theory we aim to develop, it already exhibits the essential difficulties of the vectorial local-nonlocal interaction, whose treatment requires a substantially new combination of classic local partial regularity methods, developed for the vectorial setting, and nonlocal tail estimates, adapted here to interact with them. We show that regularity of $u$ and of its gradient $Du$ hold on the very same regular set, whose singular complement has Hausdorff dimension explicitly bounded, and that on this set a full scale of results is available, improving as the assumptions on the local term are strengthened: from BMO and VMO regularity of $u$, up to its continuity and H\"older continuity, and finally to Lipschitz continuity, continuity and H\"older continuity of the gradient itself; none of these improvements requires any additional regularity on the nonlocal term, which is allowed to remain as irregular and non-coercive as in our basic assumptions throughout. 

\newpage

Our main novelties can be summarized as follows:
\begin{itemize}
\item we work in the vectorial setting and study the partial regularity of solutions (Theorems \ref{part.BMO} to \ref{thm:c1bfull}), obtaining everywhere regularity under additional structural assumptions on the local term (Theorem~\ref{thm:full-u-independent});
\item we allow for very general nonlocal operators, possibly irregular and non-coercive (see \eqref{phi:gr}--\eqref{sym:phi});
\item we obtain sharp regularity results in terms of the right-hand side of the equation, with explicit H\"older exponents together with quantitative local estimates that also keep track of the nonlocal tail of the solution (for instance, see \eqref{ass:fhold} and \eqref{f:1beta} and the related regularity results).
\end{itemize}

The present paper provides, to the best of our knowledge, the first systematic development of the vectorial regularity theory for mixed local-nonlocal systems. Its closest predecessor is the recent vectorial theory for purely nonlocal systems developed by De Filippis, Mingione \& Nowak in \cite{dmn}, where classical tools from the local partial regularity theory -- excess decay, tail estimates, blow-up procedures and harmonic approximation -- were implemented for the first time in a genuinely nonlocal vectorial setting. Here we show that, in the mixed regime, the local component can be exploited as the leading operator, while the nonlocal part can be treated as a lower-order perturbation even under minimal assumptions, namely without regularity or coercivity on the integrand, so that under appropriate assumptions on the local part solutions retain the highest regularity dictated by the local theory. This dominance, however, has a natural limit at the level of the gradient: under the sole structural assumptions on the nonlocal term, the H\"older exponent of $Du$ cannot in general exceed a threshold dictated by $s$, reflecting the fact that differentiating the equation makes the lack of regularity of the nonlocal term visible; remarkably, this threshold turns out to be sharp, matching the optimal boundary regularity known for nonlocal equations. When  the nonlocal term is assumed to have the natural structure of the fractional Laplacian, this limitation disappears entirely, and the full, sharp H\"older regularity of the gradient is recovered.

\bigskip

The precise setting we work in is the following: let $n,N\in\mathbb N$, with $n\ge2$ and $N\ge1$, let $\Omega\subset\mathbb R^n$ be open, let $u:\mathbb R^n\to\mathbb R^N$, and let $f:\Omega\to\mathbb R^N$ be a measurable function. We deal with systems of the form
\begin{equation}\label{eq1}
    Q^\mathrm{l}_A u+Q_\Phi^\mathrm{nl} u=f\qquad\text{in $\Omega$,}
\end{equation}
where $Q^\mathrm{l}_A$ is the local operator given by
\begin{equation*}
   Q_A^\mathrm{l} u(x)
   =-\mathrm{div}\big(A(x,u)Du \big)\,,
\end{equation*}
and $Q^\mathrm{nl}_\Phi$ is the nonlocal operator defined by 
\begin{equation}\label{eqNP}
    Q^{\mathrm{nl}}_\Phi u(x):=2\,\mathrm{P.V.}\int_{\rn}\Phi\big(x,y,u(x),u(y),u(x)-u(y)\big) \,\frac{dy}{|x-y|^{n+2s}},\qquad s\in(0,1)\,;
\end{equation}
here $\mathrm{P.V.}$ denotes the Cauchy principal value.

\vs
 
The coefficients $A(x,z)=\{A_{ij}^{\a\beta}(x,z)\}^{\a,\beta=1,\dots,N}_{i,j=1,\dots,n}:\Omega\times \R^N\to\R^{nN\times nN}$ are Carath\'eodory functions satisfying the Legendre and the  boundedness conditions
\begin{equation}\label{A:gr}
\langle A(x,z)\xi,\xi\rangle\geq \Lambda^{-1}\,|\xi|^2\,,\qquad  |A(x,z)|  \leq \Lambda\,,
\end{equation}
for almost all $x\in\Omega$, every $z\in \R^N$ and $\xi=\{\xi_i^\alpha\}_{i=1,\dots n}^{\alpha=1,\dots,N}\in \R^{nN}$, for some constant $\Lambda\geq 1$. We refer to Paragraph \ref{nota}  for the relevant notational conventions.

As for the nonlocal term $\Phi: \R^{2n+3N}\to \R^N$, $\Phi(x,y,u,v,t)={\{\Phi^\a(x,y,u,v,t)\}}_{\a=1,\dots,N}$, we assume that it has linear growth in the $t$-variable
\begin{equation}\label{phi:gr}
    |\Phi(x,y,u,v,t)|\leq \nu|t|
\end{equation}
for some $\nu>0$, and it satisfies the symmetry properties
\begin{equation}\label{sym:phi}
\begin{cases}
    \Phi(x,y,u,v,t) =\Phi(y,x,v,u,t)    \\
    \Phi(x,y,u,v,-t)=-\Phi(x,y,u,v,t)
    \end{cases},
\end{equation}
for almost every $x,y\in \R^n$, all $u,v\in \R^N$ and $t\in \R^N$; no additional regularity is required on the nonlocal term.

\vs

We consider weak solutions to equation~\eqref{eq1}, i.e., functions  $u\in W^{1,2}_{\rm loc}(\Omega;\R^N)\cap L^1_{2s}$ that satisfy the weak formulation of \eqref{eq1}:
\begin{multline}\label{weak.formulation}
    \int_\Omega \langle A(x,u)D u, D\vphi\rangle\dx    +  \int_{\R^n}\int_{\R^n}\Phi\big(x,y,u(x),u(y), u(x)-u(y)\big)\cdot \big(\vphi(x)-\vphi(y)\big)\,\frac{dx\dy}{|x-y|^{n+2s}}\\
     =\int_\Omega f\cdot \vphi\dx
\end{multline}
for all smooth test functions with compact support in $\Omega$, $\vphi\in C^\infty_c(\Omega;\R^N)$. Here, we employ the notation $L^1_{2s}$ for the tail space
\begin{equation*}
    L^1_{2s}\coloneqq \bigg\{w\in L^1_{\rm loc}(\R^n;\R^N):\int_{\rn}\frac{|w(x)|}{1+|x|^{n+2s}}\dx<\infty \bigg\}\,.
\end{equation*}

To better illustrate the main results of the paper, we state the following two model theorems, the first concerning the H\"older continuity of solutions.
\begin{theorem}[Partial H\"older continuity]\label{thm:model}
    Let $u\in  W^{1,2}_{\rm loc}(\Omega;\R^N)\cap L^1_{2s}$ be a weak solution to \eqref{eq1}, with \eqref{A:gr}--\eqref{sym:phi} in force. Suppose that $A(\cdot,\cdot\cdot)$ satisfies  the continuity assumption
\begin{equation}\label{es:cont}
        |A(x,u)-A(y,v)|\leq \Lambda\,\omega\big(|x-y|+|u-v|\big)\quad\text{for all $x,y\in \Omega$  and $u,v\in \R^N$,}
    \end{equation}
    for some modulus of continuity $\omega:[0,\infty)\to[0,1]$. Then  there exists an open set $\Omega_u\subseteq \Omega$ with Hausdorff dimension $\dim_{\mathcal{H}}(\Omega \setminus \Omega_u) \leq n-2-\zeta$ for some $\zeta>0$ if $n\geq3$, and $\Omega_u=\Omega$ if $n=2$ such that, if
    \begin{equation*}
        f\in L_{\rm loc}^{n/(2-\beta_0)}(\Omega;\R^N)\quad\text{for some $\beta_0\in (0,1)$}\,,
    \end{equation*} 
then
    \[
u\in C^{0,\beta_0}_{\rm loc}\big(\Omega_u;\R^N \big)\,.
    \]
    If in addition $A(x,u)\equiv A(x)$ is independent of $u$, then full regularity holds, i.e., $\Omega_u=\Omega$.
\end{theorem}

The next model theorem addresses the quantitative H\"older continuity of the gradient.

\begin{theorem}[Partial gradient H\"older continuity]\label{thm:model2}
     Let $u\in  W^{1,2}_{\rm loc}(\Omega;\R^N)\cap L^1_{2s}$ be a weak solution to \eqref{eq1}, with \eqref{A:gr}--\eqref{sym:phi} in force. Let $\beta\in (0,1)$, and suppose that $A(\cdot,\cdot\cdot)$ fulfils \eqref{es:cont} with $\omega(r)=r^\beta$. If
\begin{equation*}
f\in L^{n/(1-\beta) }_{\rm loc}(\Omega;\R^N)\,,
\end{equation*}
        then
\[
       u\in C^{1,\min\{\beta,2(1-s)\}}_{\rm loc}(\Omega_u;\R^N)\,,
\]
with $\Omega_u$ as in Theorem \ref{thm:model}. In particular,  if $A(x,u)\equiv A(x)$, then full gradient regularity holds. If in addition the nonlocal term is of the form
\begin{equation}\label{form:phi}
	\Phi(x,y,u,v,t)=a(x-y)t
\end{equation}
for some matrix-valued function  $a:\R^n\to \R^{N\times N}$ such that
\begin{equation}\label{form:phi1}
	a(z)=a(-z)\quad\text{and}\quad|a(z)|\leq \nu, \text{ for all } z\in \R^n\,,
\end{equation} then
\begin{equation}\label{full:aaa}
        u\in C^{1,\beta}_{\rm loc}(\Omega_u;\R^N)\,.
\end{equation}
\end{theorem}

We refer to Section \ref{sec:mainresults} for the precise statements of the results, where we also provide finer quantitative estimates. We now briefly comment on these results. Theorem~\ref{thm:model} provides a quantitative H\"older continuity result for solutions, in analogy with the purely local equation~\cite{giaq}:
\[
-\mathrm{div}\big(A(x,u)\,Du \big)=f\in L_{\rm loc}^{n/(2-\beta_0)}(\Omega;\R^N)
\quad\implies \quad
u\in C^{0,\beta_0}_{\rm loc}(\Omega_u;\R^N)\,.
\]
This result is essentially sharp, as can be seen by a simple scaling argument\footnote{See Theorem~\ref{thm:hold}, where a slightly weaker integrability condition on $f$ is assumed in terms of a Marcinkiewicz space.}.
This behavior is not surprising since, as already mentioned, the local operator $Q_A^\mathrm{l}$ represents the leading term in~\eqref{eq1}, and therefore it is this component that ultimately dictates the regularity of solutions. 

Another key feature of our approach, which relies precisely on the dominance of the local operator, is that we do not impose regularity or coercivity assumptions on the nonlocal term $\Phi$. In particular, we do not require the usual coercivity condition 
\[
\Phi(x,y,u,v,t)\cdot t \ge c\,|t|^2\,,\qquad c>0\,;
\]
this is particularly relevant in applications, where the nonlocal term may exhibit irregular or poorly structured behavior, see for instance~\cite{garroni}.

We also recall that everywhere regularity may fail in the vectorial setting, even when the coefficients $u\mapsto A(x,u)$ are analytic~\cite[Chapter~II.3]{giaq}. On the other hand, when the $u$-dependence on the coefficients is dropped, everywhere regularity is then recovered.

\vs

Turning to gradient regularity, Theorem~\ref{thm:model2} establishes the H\"older continuity of $Du$ with an explicit H\"older exponent dictated by the integrability of $f$. However, under such low assumptions on the nonlocal term, the exponent cannot in general exceed the threshold $C^{1,2(1-s)}$ in the regime $s>1/2$ for such general nonlocal terms $Q^\mathrm{nl}_\Phi$. This reflects the fact that, in the mildly local regime $s>1/2$, the lack of regularity of the nonlocal term influences the regularity of $u$ at the level of the gradient, in view of the fact that the proof implicitly involves differentiating equation~\eqref{eq1}. We also emphasize that the threshold $2(1-s)$ is consistent with the optimal regularity expected in this setting. In particular, it coincides with the sharp boundary regularity exponent established for solutions to nonlocal equations, see \cite{AAC}, and this highlights the role of the nonlocal operator in limiting the attainable gradient regularity, even when the equation exhibits a predominantly local character.

On the other hand, when we impose structural assumptions on $\Phi$, namely when it has the form described in~\eqref{form:phi}--\eqref{form:phi1}, this limitation on the exponent disappears in \eqref{full:aaa}; thus, we recover the full, sharp H\"older continuity of the gradient
\begin{equation*}
    f\in L^{n/(1-\beta)}_{\rm loc}(\Omega;\R^N)\quad\implies\quad u\in C^{1,\beta}_{\rm loc}(\Omega_u;\R^N)\quad\text{for all $\beta\in (0,1)$.}
\end{equation*}
The prototypical example of $\Phi$ fulfilling \eqref{form:phi}--\eqref{form:phi1} is the fractional Laplacian $(-\Delta)^s u$, for which $\Phi(x,y,u,v,t)=\Phi(t)=t$\,.

\subsection*{Outline of the proof.} 
The proof proceeds in several steps and follows, in principle, a perturbative approach as in \cite{AC25,dmmixed}, but carried out in a finer and more quantitative form, making the interplay between the local and nonlocal operators more explicit.

As in the classical regularity theory for purely local systems, the first step consists in establishing higher integrability of the gradient, namely $Du\in L^{2+\zeta}_{\rm loc}$ in  Proposition~\ref{prop:highint}. The delicate point here is to keep explicit track of the dependence on the growth factor $\nu$ associated with the nonlocal term in~\eqref{phi:gr}, see Remark \ref{rem:r}. 

The next step is to prove an approximate harmonicity property for solutions--see Proposition~\ref{prop:a}. The key observation is that if $u$ solves~\eqref{eq1} in the ball $B_\vrho(x_0)\Subset \Omega$, then the blow-up function $u_\vrho(x):=u(x_0+\vrho x)$ solves a similar equation in which the nonlocal term is multiplied by the factor $\vrho^{2(1-s)}$, see equation~\eqref{sol:blow}.  More explicitly, in the case $f=0$, the function $u_\vrho$ is solution of
\begin{equation*}
  Q^\mathrm{l}_{A_\vrho} u_\vrho+\vrho^{2(1-s)}Q^\mathrm{nl}_{\Phi_\vrho} u_\vrho=0\qquad\text{in $B_1$}
\end{equation*}
with coefficient matrix $A_{\vrho}$ and $\Phi_\vrho$ still satisfying the structural assumptions \eqref{A:gr} to \eqref{sym:phi}. Consequently, by choosing the radius $\vrho$ sufficiently small, the nonlocal contribution in this equation can be made arbitrarily small, see Lemma~\ref{lemma:smnloc}. 

At the same time, the higher integrability result obtained in the first step guarantees that the $L^{2+\zeta}$-norm of $Du_\vrho$ remains uniformly controlled for any small radii $\vrho$ -- this is the reason of the care in tracking the dependence on the size of the nonlocal term $\nu$. This allows us to apply the approximate harmonicity Lemma~\ref{lem:harapprox}, see also \eqref{small:vr}; hence $u$ is quantitatively close to a harmonic function, which in turn leads to a quantitative decay estimate for the excess functional -- see Paragraph \ref{par:apprharm} and Proposition~\ref{dec:one}. 

From this point on, the proof of the regularity of $u$ follows a more standard route. In particular, we exploit the recently developed machinery of potential theory to control the source term $f$, together with suitable decay estimates for the nonlocal tail over shrinking balls, see Lemma~\ref{lemma:tail} and Lemma~\ref{lemma:affine}.




\subsection*{Plan of the paper} In Section \ref{sec:mainresults} we state the main results of the paper in full detail, while Section \ref{sec:prelim} provides the notation and some preliminary results. In Section \ref{sec:energy} we prove the Caccioppoli inequality and the higher integrability of solutions. Section \ref{linearization} is devoted to the proof of approximate harmonicity of solutions and the excess decay estimate, which will be the starting point of the proofs for the main results.
In Section \ref{sec:regset} we characterize the regular set of solutions and we provide the proof of the results involving the regularity of $u$, while in Section \ref{sec:gradreg} we study the regularity of the gradient $Du$. Finally, in Section \ref{sec:everywhere} we discuss in detail the everywhere regularity of solutions in the case of $u$-independent coefficients.

\section{Main results}\label{sec:mainresults}
In this section we state in full detail the results of the paper.  Unless otherwise stated, we shall assume that
\begin{equation}\label{ass:f}
    f\in L^\chi_{\rm loc}(\Omega;\R^N)\quad \text{for some }\chi>2_*\,,
\end{equation}
where we set
\begin{equation}\label{def:2}
    2_*=\begin{cases}
\displaystyle{\frac{2n}{n+2}}\quad & n>2
        \\[3mm]
        \text{any number in $(1,2)$}\quad & n=2
    \end{cases},\qquad\quad
    2^*=\begin{cases}
\displaystyle{\frac{2n}{n-2}}\quad & n>2
        \\[3mm]
        \text{any number in $(2,+\infty)$}\quad & n=2
    \end{cases}.
\end{equation}
The exponent \(\chi\) can be auxiliary: in the results where \(f\) is assumed to belong to a Marcinkiewicz space, \(\chi\) is chosen below the corresponding Marcinkiewicz exponent, as explained after \eqref{density.Marc}. We denote by  $\data$ the set of parameters
\begin{equation}\label{data}
    \data=\{n,N,s,\Lambda,\chi\}\,.
\end{equation}
Also, we shall always assume that $A(x,\cdot)$ is uniformly continuous; namely,  there exists a modulus of continuity $\omegau(\,\cdot\,)$ such that
\begin{equation}\label{mod:cont.u}
\big|A(x,u)-A(x,v)\big|\leq \Lambda\,\omegau\big(|u-v| \big)\qquad\text{for almost all $x\in \Omega$ and $u,v\in \R^N$}.
\end{equation}
Our first two main results are concerned with BMO and VMO regularity of solutions to \eqref{weak.formulation}. These are stated in terms of a smallness assumption on the $x$-mean oscillation of $A(\cdot,\cdot\cdot)$. Specifically, for any ball $B_{r}(x_{0})\subseteq\Omega$, we define the averaged vector field
\begin{equation}\label{A:mean}
{(A)}_{B_{r}(x_{0})}(z):=\mint_{B_{r}(x_{0})}A(x,z)\dx\,;
\end{equation}
then, for every $x_0\in\Omega$ and for some $r_0>0$, we define
\begin{equation}\label{E_A}
E_{r_0}(A;x_0)=\sup_{\substack{\vrho\in(0,r_0]:B_{\vrho}(x_{0})\subseteq\Omega,\\ z\in\R^N,z\neq 0}} \mint_{B_{\vrho}(x_{0})}\big|{A(x,z)-{(A)}_{B_{\vrho}(x_{0})}(z)}\big|\dx\,,
\end{equation} 
the $x${\em -mean oscillation of} $A(\cdot,\cdot\cdot)$. 
\begin{theorem}[Partial BMO Regularity]\label{part.BMO}
Let $u\in W^{1,2}_{\rm loc}(\Omega;\R^N)\cap L^1_{2s}$ be a local weak solution to \eqref{eq1}. Suppose that $A(\cdot,\cdot)$ satisfies \eqref{A:gr}, \eqref{mod:cont.u}, and that $\Phi$ fulfills \eqref{phi:gr}, \eqref{sym:phi}. Then there exist constants
\[
\varepsilon_b=\varepsilon_b\big(\data,\omegau(\,\cdot\,)\big)\in(0,1),
\qquad
\delta_0=\delta_0(\data)\in(0,1),
\]
such that the following implication holds: if for every compact set $K\Subset\Omega$ there exist radii $r_A,r_f\in(0,{\rm dist}(K,\partial\Omega))$ such that
\begin{equation}\label{x0:BMOd}
\sup_{x_0\in K}E_{r_A}(A;x_0)<\delta_0,
\end{equation}
and
\begin{equation}\label{ass.BMO.f0}
\sup_{\substack{x_0\in K,\\ 0<t\leq r_f}} t^2 \bigg( \mint_{B_t(x_0)}|f|^\chi\,dx \bigg)^{1/\chi} <\varepsilon_b,
\end{equation}
where $\chi$ is as in \eqref{ass:f}, then there exists an open set $\Omega_u\subseteq\Omega$ such that $\Omega_u=\Omega$ if $n=2$, while for $n\ge3$,
\[
\dim_{\mathcal H}(\Omega\setminus\Omega_u)\leq n-2-\zeta
\]
with $\zeta=\zeta(\data,\max\{\nu,1\})\in(0,1)$, and
\begin{equation}\label{local.BMO}
u\in BMO_{\rm loc}(\Omega_u;\R^N)\,.
\end{equation}
\end{theorem}
Under assumptions stronger than \eqref{ass.BMO.f0}, we are also able to obtain VMO-regularity of solutions.

\begin{theorem}[Partial VMO Regularity]\label{teo.partVMO}
Under the assumptions and notations of Theorem \ref{part.BMO}, if $f$ satisfies
\begin{equation}\label{ass.BMO.f}
\lim_{\vrho\to0^+}\vrho^2 \bigg(\mint_{B_\vrho(x)} |f|^\chi \dx\bigg)^{1/\chi} =0\quad\text{locally uniformly in $x\in \Omega$,}
\end{equation}
$\chi$ as in \eqref{ass:f}, then 
\[u \in VMO_{\rm loc}(\Omega_u; \R^N)\,.
\]
\end{theorem}

Our next results deal with continuity of solutions and their gradient. In order to state them, we introduce the \textit{(truncated) Riesz-type potential} of $f$:
\begin{equation}\label{dfriesz:pot}
    \mathbf{I}^{f}_{\alpha,\chi}(x_0,\sigma):=\int_0^\sigma \vrho^\alpha\bigg( \mint_{B_\vrho(x_0)}|f|^\chi\dx\bigg)^{1/\chi}\,\frac{d\vrho}{\vrho}
\end{equation}
for  $x_0\in \Omega$ and $\sigma>0$. For $\chi=1$, the definition reduces to the averaged form of the classical truncated Riesz potential, which can be evaluated both on functions and, by replacing $|f|\,\dx$ with a measure, on locally finite signed measures. For $\chi>1$, $\mathbf I^f_{\alpha,\chi}$ is a $\chi$-averaged counterpart of the truncated Riesz potential and retains the same natural scaling properties, mutatis mutandis; see \cite[Chapter 4]{dmInv}.

\begin{theorem}[Partial continuity]\label{cont.t}
Under the assumptions and notations of Theorem \ref{part.BMO}, suppose that $f$ fulfills
\begin{equation}\label{I1.uniformly}
    \lim_{\sigma\to 0^+}\, \mathbf{I}^{f}_{2,\chi}(\cdot,\sigma)=0\quad \mbox{locally uniformly in $\Omega$}\,;
\end{equation}
then
\[
    u\in C^0(\Omega_u;\R^N)\,.
\]
\end{theorem}

For what concerns H\"older continuity of solutions, this can be quantified whenever the right-hand side of \eqref{eq1} belongs to a certain Marcinkiewicz, or weak-Lebesgue, space $\mathcal M^m(\Omega;\R^N)$ (often denoted by $ L^{m,\infty}(\Omega;\R^N)$).  Specifically, we say that  $f\in \mathcal M^m(\Omega;\R^N)$, for $m\geq1$, if 
\begin{equation}\label{def:marcke}
{\|f\|}_{\mathcal M^m(\Omega;\R^N)}^m\coloneqq\sup_{\lambda>0}\lambda^m \big|\big\{x\in \Omega:|f(x)|\geq\lambda\big\}\big|<\infty\,.
\end{equation}
We recall that  standard interpolation theory \cite[Theorem 3.18.8]{pick} yields
\begin{equation}\label{emb:weak}
    f\in L^{m}(\Omega;\R^N)\implies f\in \M^{m}(\Omega;\R^N),\quad m\in [1,\infty) 
\end{equation}
and $\M^{m}(\Omega;\R^N)\subseteq L^{m-\e}(\Omega;\R^N)$ for $\e\in(0,m-1]$, see also \eqref{density.Marc}.

\begin{theorem}[Quantified partial H\"older regularity]\label{thm:hold}
Assume that \eqref{A:gr}--\eqref{sym:phi} and \eqref{mod:cont.u} are in force, and let $u\in W^{1,2}_{\rm loc}(\Omega;\R^N)\cap L^1_{2s}$ be a local weak solution to \eqref{eq1}, with $s\in(0,1)$. Let $\beta_0\in(0,1)$ and assume that
\begin{equation}\label{ass:fhold}
    f\in \mathcal M^{n/(2-\beta_0)}_{\rm loc}(\Omega;\R^N)\,.
\end{equation}
Then there exists a constant $\delta_0=\delta_0(\data,\beta_0)\in(0,1)$ with the following property. If, for every compact set $K\Subset\Omega$, there exists a radius $r_A=r_A(K)>0$ such that
\begin{equation}\label{satip}
    \sup_{x_0\in K} E_{r_A}(A;x_0)<\delta_0,
\end{equation}
then there exists an open set $\Omega_u\subseteq\Omega$ such that $\Omega_u=\Omega$ if $n=2$, while, if $n\geq3$,
\[
\dim_{\mathcal H}(\Omega\setminus\Omega_u)\leq n-2-\zeta,\qquad \zeta=\zeta(\data,\max\{\nu,1\})\in(0,1)\,,
\]
and
\[
u\in C^{0,\beta_0}_{\rm loc}(\Omega_u;\R^N)\,.
\]
Finally, for every $x_0\in \Omega_u$, there exist radii $r_{x_0}\leq \vrho_{x_0}/2$ such that $B_{2\vrho_{x_0}}(x_0)\Subset \Omega$, and
{\rm
\begin{multline}\label{local.Holder1}
[u]_{C^{0,\beta_0}(B_{r_{x_0}}(x_0))}\leq c\,\vrho_{x_0}^{-\beta_0}\bigg( \mint_{B_{\vrho_{x_0}}(x_0)} {\big|u - (u)_{B_{\vrho_{x_0}}(x_0)}\big|}^2 \dx \bigg)^{1/2}\\
+c\,\vrho_{x_0}^{1-\beta_0-2s}\tail\big(u-(u)_{B_{\vrho_{x_0}}(x_0)};B_{\vrho_{x_0}}(x_0)\big)+c\,{\|f\|}_{\M^{n/(2-\beta_0)}(B_{2\vrho_{x_0}}(x_0))},
\end{multline}}
$\!\!$with constant $c=c(\data,\beta_0)$. 
\end{theorem}

Observe that if $A(\cdot,\cdot\cdot)$ satisfies the continuity condition~\eqref{es:cont}, then it certainly fulfills~\eqref{satip}, since by \eqref{A:mean}--\eqref{E_A} and \eqref{es:cont} one easily deduces for every $K\Subset\Omega$
\[
\sup_{x_0\in K}E_{r}(A;x_0)\leq \Lambda\omega(r)\xrightarrow{r\to 0} 0 
\]
uniformly. Therefore, taking also into account  \eqref{emb:weak}, Theorem~\ref{thm:hold} is a stronger version  of Theorem~\ref{thm:model}.

We also stress that, for general coefficients $A(x,u)$, the  radii $r_{x_0},\vrho_{x_0}$ may depend also on the solution $u$. More precisely, they must fulfill the excess decay conditions \eqref{non:cancellare} and \eqref{local.Egamma.bdd}. However, in the case $A(x,u)=A(x)$ these conditions are no longer necessary, hence $r_{x_0},\vrho_{x_0}$ will become independent of $u$, so that the quantitative estimate \eqref{local.Holder1} becomes particularly useful for compactness arguments. We also stress that estimate~\eqref{local.Holder1} is  a direct consequence of a more general bound involving a maximal operator applied to $f$. We refer to Theorem~\ref{partial.Holder} and to equation~\eqref{vera:localHolder} below for the precise statement.

\vs

For what concerns gradient regularity of solutions, we need to make the following stronger assumptions on $A(\cdot,\cdot\cdot)$. Specifically, we shall assume that there exists another modulus of continuity $\omegax(\,\cdot\,)$ such that
\begin{equation}\label{mod:cont.x}
\big|A(x,u) - A(y,u)\big| \leq \Lambda\,\omegax(|x-y|)\qquad\text{for all $x,y\in \Omega$ and $u\in \R^N$}\,.
\end{equation}
In order to obtain Lipschitz continuity of solutions, a standard condition to impose on the moduli of continuity $\omegax, \omegau$ is that they satisfies the so-called the \textit{Dini property}
\begin{equation}\label{Dini}
\int_0^1\omega(\vrho)\,\frac{d\vrho}{\vrho}<\infty\,.
\end{equation}

\begin{theorem}[Partial Lipschitz continuity]\label{thm:lip}
 Let $u\in W^{1,2}_{\rm loc}(\Omega;\R^N)\cap L^1_{2s}$ be a local  weak solution to \eqref{eq1}, for $s\in(0,1)$. Assume that \eqref{A:gr}--\eqref{phi:gr} and \eqref{mod:cont.x} are in force, that $\omegau(\,\cdot\,),\omegax(\,\cdot\,)$ fulfill the Dini continuity property \eqref{Dini}, and that $f\in L^\chi_{\rm loc}(\Omega;\R^N)$ is such that for every $K\Subset\Omega$ there exists $r_f=r_f(K)>0$ such that 
    \begin{equation}\label{fin:riesz}
        \sup_{x_0\in K} {\bf I}_{1,\chi}^f(x_0,r_f)<\infty\,.
    \end{equation}
Then there exists an open set $\Omega_u\subseteq \Omega$, with $\Omega_u=\Omega$ in dimension $n=2$, satisfying $\mathrm{dim}_\H (\Omega\setminus \Omega_u)\leq n-2-\zeta$, $\zeta=\zeta(\data,\max\{\nu,1\})\in(0,1)$ in higher dimension $n\geq3$, such that
\begin{equation*}
Du\in L^\infty_{\rm loc}(\Omega_u;\R^N)\,.
\end{equation*}
Moreover, for every $x_0\in \Omega_u$, there exist three radii $\bar r_{x_0}\, r_{x_0}$ and $\vrho_{x_0}$ such that $\bar r_{x_0}\leq r_{x_0}/4$, $r_{x_0}\leq \vrho_{x_0}/8$, $B_{r_{x_0}}(x_0)\Subset \Omega_u$, $B_{2\vrho_{x_0}}(x_0)\Subset \Omega$,  and
{\rm \begin{multline}\label{quant:lip}
    {\|Du\|}_{L^\infty(B_r(x_0))}\leq c\bigg( \mint_{B_{2r}(x_0)}|Du|^2\dx\bigg)^{1/2}+c\bigg( \mint_{B_{ \vrho_{x_0}}(x_0)}|Du|^2\dx\bigg)^{1/2}
+c\,{\big\|{\bf I}^f_{1,\chi}(\,\cdot\,,r_{x_0})\big\|}_{L^\infty (B_{r_{x_0}}(x_0))}\\
+c\,r\Big[r_{x_0}^{-2s} \tail\big(u-(u)_{B_{r_{x_0}}(x_0)} ,B_{r_{x_0}}(x_0)\big)+\vrho_{x_0}^{-2s}\tail\big(u-(u)_{B_{\vrho_{x_0}}(x_0)} ,B_{\vrho_{x_0}}(x_0)\big)\Big]
\end{multline}}
$\!\!$for all $r\in (0,\bar r_{x_0}]$, with constant $c=c(\data,\max\{\nu,1\})$.
\end{theorem}

Observe that, by the absolute continuity of the Lebesgue integral, condition \eqref{fin:riesz} implies 
\begin{equation*}
    \lim_{\s\to 0} {\bf I}^f_{1,\chi}(x,\s)=0\quad\text{for every $x\in \Omega$}\,.
\end{equation*}
By upgrading this hypothesis to uniform convergence, we attain gradient continuity of solutions. This is the content of the following 
\begin{theorem}[Partial gradient continuity]\label{thm:gradcont}
Under the assumptions of Theorem \ref{thm:lip}, suppose additionally that
\begin{equation}\label{f:riesz1}
\lim_{\sigma\to 0} {\bf I}^f_{1,\chi}(\,\cdot\,,\sigma)=0\quad\text{locally uniformly in $\Omega$}\,;
\end{equation}
then
\begin{equation*}
Du\in C^0(\Omega_u;\R^N)\,.
\end{equation*}
\end{theorem}
We point out that \eqref{f:riesz1} is satisfied, for instance, under the Lorentz-space assumption 
\begin{equation*}
f\in L^{n,1}_{\rm loc}(\Omega;\R^N)\,,
\end{equation*}
see \cite[Lemma 3]{DM100}. $L^{n,1}_{\rm loc}(\Omega;\R^N)$ consists of the measurable functions $f:\Omega\to \R^N$ such that
\[
{\|f\|}_{L^{n,1}(\Omega;\R^N)}:=
\int_0^\infty \l\,|\{x\in \Omega:\,|f(x)|>\l\} |^{1/n}<\infty
\]
and the local variant is defined in the usual way. We recall that, by standard interpolation theory \cite[Proposition 8.1.8]{pick}, we have the continuous embeddings
\begin{equation*}
    L^{q}(\Omega) \hookrightarrow L^{n,1}(\Omega) \hookrightarrow L^n(\Omega)\quad \text{for any $q>n$.}
\end{equation*}
The next theorem concerns H\"older continuity of the gradient up to the threshold $2(1-s)$ in the case $s>1/2$ for general nonlocal terms merely satisfying \eqref{phi:gr} and \eqref{sym:phi}.

\begin{theorem}[Partial $C^{1,\beta}$-regularity I]\label{thm:c1b}
Let $u\in W^{1,2}_{\rm loc}(\Omega;\R^N)\cap L^1_{2s}$ be a local weak solution to \eqref{eq1}, for $s\in(0,1)$. Let $\beta\in (0,1)$, and assume that $A(\cdot,\cdot\cdot)$ satisfies \eqref{A:gr},\eqref{mod:cont.u} and \eqref{mod:cont.x} with
\begin{equation}\label{omega:rb}
\omegax(r)=\omegau(r)=r^\beta\,,
\end{equation} 
 that $\Phi$ satisfies \eqref{phi:gr} and  \eqref{sym:phi}, and  
 \begin{equation}\label{f:1beta}
      f\in \mathcal{M}_{\rm loc}^{n/(1-\beta)}(\Omega;\R^N)\,.
  \end{equation}
Then there exists an open set $\Omega_u\subseteq \Omega$, with $\Omega_u= \Omega$ when $n=2$, $\mathrm{dim}_\H (\Omega\setminus \Omega_u)\leq n-2-\zeta$, $\zeta=\zeta(\data,\max\{\nu,1\})\in(0,1)$ if $n\geq3$, such that 
\[
u\in C^{1,\min\{\beta,2(1-s)\}}_{\rm loc}(\Omega_u;\R^N)\,.
\]
Moreover, for every $x_0\in \Omega_u$, there exist three radii $\bar r_{x_0}\, r_{x_0}, \vrho_{x_0}$ such that $\bar r_{x_0}\leq r_{x_0}/4$, $r_{x_0}\leq \vrho_{x_0}/8$, $B_{r_{x_0}}(x_0)\Subset \Omega_u$, $B_{2\vrho_{x_0}}(x_0)\Subset \Omega$, and the quantitative estimates
{\rm \begin{multline}\label{new:quantlip}
{\|Du\|}_{L^\infty(B_{r}(x_0))}\leq c\,\bigg( \mint_{B_{2r}(x_0)}|Du|^2\dx\bigg)^{1/2}+c\,\bigg( \mint_{B_{ \vrho_{x_0}}(x_0)}|Du|^2\dx\bigg)^{1/2}+c\,{\|f\|}_{\M^{n/(1-\beta)}(B_{\vrho_{x_0}}(x_0))}\\
+c\,\bar r_{x_0}\,r_{x_0}^{-2s} \tail\big(u-(u)_{B_{r_{x_0}}(x_0)};B_{r_{x_0}}(x_0)\big)+c\,\bar r_{x_0}\,\vrho_{x_0}^{-2s}\tail\big(u-(u)_{B_{\vrho_{x_0}}(x_0)};B_{\vrho_{x_0}}(x_0)\big)\,,
\end{multline}}
for all $r\in (0,\bar r_{x_0}]$, and
{\rm \begin{multline}\label{C1:depend}
{[Du]}_{C^{0,\min\{\beta,2(1-s)\}}(B_{\bar r_{x_0}/8}(x_0))}\leq c\bigg[\bar r_{x_0}^{-\min\{\beta,2(1-s)\}}\bigg(\mint_{B_{\bar r_{x_0}}(x_0)}{\big|Du-{(Du)}_{B_{\bar r_{x_0}}( x_0)}\big|}^2\dx \bigg)^{1/2}\\
+ {\|Du\|}_{L^\infty(B_{\bar r_{x_0}}(x_0))}^{1+\beta} +\big(1+\bar r_{x_0}^{1-2s}\big){\|Du\|}_{L^\infty(B_{\bar r_{x_0}}(x_0))}\\
+{\|f\|}_{\M^{n/(1-\beta)}(B_{\bar r_{x_0}}(x_0))}+\bar r_{x_0}^{-2s}\tail\big(u-(u)_{B_{\bar r_{x_0}}(x_0)};B_{\bar r_{x_0}}(x_0)\big)  +1 \bigg]
 \end{multline}}
$\!\!$hold true, with $c=c(\data,\max\{\nu,1\},\beta)$.
\end{theorem}
Under additional assumptions on $\Phi$, we are instead able to obtain $C^{1,\beta}$ regularity, regardless of the value of $s\in (1/2,1)$. For the reader's convenience, we write them once again. We assume that there exists a matrix valued function $a:\R^n\to \R^{N^2}$, $a=\{a_{\a,\b}\}^{\a,\b=1,\dots,N}$ such that
\begin{equation}\label{ass:phi1}
\begin{cases}
a(z)=a(-z)\\[3mm]
|a(z)|\leq \nu
\end{cases}
\qquad z\in \R^n\,,
\end{equation}
and 
\begin{equation}\label{ass:phi2}
    \Phi(x,y,v,w,t)=a(x-y)\,t\,.
\end{equation}

\begin{theorem}[Partial $C^{1,\beta}$ regularity II]\label{thm:c1bfull}
 Let $u\in W^{1,2}_{\rm loc}(\Omega;\R^N)\cap L^1_{2s}$ be a local weak solution to \eqref{eq1}, for $s\in(0,1)$. Let $\beta\in (0,1)$, and assume that $A(\cdot,\cdot\cdot)$ satisfies \eqref{A:gr},\eqref{mod:cont.u}, \eqref{mod:cont.x} and \eqref{omega:rb}. Suppose also that $\Phi$ is of the form \eqref{ass:phi1}--\eqref{ass:phi2}, and that
\[
      f\in \mathcal{M}^{n/(1-\beta) }_{\rm loc}(\Omega;\R^N)\,.
 \]
 Then there exists an open set $\Omega_u\subseteq \Omega$, with $\Omega_u= \Omega$ if $n=2$ and $\mathrm{dim}_\H (\Omega\setminus \Omega_u)<n-2-\zeta$ if $n\geq 3$, $\zeta=\zeta(\data,\max\{\nu,1\})$, such that
\begin{equation}
u\in C^{1,\beta}_{\rm loc}(\Omega_u;\R^N)
\end{equation}
and for every $x_0\in \Omega_u$, there exist radii  $R_{x_0}=\bar r_{x_0}/32\leq \bar r_{x_0}\leq r_{x_0}\leq \vrho_{x_0}$ such that $B_{2\vrho_{x_0}}(x_0)\Subset \Omega$, estimate \eqref{new:quantlip} holds, and we have
\begin{equation}\label{c1bfull}
    [Du]_{C^{0,\beta}(B_{R_{x_0}}(x_0))}\leq C_1\,,
\end{equation}
where $C_1$ depends on 
$\data,\allowbreak \max\{\nu,1\},\allowbreak \beta,\allowbreak 
 \bar r_{x_0},\allowbreak r_{x_0},\allowbreak 
\vrho_{x_0}$, and on any upper bound for {\rm ${\|f\|}_{\M^{n/(1-\beta)}(B_{\vrho_{x_0}}(x_0))}, \allowbreak \tail\big(u-(u)_{B_{4R_{x_0}}(x_0)}; B_{4R_{x_0}}(x_0)\big),\allowbreak
{\|Du\|}_{L^\infty(B_{\bar r_{x_0}}(x_0))}$.}
\end{theorem}

We conclude this section by emphasizing that, in the $u$-independent setting $A(x,u)=A(x)$, we obtain full-regularity and the quantitative estimates in Theorems~\ref{thm:hold}--\ref{thm:c1bfull}
no longer exhibit the implicit dependence on $u$ through the radii of the balls on which the estimates are stated.  Instead, the solution $u$ only enters the estimates  explicitly through its $L^1$-norm and its tail, as specified in the statement below. This feature provides a powerful tool in compactness arguments for the family of operators \eqref{eq1}, since it yields uniform local regularity estimates that are well suited for passing to the limit along sequences of solutions with controlled $L^1$ and tail bounds. We decided to regroup all the results in the following statement:
\begin{theorem}[Full regularity in the $u$-independent case]\label{thm:full-u-independent}
Assume that $A(x,u)\equiv A(x)$ does not depend on $u$. Then each of the preceding partial regularity results, namely Theorems \ref{part.BMO} to \ref{thm:c1bfull}, holds under the same assumptions with the same conclusions, but with $\Omega_u=\Omega$, and all the corresponding local estimates are also valid on arbitrary compact subsets of $\Omega$. The same full-regularity conclusion also applies to the forthcoming Theorem \ref{partial.Holder}.

Moreover, the radii involved in the local estimates of the corresponding theorems can be chosen independently of $u$. More precisely, for every $x_0\in\Omega$, the initial radii $\varrho_{x_0}$ and $r_{x_0}$  can be chosen depending only on $d={\rm dist}(x_0,\partial\Omega), \data, \max\{\nu,1\}$ and
\begin{itemize}
\item $r_A(\,\cdot\,)$, $\beta_0$, and any upper bound on $\allowbreak{\|f\|}_{\mathcal{M}^{n/(2-\beta_0)}(B_{d/2}(x_0))}$ for Theorem \ref{thm:hold};
\item $r_A(\,\cdot\,)$, $\beta_0$, and any upper bound on $\displaystyle{\sup_{\substack{x\in B_{d/2}(x_0)\\ 0<t\leq d/4}} t^{2-\beta_0} \bigg(\mint_{B_t(x)} |f|^\chi\,dx\bigg)^{1/\chi}}$ for Theorem  \ref{partial.Holder} ;
\item $\omegax(\,\cdot\,)$ and any upper bound on $\|{\bf I}^f_{1,\chi}(\,\cdot\,,d/4)\|_{L^\infty(B_{d/2}(x_0))}$ for Theorem \ref{thm:lip};
\item $\beta$ and any upper bound on $\|f\|_{\mathcal M^{n/(1-\beta)}(B_{3d/4}(x_0))}$ for Theorems \ref{thm:c1b} and \ref{thm:c1bfull}.
\end{itemize}
The radii $\bar r_{x_0}$ (for Theorems \ref{thm:lip}, \ref{thm:c1b} and \ref{thm:c1bfull}) and $R_{x_0}$ (for Theorem \ref{thm:c1bfull}) can also be chosen independently of $u$. Since they are chosen not larger than $\varrho_{x_0}$, they inherit the corresponding dependences of $\varrho_{x_0}$ listed above; in Theorem \ref{thm:c1bfull} one may take $R_{x_0}:=\bar r_{x_0}/32$.
\end{theorem}

\section{Notation, functions spaces and preliminary results}\label{sec:prelim}
In the following section we collect a few basic definitions, and  we recall a few standard inequalities and properties which we shall use throughout the paper. 

\subsection{Notation}\label{nota}
We denote by $c$ a generic constant larger or equal than one, possibly varying from line to line, whose relevant dependencies will be emphasized using parentheses: for instance, $c = c(n,s)$ will mean that $c$ depends on $n$ and $s$. In order to simplify the notation for various dependencies, we shall denote by $\data$ the set of  given by \eqref{data}.  On the other hand,  we shall often keep explicit the dependence on the constant $\nu$ due to scaling reasons (see Proposition \ref{prop:highint} for instance). Special occurrences of constants will be denoted by $c_1$, $c_*$, $\bar c$ or the like. We also write $\omega_n=|B_1|$ the Lebesgue measure of the unit ball on $\R^n$; $\N=\{1,2,3,\dots\}$, while $\N_0=\N\cup\{0\}$.

As we deal with vector-valued functions and want to avoid componentwise notation as much as possible, we distinguish the scalar product in $\R^N$ from that in $\R^{n\times N}\simeq\R^{nN}$. Namely, for $\xi,\mu\in\R^N$ and $p,q\in\R^{nN}$, we set
\[
\xi\cdot\mu=\sum_{\alpha=1}^N \xi^\alpha\mu^\alpha\,,
\qquad
\langle p,q\rangle=\sum_{i=1}^n\sum_{\alpha=1}^N p_i^\alpha q_i^\alpha\,.
\]
Vectors in $\R^n$ are regarded as column vectors, while vectors in $\R^N$ are regarded as row vectors. With this convention, gradients of scalar functions are matrices with $n$ rows and one column, whereas gradients of vector-valued functions are matrices with $n$ rows and $N$ columns. More precisely, if $u=(u^1,\ldots,u^N)$, then the $\alpha$-th column of $Du$ is $Du^\alpha$. Integrals and averages of vectors and matrices are taken componentwise.

For tensors $A\in \R^{nN\times nN}$ and matrices $\xi\in\R^{nN}$, typically gradients of vector-valued functions, we use the shorthand $A\xi$ for the matrix in $\R^{nN}$ whose components are
\[
(A\xi)_i^\alpha
=
\sum_{j=1}^n\sum_{\beta=1}^N
A_{ij}^{\alpha\beta}\xi_j^\beta\,,
\qquad
i=1,\ldots,n,\quad \alpha=1,\ldots,N\,.
\]
With $E \subseteq \R^{n}$ being a measurable set with positive, finite measure $0<|E|<+\infty$, and with $F\in L^1(E;\R^\ell)$, $\ell\geq 1$, we shall denote by 
\[
   (F)_{E} \equiv \mint_{E}  F(x) \dx  := \frac{1}{|E|}\int_{E} F(x) \dx
\]
its integral average. 
For $p\in (1,\infty)$, we write $p'=p/(p-1)$ for its H\"older conjugate. 

We denote by $W^{1,2}(\Omega;\mathbb{R}^N)$ the Sobolev space of vector-valued functions $W^{1,2}(\Omega;\mathbb{R}^N)
:= \big\{ u \in L^2(\Omega;\mathbb{R}^N) : Du \in L^2(\Omega;\mathbb{R}^{nN}) \big\}$, where $Du$ denotes the distributional gradient of $u$. Accordingly, $W^{1,2}_{\mathrm{loc}}(\Omega;\mathbb{R}^N)$ denotes the class of functions $u$ such that $u \in W^{1,2}(A;\mathbb{R}^N)$ for every open set $A$ compactly contained in $\Omega$: $A \Subset \Omega$, while $W^{1,2}_{c}(\Omega;\mathbb{R}^N)$ denotes the subspace of functions with compact support in $\Omega$. We will also denote by ${[w]}_{W^{s,2}(B_r)}$ the $s$-fractional seminorm
\[
{[w]}_{W^{s,2}(B_r)}^2:=\int_{B_r}\int_{B_r}\frac{|w(x)-w(y)|^2}{|x-y|^{n+2s}}\dx\,dy\,.
\]


\subsection{Function spaces}\label{FS}

In this section we are going to consider measurable maps $F:U\to \R^\ell$, $\ell\in \N\setminus \{0\}$ and $U\subseteq \R^n$ measurable, not necessarily bounded,  $n\in \N,\, n\geq2$, and define several regularity function spaces where $F$ can belong to.

The ($L^2$-)excess of $F\in L^2(U;\R^\ell)$ on balls is classically defined as
\begin{equation}\label{defL2:excess}
    E\big(F;B_\vrho(x_0)\big)\coloneqq \bigg(\mint_{B_{\vrho}(x_0)}{\big|F-(F)_{B_{\vrho}(x_0)}\big|}^2\dx\bigg)^{1/2}
\end{equation}
if $B_\vrho(x_0)\subseteq U$. It is very standard to see that 
\begin{equation}\label{media}
\bigg(\mint_{B_\vrho(x_0)}\big|F-{(F)}_{B_\vrho(x_0)}\big|^2\dx\bigg)^{1/2}\leq \bigg(\mint_{B_\vrho(x_0)}|F-z|^2 \dx\bigg)^{1/2}
\end{equation}
holds true for any $z\in \R^\ell$.  As an immediate consequence of \eqref{media}, for every $B_{\vrho}(x_1)\subseteq B_r(x_2)$ not necessarily concentric balls, we deduce
\begin{equation}\label{media:nonconc}
E\big(F;B_\vrho(x_1)\big)\leq \Big(\frac{r}{\vrho}\Big)^{n/2}\,E\big(F;B_r(x_2)\big)\,.
\end{equation}
As we are only interested in local results, we localize the classic definition of the $BMO$ space in the following way: for $F\in L^2(U;\R^\ell)$ as above and a threshold $r_0\in(0,+\infty]$, we set
\begin{equation*}
E_{r_0}\big(F;x_0\big):=\sup_{\substack{\vrho\in(0,r_0]:\\B_{\vrho}(x_{0})\subseteq U}}    E\big(F;B_\vrho(x_0)\big)\,;
\end{equation*}
notice that for $0<r_0<r_1\leq +\infty$, $E_{r_0}(F;x_{0})<+\infty$ if and only if $E_{r_1}(F;x_{0})<+\infty$. Motivated by this fact, we say that $F(\,\cdot\,)$ is:
\begin{itemize}
\item $BMO$ in $U$ ($F\in BMO(U;\R^\ell)$) if these exist $c\geq0$ and $r_0>0$ such that $E_{r_0}(F;x_{0})\leq c$ for a.e. $x_0\in U$; we set in this case
\[
{[F]}_{BMO(U)} =\sup_{x_0\in U} E_{r_0}\big(F;x_0\big)=\sup_{\substack{\vrho\in(0,r_0]:\\B_\vrho(x_0)\subseteq U}} E\big(F;B_\vrho(x_0)\big)
\]
where $\sup$ is the essential supremum;
\item locally $BMO$ in $U$ ($F\in BMO_{\rm loc}(U;\R^\ell)$) if it is $BMO$ in $K$ for every compact subset $K\Subset U$;
\item \textnormal{VMO} in $x_0\in U$ if
\[
\lim_{r\to 0^+}E_r\big(F;x_0\big)=0\,;
\]
\item locally $VMO$ in $U$ ($F\in VMO_{\rm loc}(U;\R^\ell)$) if $E_r(\,\cdot\,)$ converges to zero locally uniformly in $U$, that is, for every $K\Subset U$, $E_r(\,\cdot\,)$ converges to zero uniformly in $K$.
\end{itemize}

\begin{remark}
{\rm Our definition of $BMO$ uses the $L^2$-excess for simplicity, but it is equivalent to the more classical one employing the $L^1$-excess, see \cite[Section 1.3]{Stein1}. Note that, as a consequence of our definition, $F\in BMO_{\rm loc}(U;\R^\ell)$ if and only if for every $K\Subset U$ there exists a threshold $r_K\in (0,{\rm dist}(K,\partial U))$ such that $E_{r_K}(F;\,\cdot\,)$ is bounded in $K$.
Moreover notice that $F\in VMO_{\rm loc}(U;\R^\ell)$ if and only if for every $K\Subset U$ and  $\delta>0$, there exists a non-negative $\bar r=\bar r(\delta,K)<{\rm dist}(K,\partial U)$ such that $\sup_K E_{\bar r}(\,\cdot\,)\leq \delta$. 
}
\end{remark}

Let $F\in C^0(U;R^\ell)$; we say that  $\omega:[0,\infty)\to [0,1]$ is a modulus of continuity for $F$ if $\omega(\,\cdot\,)$ is a concave, non-decreasing function such that $\omega(0)=\lim_{r\to 0^+}\omega(r)=0$,  and
\begin{equation*}
    |F(x)-F(y)|\leq \Lambda\,\omega\big(|x-y| \big),\qquad x,y\in U,
\end{equation*}
for some $\Lambda>0$. We recall that $F$ if Dini-continuous if $\omega(\,\cdot\,)$ satisfies \eqref{Dini}. Notice that if $\delta\in (0,1/2]$ and $R_0>0$, using the monotonicity of $\omega(\,\cdot\,)$
\begin{align}\label{Dini.dyadic}
\sum_{i=0}^\infty \omega(\delta^i R_0)&\leq \frac1{\log(1/\delta)}\sum_{i=1}^\infty\omega(\delta^i R_0)\int^{\delta^{i-1} R_0}_{\delta^{i} R_0}\frac{d\vrho}\vrho+\frac1{\log2}\,\omega(R_0)\int^{2R_0}_{R_0}\frac{d\vrho}\vrho\notag\\
&\leq \frac1{\log(1/\delta)}\bigg[\sum_{i=1}^\infty\int^{\delta^{i-1} R_0}_{\delta^{i} R_0}\omega(\vrho)\,\frac{d\vrho}\vrho+\int_{R_0}^{2R_0}\omega(\vrho)\,\frac{d\vrho}\vrho\bigg]=\frac1{\log(1/\delta)}\int_0^{2R_0}\omega(\vrho)\,\frac{d\vrho}\vrho\,.
\end{align}
Suppose that $F$ belongs to the Marcinkiewicz, or weak-Lebesgue, space $\mathcal M^m(U;\R^\ell)$, i.e. \eqref{def:marcke} holds.
By prescribing the decay of the measure of the super-level set, cake-layers formula implies a precise decay of the averages of $F$. Let indeed $F$ be in $\mathcal M^m(U;\R^\ell)$ for some $m\geq1$; then for every $\chi\in[1,m)$ and every ball $B_R(x_0)\subseteq U$, there holds
\begin{equation}\label{density.Marc}
\bigg( \mint_{B_{R}(x_0)} {|F|}^\chi \dx \bigg)^{1/\chi} \leq \Big[\frac\chi{m-\chi}\Big]^{1/m} {|B_R|}^{-1/m} {\|F\|}_{\mathcal M^m(B_{R}(x_0);\R^\ell)} \,.
\end{equation}
Notice that assumptions \eqref{ass:fhold} and \eqref{f:1beta} imply \eqref{ass:f} for a suitable exponent \(\chi>2_*\): in the case of \eqref{ass:fhold} we can choose $\displaystyle{2_*<\chi<\frac{n}{2-\beta_0}<n}$ and in the case of \eqref{f:1beta} we can choose $\displaystyle{2_*<\chi<\frac{n}{1-\beta}<n}$. In dimension \(n=2\), whenever a Marcinkiewicz assumption with exponent depending on \(\beta_0\) is used, the auxiliary exponent \(2_*\in(1,2)\) is chosen after \(\beta_0\), sufficiently close to \(1\), so that the interval $(2_*,\,2/(2-\beta_0))$ is non-empty.
%
We conclude this paragraph with two elementary inequalities.

\begin{lemma}[Fractional embedding]
    Let $s\in (0,1)$ and $B_R(x_0)\subseteq \R^n$. Assume that $w\in W^{1,2}(B_R(x_0)\R^N)$; then $w\in W^{s,2}(B_R(x_0);\R^N)$ with
    \begin{equation}\label{frac:emb}
\frac1{|B_R(x_0)|}{[w]}_{W^{s,2}(B_R(x_0))}^2=\int_{B_R(x_0)}\mint_{B_R(x_0)}\frac{|w(x)-w(y)|^2}{|x-y|^{n+2s}}\dx\,dy\leq \frac{c(n,N)}{1-s}R^{2(1-s)}\mint_{B_R(x_0)}|Dw|^2\dx\,.
    \end{equation}
\end{lemma}

\begin{proof}
    By a scaling argument, we may assume that $x_0=0,R=1$. We first extend $w$ from $B_1$ to $\R^n$ via the spherical inversion map. Specifically, let
    \begin{equation*}
        \tilde{w}(x)=\begin{cases}
            w(x)\quad &x\in B_1
            \\[2mm]
\displaystyle{w\Big( \frac{x}{|x|^2}\Big)}\quad &x\in \R^n\setminus B_1,
        \end{cases}
    \end{equation*}
    and with some work one can show that
    \begin{equation}\label{D:invsph}
        {\|D\tilde{w}\|}_{L^2(B_4)}\leq c(n,N){\|Dw\|}_{L^2(B_1)}\,,
    \end{equation}
    see for instance \cite[ Proposition 3.9.1]{BBook}. Via a change of variables, the fundamental theorem of calculus, Fubini's theorem and \eqref{D:invsph}, we obtain
    \begin{equation*}
        \begin{split}
            \int_{B_1}\int_{B_1}\frac{|\tilde{w}(x)-\tilde{w}(y)|^2}{|x-y|^{n+2s}}\dx\dy & \leq \int_{B_1}\int_{B_2}\frac{|\tilde{w}(x)-\tilde w(z+x)|^2}{|z|^{n+2s}}\,dz\dx
            \\
            &\leq \int_{B_1}\int_{B_2}\int_0^1 |D\tilde w(x+tz)|^2\,dt \,\frac{dz}{|z|^{n+2(s-1)}}\dx
            \\
            &=\int_{B_2}\int_0^1\int_{B_1}|D\tilde w(x+tz)|^2\dx\,dt\,\frac{dz}{|z|^{n+2(s-1)}}
            \\
            &\leq \int_{B_1}\frac{dz}{|z|^{n+2(s-1)}}\int_{B_3} |D\tilde w|^2\dx\leq \frac{c(n,N)}{1-s}\int_{B_1} |Dw|^2\dx,
        \end{split}
    \end{equation*}
    which proves \eqref{frac:emb}.
\end{proof}

The following interpolation inequality can be found in \cite[Theorem 6.23]{leo:frac}.
\begin{lemma}
 Let $s\in (0,1)$, $R>0$ and $B_R(x_0)\subseteq \R^n$. If $u\in W^{s,2}(B_R(x_0))$ and $\varphi \in W^{1,\infty}(B_R(x_0))$, then $u\varphi \in W^{s,2}(B_R(x_0))$ and (if $\varphi$ is not identically zero)
\begin{equation}\label{interpol}
{[u\varphi]}_{W^{s,2}(B_R(x_0))}\leq c(n,s){\|\varphi\|}_{L^\infty(B_R(x_0))}\Big[{\|\varphi\|}_{L^\infty(B_R(x_0))}^{-s}{\|D\varphi\|}_{L^\infty(B_R(x_0))}^{s}{\|u\|}_{L^p(B_R(x_0))}+{ [u]}_{W^{s,2}(B_R(x_0))}\Big]\,.
\end{equation}
\end{lemma}

\subsection{Nonlinear potentials of Riesz-type}

We shall repeatedly use the following elementary comparison between the $\chi$-averaged truncated Riesz-type potential introduced in \eqref{dfriesz:pot},
\[
    \mathbf I^f_{\alpha,\chi}(x_0,R)
    =
    \int_0^R
    \vrho^\alpha
    \bigg(\mint_{B_\vrho(x_0)} |f|^\chi\,\dx\bigg)^{1/\chi}
    \,\frac{d\vrho}{\vrho},
\]
and its dyadic counterpart. In particular, the continuous potential controls the sum of the corresponding quantities over a geometric sequence of radii, as well as the related one-scale supremum.
\begin{lemma}
Let $f\in L^\chi(B_{2R}(x_0);\R^N)$ for some $\chi\geq1$ and some ball $B_{2R}(x_0)\subseteq\R^n$; let moreover $\alpha\in (0,n/\chi)$ and $\tau\in (0,1/2]$ be a parameter. Then, setting $R_j=\tau^j R$ for $j\in\N_0$, $B_j=B_{R_j}(x_0)$, and $R_{-1}=2R$, we have
\begin{equation}\label{est.potential.sum}
\sum_{j=\bar \jmath}^\infty R_j^{\alpha} \bigg(\mint_{B_j} |f|^\chi \dx \bigg)^{1/\chi}\leq c\,{\bf I}^f_{\alpha,\chi}(x_0, R_{\bar\jmath-1})\qquad\text{for all $\bar\jmath\in\N_0$}
\end{equation}
and 
\begin{equation}\label{est.potential.onescale}
\sup_{t\in(0,R_{\bar\jmath}]}t^\alpha \bigg(\mint_{B_t(x_0)} |f|^\chi \dx \bigg)^{1/\chi} \leq c\,{\bf I}^f_{\alpha,\chi}(x_0, R_{\bar\jmath-1})\,,
\end{equation}
where both the constants depend only on $n,\chi$ and $\tau$.
\end{lemma}

\begin{proof}
Recalling that $R_j=\tau^jR$, we have
\begin{align*}
        \sum_{j=\bar \jmath}^\infty  R_j^{\alpha} \bigg(\mint_{B_j} |f|^\chi \dx \bigg)^{1/\chi} &\leq \frac{1}{\log(1/\tau)}\sum_{j=\bar \jmath}^\infty  R_j^\alpha \bigg(\mint_{B_j} |f|^\chi \dx \bigg)^{1/\chi}\int_{R_j}^{R_{j-1}}\frac{d\vrho}{\vrho}
        \\
        &\leq \frac{\tau^{-n/\chi}}{\log(1/\tau)}\sum_{j=\bar \jmath}^\infty \int_{R_j}^{R_{j-1}}\vrho^\alpha\bigg( \mint_{B_\vrho(x_0)}|f|^{\chi}dx\bigg)^{1/\chi}\frac{d\vrho}{\vrho}=\frac{\tau^{-n/\chi}}{\log(1/\tau)}\,{\bf I}^f_{\alpha,\chi}(x_0, R_{\bar\jmath-1})
\end{align*}
which proves \eqref{est.potential.sum}. For the proof of \eqref{est.potential.onescale}, we observe that if $0<t\leq R_{\bar\jmath}$, then we may find $j\geq\bar\jmath$ such that $R_{j+1}< t\leq R_j$; we thus  have
\[
t^\alpha \bigg(\mint_{B_t(x_0)} |f|^\chi \dx \bigg)^{1/\chi}\leq \tau^{-n/\chi}R_j^\alpha \bigg(\mint_{B_j(x_0)} |f|^\chi \dx \bigg)^{1/\chi},
\]
and the proof is completed by  applying \eqref{est.potential.sum}.
\end{proof}
\subsection{\texorpdfstring{$\mathsf{A}$-}{A-}harmonic vectorial approximation}
In the due course of the proof, we will exploit the quantitative harmonic type approximation; the version we are going to use is a small variation of the very clean and sharp version that can be found in \cite{CUD}. We recall that, for a ball $B$ and a tensor $\mathsf{A}$ as in the statement below, we say that a function $h\in W^{1,2}(B;\R^N)$ is $\mathsf{A}$-harmonic if it solves 
\begin{equation*}
-{\rm div}\big(\mathsf{A}Dh)=0 \quad\text{in $B$,}\qquad\qquad\text{i.e.,}\qquad \mint_{B} \langle\mathsf{A}Dh, D\vphi\rangle\dx=0\qquad\text{for all $\vphi\in C^\infty_c(B)$}\,. 
\end{equation*}

\begin{lemma}[Harmonic approximation Lemma]\label{lem:harapprox}
Let $n\geq2, N\geq1$ and $\mathsf{A}\in \R^{nN\times nN}$ be a constant tensor satisfying for some $\Lambda\geq1$
\begin{equation}\label{Legendre.A}
\langle \mathsf{A}\xi,\xi\rangle \geq \Lambda^{-1}|\xi|^2\qquad\text{and}\qquad  |\mathsf{A}| \leq \Lambda\qquad\text{for all $\xi\in \R^{nN}$}\,. 
\end{equation}
Let $B\subseteq\R^n$ be a ball and assume that $u\in W^{1,2}(B;\R^N)$ is a function satisfying
\begin{equation}\label{quant:control}
\mint_B {|Du|}^{2+\zeta}\dx\leq c_0\,,
\end{equation}
for some $\zeta>0$ and $c_0\geq1$. If $u$ is also approximately $\mathsf{A}$-harmonic, in the sense that it satisfies
\begin{equation*}
\left|\mint_{B}\langle\mathsf{A}Du,D\vphi\rangle\dx\right|\leq \e M{\|D\vphi\|}_{L^\infty(B)}\quad\text{for all $\vphi\in C^\infty_c(B;\R^N)$}
\end{equation*}
for some $\varepsilon>0$ and $M>0$, there exists a (unique) $\mathsf{A}$-harmonic function $h\in u+ W^{1,2}_0(B;\R^N)$ such that
\begin{equation*}
\mint_{B}{|Du-Dh|}^2\dx\leq c\, \e^{\zeta/(1+\zeta)}\big[1+M^2\big]\,,
\end{equation*}
for a constant $c$ depending only $n,N,\Lambda,\zeta$ and $c_0$.
\end{lemma}

\begin{proof}
The Lemma follows from the main Theorem of \cite{CUD} and the subsequent Remark, choosing $p=2,q=2+\zeta$, using our assumption \eqref{quant:control} and performing basic algebraic manipulations.
\end{proof}

\subsection{Nonlocal tail}
An important quantity is the generalization of the so-called tail, namely for $\gamma>0$, we shall use the $\gamma$-tail defined by
\begin{equation}\label{def:tail}
    \tail_\gamma\big(u;B_R(x_0)\big)\coloneqq R^{\gamma}\int_{\R^n\setminus B_R(x_0)}\frac{|u(x)|}{|x-x_0|^{n+2s}}\,dx\,,
\end{equation}
for $x_0\in \R^n$ and $R>0$. When $\gamma=2s$, this is the classical Tail term introduced in \cite{dicastro,korv}, and for brevity we set
\begin{equation}\label{relaz:tails}
    \tail_{2s}\big(u;B_R(x_0)\big)\equiv \tail\big(u;B_R(x_0)\big)\quad\implies\quad     \tail_\gamma\big(u;B_R(x_0)\big)=R^{\gamma-2s}\tail\big(u;B_R(x_0)\big)\,.
\end{equation}
Let us recall a few properties concerning the tail: clearly, whenever this is defined,
\begin{equation}\label{tail.triangle}
\tail_\gamma\big(u+\lambda v;B_R(x_0)\big)\leq\tail_\gamma\big(u;B_R(x_0)\big)+|\lambda|\tail_\gamma\big(v;B_R(x_0)\big)
\end{equation}
for $\lambda\in\R$, but also
\begin{equation}\label{tail.constant}
\tail\big(1;B_R(x_0)\big)=\frac{n\,\omega_n}{1+2s},\qquad \tail_\gamma\big(1;B_R(x_0)\big)=\frac{n\,\omega_n}{1+2s}\,R^{\gamma-2s}\,.
\end{equation}
Note moreover that the $\gamma$-tail increases as the radius of the ball decreases:
\begin{equation}\label{counter.monotonicity}
0<R_1\leq R_2\qquad\implies\qquad \tail_\gamma\big(u;B_{R_2}(x)\big)\leq \tail_\gamma\big(u;B_{R_1}(x)\big)\,. 
\end{equation}
The proof of the following Lemma is a consequence of \cite[Lemma 3.2]{dmn} and \eqref{relaz:tails}.
\begin{lemma}\label{lemma:tail}
Let $w\in L^{1}_{2s}$. There exists $c=c(n,s)\geq 1$ such that
{\rm
\begin{multline}\label{tail:1}
\tail_\gamma\big(w-(w)_{B_{\vrho}};B_{\vrho}\big) \leq  c\left(\frac{\vrho}{r}\right)^{\gamma}\tail_\gamma\big(w-(w)_{B_{r}};B_{r}\big)+c\,\frac{\vrho^\gamma}{r^{2s}}\mint_{B_r} |w-(w)_{B_{r}}|\dx  \\  
 +c\int_{\vrho}^{r}\frac{\vrho^\gamma}{\mu^{2s}}\mint_{B_\mu} |w-(w)_{B_{\mu}}|\dx\,\frac{d\mu}{\mu}
\end{multline}
}
if $B_{\vrho}\Subset B_{r}\subseteq \mathbb{R}^{n}$ are concentric balls.
Moreover, if $B_{\vrho}(x_{1})\subseteq B_{r}(x_{2})$ are not necessarily concentric balls, then
{\rm
\begin{multline}\label{tail:3}
\tail_\gamma\big(w-(w)_{B_{\vrho}(x_{1})};B_{\vrho}(x_{1})\big)\leq c\,\Big(\frac{\vrho}{r}\Big)^{\gamma}\Big(\frac{r}{r-|x_{1}-x_{2}|}\Big)^{n+2s}\tail_\gamma\big(w-(w)_{B_{r}(x_{2})};B_{r}(x_{2})\big)\\
 +c\,\vrho^{\gamma-2s}\Big(\frac{r}{\vrho}\Big)^{n}\mint_{B_r(x_2)}\big|w-(w)_{B_r(x_2)}\big|\dx \,,
\end{multline}
and 
\begin{equation}\label{tail:4}
    \tail_\gamma\big(w;B_{\vrho}(x_{1})\big)\leq c\Big(\frac{\vrho}{r}\Big)^{\gamma}\Big(\frac{r}{r-|x_{1}-x_{2}|}\Big)^{n+2s}\tail_\gamma\big(w;B_{r}(x_{2})\big)+c\vrho^{\g-2s}\Big(\frac{r}{\vrho}\Big)^{n} \mint_{B_r(x_2)}|w|\dx\,.
\end{equation}
}
\end{lemma}

\subsection{A modified excess}
For our approach we need to introduce a modified excess, incorporating the non-locality in a weighted tail-type term: we set
\begin{align}\label{def:exc}
    \mathcal{E}_{\gamma}\big(w;B_\vrho(x_0)\big)&\coloneqq \bigg(\mint_{B_{\vrho}(x_0)}{\big|w-(w)_{B_{\vrho}(x_0)}\big|}^2\dx\bigg)^{1/2}+\tail_\gamma\big(w-(w)_{B_\vrho(x_0)};B_{\vrho}(x_0)\big)\notag\\
    &= E\big(w;B_{\vrho}(x_0)\big)+\tail_\gamma\big(w-(w)_{B_\vrho(x_0)};B_{\vrho}(x_0)\big)\,,
\end{align}
where $\tail_\gamma$ is defined in \eqref{def:tail}, and $E(\,\cdot\,)$ is the local $L^2$-excess defined in \eqref{defL2:excess}.

Moreover, the map $x\mapsto \mathcal{E}_\g(w;B_\vrho(x))$ is continuous--see, for instance, \cite[Proposition 3.4]{dmn} for the continuity of the $\tail$-term, and therefore of $\tail_\g$ by \eqref{relaz:tails}. 

\vs

We now prove two properties; the first one shows that our modified excess enjoys, qualitatively, the same behavior as the standard excess under enlargement of the integration domain.
\begin{lemma}
Let $w\in L^2_{\rm loc}(\Omega;\R^N)\cap L^{1}_{2s}$ and suppose that $\gamma\geq 2s$. Then, for every couple of radii $0< \vrho\leq r\leq 1$, there holds
\begin{equation}\label{ut:exc}
\mathcal{E}_{\g}\big(w;B_{\vrho}(x_0)\big)\leq c(n, s)\Big(\frac r \vrho\Big)^{n/2}\mathcal{E}_{\g}\big(w;B_{r}(x_0)\big)\,.
\end{equation}
\end{lemma}

\begin{proof}
Let us denote, for a moment, $w_r=w-(w)_{B_r(x_0)}$ for $r\in(0,\vrho]$. First observe that by \eqref{tail:1}, H\"older's inequality and \eqref{media}, we get
\begin{align}\label{decay.tailgamma}
    \tail_\g\big(w_{\vrho};B_{\vrho}(x_0)\big)&\leq  c\Big(\frac\vrho r\Big)^\g\tail_\g\big(w_r;B_r(x_0)\big)+c\frac{\vrho^\gamma}{r^{2s}}\bigg(\mint_{B_r(x_0)}|w_r|^2\dx \bigg)^{1/2}\notag\\
    &\hspace{5cm}+ c\,\vrho^{\g}\int_{\vrho}^r\mu^{-2s}\bigg(\mint_{B_\mu(x_0)}|w_\mu|^2\dx \bigg)^{1/2}\frac{d\mu}{\mu}\notag
    \\
    &\leq c\Big(\frac\vrho r\Big)^\g\tail_\g\big(w_r;B_r(x_0)\big)+c\frac{\vrho^\gamma}{r^{2s}}\bigg(\mint_{B_r(x_0)}|w_r|^2\dx \bigg)^{1/2}\notag\\
    &\pushright{+\, c\vrho^{\g}r^{n/2}\bigg(\mint_{B_r(x_0)}|w_r|^2dx \bigg)^{1/2}\int_\vrho^r\mu^{-2s-n/2}\,\frac{d\mu}{\mu}}\notag\\
    &\leq c\Big(\frac\vrho r\Big)^\g\tail_\g\big(w_r;B_r(x_0)\big)+c\,\frac{\vrho^\gamma}{r^{2s}}\bigg[1+\Big(\frac r\vrho\Big)^{2s+n/2}\bigg]\bigg(\mint_{B_r(x_0)}|w_r|^2\dx \bigg)^{1/2}
\end{align}
for some constant $c=c(n,s)$. Recalling that $\vrho\leq r\leq 1$ and $\gamma\geq 2s$, we get 
\[
\tail_\g\big(w_{\vrho};B_{\vrho}(x_0)\big)\leq c\,\Big(\frac\vrho r\Big)^\g\tail_\g(w_r;B_r(x_0))+c\,\vrho^{\gamma-2s}\Big(\frac r\vrho\Big)^{n/2}E_{\g}\big(w;B_{r}(x_0)\big)\,;
\]
merging this inequality with the obvious estimate for the excess part of $\mathcal{E}_\gamma$ (using also \eqref{media}), we deduce \eqref{ut:exc}. 
\end{proof}

The second property we prove justifies the fact that, qualitatively, the behavior of the modified excess $E_\gamma(\,\cdot\,)$ is dictated by the decay of the excess only.

\begin{lemma}
Let $w\in L^{1}_{2s}$ and suppose that $\gamma\geq 2s$ as in the previous Lemma. Then 
{\rm
\begin{equation*}
\limsup_{\vrho\to0^+} E\big(w;B_\vrho(x_0)\big)=0\qquad\implies\qquad \limsup_{\vrho\to0^+} \tail_\gamma\big(w-(w)_{B_\vrho(x_0)};B_{\vrho}(x_0)\big)=0
\end{equation*}
}
so that, in particular
\begin{equation}\label{exc.vanishing}
    \limsup_{\vrho\to 0^+}\, \mathcal{E}_{\gamma}\big(w;B_\vrho(x_0)\big)>0\qquad\implies\qquad     \limsup_{\vrho\to0^+} E\big(w;B_\vrho(x_0)\big)>0\,.
\end{equation}
\end{lemma}
\begin{proof}
The proof the first implication follows again from \eqref{tail:1}, which implies, for any $\vrho\in(0,r)$ (see the first inequality in \eqref{decay.tailgamma}; here we use the same notation as there)
\begin{multline*}
\tail_\gamma\big(w-(w)_{B_\vrho(x_0)};B_{\vrho}(x_0)\big)\leq c\,\vrho^\gamma\bigg[\frac1{r^\gamma}\tail_\g\big(w_r;B_r(x_0)\big)+\frac{1}{r^{2s}}\bigg(\mint_{B_r(x_0)}{\big|w-(w)_{B_r(x_0)}\big|}^2\dx \bigg)^{1/2}\bigg]\\
+c\,\vrho^{\g}\int_{\vrho}^r\mu^{-2s}\bigg(\mint_{B_\mu(x_0)}{\big|w-(w)_{B_\mu(x_0)}\big|}^2\dx \bigg)^{1/2}\frac{d\mu}{\mu}
\end{multline*}
and it is clear that we only need to check that the latter term is vanishing as $\vrho\to0^+$. However, supposing the $\limsup$ of the integral as $\vrho\to0^+$ is $+\infty$ (otherwise the product, trivially, goes to zero), we can use de l'H\^opital rule and get, as $\gamma-2s\geq0$,
\begin{align*}
\limsup_{\vrho\to 0^+}\vrho^{\g}\int_{\vrho}^r\mu^{-2s}&\bigg(\mint_{B_\mu(x_0)}{\big|w-(w)_{B_\mu(x_0)}\big|}^2\dx \bigg)^{1/2}\frac{d\mu}{\mu}\\
&\leq 
\limsup_{\vrho\to 0^+}\frac1{\gamma\vrho^{-\g-1}}\vrho^{-2s-1}\bigg(\mint_{B_\vrho(x_0)}{\big|w-(w)_{B_\vrho(x_0)}\big|}^2\dx \bigg)^{1/2}\\
&=\limsup_{\vrho\to 0^+}\frac{\vrho^{\gamma-2s}}{\gamma}\bigg(\mint_{B_\vrho(x_0)}{\big|w-(w)_{B_\vrho(x_0)}\big|}^2\dx \bigg)^{1/2}=0\,.
\end{align*}
The second implication is obvious by the definitions of $E(\,\cdot\,)$ and $\mathcal{E}_\gamma(\,\cdot\,)$.
\end{proof}

\section{Caccioppoli's inequality and higher integrability}\label{sec:energy}
The first main result of this section is a Caccioppoli-type inequality. For our future purposes, we prove it for a wider class of operators than \eqref{eq1}. 
Specifically, let $\mathcal{A}:\Omega\times \R^N\times \R^{nN}\to\R^{nN}$, $\mathcal{A}(x,z,p)={\{\A^\a_i(x,z,p)\}}_{i=1,\dots,n}^{\a=1,\dots,N}$ be a Carath\'eodory function satisfying the quadratic growth and ellipticity assumptions
\begin{equation}\label{gr:Agen}
    \big|\A(x,z,p)\big|\leq \Lambda\big(1+|p|\big)\qquad\text{and }\qquad \langle \A(x,z,p),p\rangle\geq \Lambda^{-1}\big(|p|^2-1\big)\,,
\end{equation}
for almost all $x\in \Omega$ and all $z\in \R^N,p\in \R^{nN}$, for a given constant $\Lambda>0$. Moreover, let $f$ be as in \eqref{ass:f} and $G={\{G^\a_i\}}_{i=1,\dots,n}^{\a=1,\dots,N}$ is a vector field in $L^2_{\rm loc}(\Omega;\R^{nN})$.

\vs

We consider local weak solutions to 
\begin{equation}\label{eq2}
    -\mathrm{div}\,\A(x,u,Du)+Q_\Phi^\mathrm{nl} u=f+\mathrm{div}\,G\qquad\text{in $\Omega$}
\end{equation}
that is, functions $u\in W^{1,2}_{\rm loc}(\Omega;\R^N)\cap L^1_{2s}$ that satisfy
\begin{multline*}
\int_\Omega \langle \A(x,u,Du), D\vphi\rangle \dx+\int_{\R^n}\int_{\R^n} \Phi(x,y,u(x),u(y), u(x)-u(y))\cdot \big(\vphi(x)-\vphi(y)\big)\,\frac{dx\dy}{|x-y|^{n+2s}}\\
=\int_\Omega f\cdot \vphi\dx+\int_\Omega \langle G, D\vphi\rangle\dx\,,
\end{multline*}
for all test functions $\vphi\in W^{1,2}_c(\Omega;\R^N)$.  

    \begin{lemma}[Caccioppoli's inequality]\label{lemma:cacc}
        Let $u\in W^{1,2}_{\rm loc}(\Omega;\R^N)\cap L^1_{2s}$ be a weak solution to \eqref{eq2} with the assumptions  \eqref{phi:gr}, \eqref{sym:phi} and \eqref{gr:Agen} in force. Then there exists a radius $r_0=r_0(\data,\nu)\in (0,1)$ such that for any ball $B_r\equiv B_r(x_0)\Subset \Omega$ with $0<r\leq r_0$, and for all $\sigma\in (0,1)$ and $u_0\in \R^N$
\rm{\begin{multline}\label{caccioppoli}
\s^n \mint_{B_{\s r}}|Du|^2\dx\leq c\,\Big(\frac1{(1-\sigma)^2r^2}+\frac\nu{(1-\sigma)^{2s}r^{2s}}\Big)\mint_{B_r}|u-u_0|^2\dx\\
+\frac{c}{(1-\s)^{n+2s}}\,\nu\,r^{-2s}\tail\big(u-u_0; B_r\big)\mint_{B_r}|u-u_0|\dx + c\,r^2\bigg(\mint_{B_r} |f|^\chi\dx\bigg)^{2/\chi}+c\mint_{B_r} \big[|G|^2+1\big]\dx
\end{multline}
}
holds true with a constant $c$ depending on $\data$. 
\end{lemma}

\begin{remark}
   \rm{ We remark that the constant $c$ appearing in \eqref{caccioppoli} is independent of $\nu$. The only implicit dependence on $\nu$ in Lemma \ref{lemma:cacc} is in the radius $r_0$, which has to satisfy
   \begin{equation}\label{cond:r0}
    0< r_0\leq \min\bigg\{1, \left(\frac{1}{c(\data)\,\nu}\right)^{1/[2(1-s)]}\bigg\}\,.   
   \end{equation}
  In particular, for $\nu\ll1$, we may take $r_0(\data, \nu)\equiv 1$.}
\end{remark}

\begin{proof}[Proof of Lemma \ref{lemma:cacc}]
It is sufficient to prove the estimate for $u_0=0$, for otherwise the function $u-u_0$ is a weak solution to $-\mathrm{div}\big( \A_0(x,v,Dv)\big)+Q^\mathrm{nl}_{\Phi_0}v=f+{\rm div}\,G$, where $\A_0(x,z,p)=\A(x, z+u_0,p)$ still satisfies \eqref{gr:Agen} and $\Phi_0(x,y,v,w,t)=\Phi(x,y,v+u_0,w+u_0,t)$ satisfies \eqref{phi:gr} and \eqref{sym:phi}. We fix $\s\in (0,1)$ and two intermediate radii $\sigma\,r\leq \tau_1<\tau_2\leq r$, and we define $\tilde{\tau}_1=(\tau_1+\tau_2)/2$ and $\tilde{\tau}_2=(3\,\tau_2+\tau_1)/4$, so that $\tau_1<\tilde{\tau}_1<\tilde{\tau}_2<\tau_2$. 

\medskip 
Let $\eta$ be a cut-off function $\eta\in C^\infty_c(B_{\tilde \tau_2})$ such that $0\leq \eta\leq 1$, $\eta\equiv 1$ in $B_{\tilde{\tau}_1}$, and ${\|D\eta\|}_{\infty}\leq c/(\tau_2-\tau_1)$. Take $\vphi=(u-(u)_{B_{\tau_2}})\eta^2$ as a test function in \eqref{eq2}. Using \eqref{sym:phi}, we get
\begin{multline}\label{t0}
(O_1)+(O_2)\coloneqq \int_{B_{\tau_2}} f\cdot \vphi\dx+\int_{B_{\tau_2}}\langle G, D\vphi\rangle\dx=  \int_{B_{\tau_2}} \langle \A(x,u,Du), D\vphi\rangle\dx\\
+\int_{B_{\tau_2}} \int_{B_{\tau_2}} \Phi(x,y,u(x),u(y),u(x)-u(y))\cdot (\vphi(x)-\vphi(y))\,\frac{dx\dy}{|x-y|^{n+2s}} \\
+2 \int_{\R^n\setminus B_{\tau_2}}\int_{B_{\tau_2}}\Phi(x,y,u(x),u(y),u(x)-u(y)) \cdot \vphi(x)\,\frac{dx\dy}{|x-y|^{n+2s}}\eqqcolon (I)+(II)+(III)\,.
\end{multline}
We estimate separately each term in \eqref{t0}: via H\"older's inequality and, after recalling \eqref{ass:f}--\eqref{def:2}  that implies $\chi'\leq 2^*$, using Sobolev's inequality we obtain
\begin{align}\label{t1}
(O_1)&\leq \,\bigg(\int_{B_{\tau_2}}|f\eta|^{\chi}\dx\bigg)^{1/\chi}\bigg(\int_{B_{\tau_2}} \big|(u-(u)_{B_{\tau_2}})\eta\big|^{\chi'}dx\bigg)^{1/\chi'} \notag
    \\
    & \leq  c(n,N,\chi) \bigg(\int_{B_{\tau_2}}|f|^{\chi}\dx\bigg)^{1/\chi}\tau_2^{1+n/2-n/\chi}\bigg(2\int_{B_{\tau_2}} |Du|^2\eta^2\dx+4\int_{B_{\tau_2}} \big|u-(u)_{B_{\tau_2}}\big|^2|D\eta|^2\dx\bigg)^{1/2}\notag
    \\
    &\leq \delta\int_{B_{\tau_2}} |Du|^2\eta^2\dx+c(n,N,\chi,\delta)\,r^{n+2(1-\frac n\chi)}\bigg(\int_{B_r}|f|^{\chi}\dx \bigg)^{2/\chi}+\frac{c}{(\tau_2-\tau_1)^2}\int_{B_{\tau_2}}|u|^2\dx\,,
\end{align}
where in the last inequality we used Young's inequality for $\delta>0$ to be chosen later, and triangle inequality too. Then, by Young's inequality and the triangle inequality, we similarly compute
\begin{align*}
    (O_2) &=\int_{B_{\tau_2}}\langle G, Du\rangle\eta^2 \dx+2\int_{B_{\tau_2}} \langle G, D\eta\rangle \big(u-(u)_{B_{\tau_2}}\big)\eta\dx\notag\\
     &\leq \delta\int_{B_{\tau_2}}|Du|^2\eta^2\dx+c(n,N,\delta)\int_{B_r}|G|^2\dx+\frac{c(n,N,\delta)}{(\tau_2-\tau_1)^2}\int_{B_{\tau_2}}|u|^2\dx\,,
\end{align*}
again for $\delta>0$ to be chosen. Next, by \eqref{gr:Agen} we deduce
\begin{align*}
    (I) & \geq \Lambda^{-1}\int_{B_{\tau_2}} \big(|Du|^2-1\big)\eta^2\dx-2\,\Lambda\int_{B_{\tau_2}}\big(1+|Du|\big)|D\eta||u-(u)_{B_{\tau_2}}|\,\eta\dx\notag\\
    & \geq \Lambda^{-1}\int_{B_{\tau_2}} |Du|^2\eta^2\dx-\Lambda^{-1}-\frac{\Lambda^{-1}}2\int_{B_{\tau_2}}|Du|^2\eta^2\dx-c(\Lambda)\int_{B_{\tau_2}}{\big|u-(u)_{B_{\tau_2}}\big|}^2{|D\eta|}^2\dx\notag\\
    &\pushright{-\,2\Lambda\int_{B_{\tau_2}}|D\eta|\big|u-(u)_{B_{\tau_2}}\big|\eta\dx}\notag\\
&\geq \frac{\Lambda^{-1}}{2}\int_{B_{\tau_2}}|Du|^2\eta^2\dx-\frac{c}{(\tau_2-\tau_1)^2}\int_{B_{\tau_2}} |u|^2\dx-c\,,
\end{align*}
where we have used Young's inequality several times, the assumptions on $\eta$ and the triangle inequality. Using \eqref{phi:gr}, H\"older's inequality, the product rule \eqref{interpol} and \eqref{frac:emb}, we infer
\begin{align*}
    -(II) & \leq \nu\int_{B_{\tau_2}}\int_{B_{\tau_2}}\frac{|u(x)-u(y)|\,|\vphi(x)-\vphi(y)|}{|x-y|^{n+2s}}\dx\dy\\
    &\leq \nu\,{[u]}_{s,2;B_{\tau_2}}\,{[\vphi]}_{s,2;B_{\tau_2}}\notag
    \\
    &\leq \frac{\nu}{2}{[u]}^2_{s,2;B_{\tau_2}}+c\,\nu\bigg[{\|\eta^2\|}_{L^\infty}^{2(1-s)}\,{\|D\eta^2\|}_{L^\infty}^{2s}\int_{B_{\tau_2}}\big|u-(u)_{B_{\tau_2}}\big|^2\dx+{\|\eta^2\|}_{L^\infty}^2\,{[u]}^2_{s,2;B_{\tau_2}}\bigg]\notag\\
    &\leq c\,\nu{[u]}^2_{s,2;B_{\tau_2}}+\frac{c\,\nu}{(\tau_2-\tau_1)^{2s}}\int_{B_{\tau_2}}\big|u-(u)_{B_{\tau_2}}\big|^2\dx\notag\\
    &\leq c\,\nu\,r^{2(1-s)}\int_{B_{\tau_2}}|Du|^2\dx+\frac{c\,\nu}{(\tau_2-\tau_1)^{2s}}\int_{B_{\tau_2}}|u|^2\dx\,,
\end{align*}
with $c=c(n,s)$. Finally, using \eqref{phi:gr}, by using 
\[
\tail\big(u-(u)_{B_{\tau_2}};B_{\tau_2}\big)\leq \tail\big(u;B_{\tau_2}\big)+c(n,s)|(u)_{B_{\tau_2}}|
\]
and the fact that, for $x\in B_{\tilde \tau_2}=B_{\tilde\tau_2}(x_0)$ and $y\in\R^n\setminus B_{\tau_2}(x_0)$, one has
\begin{equation}\label{est:dist}
    |y-x|\geq \frac{\tau_2-|x-x_0|}{\tau_2}|y-x_0|\geq \frac{\tau_2-\tilde{\tau}_2}{\tau_2}|y-x_0|\approx\frac{\tau_2-\tau_1}{\tau_2}|y-x_0|
\end{equation}
(as $\tau_2|y-x|\leq \tau_2|y-x_0|+\tau_2|x-x_0|\leq \tau_2|y-x_0|+|y-x_0||x-x_0|$), we estimate 
\begin{align}\label{t4}
-(III)& \leq c\,\nu\int_{\R^n\setminus B_{\tau_2}}\int_{B_{\tilde{\tau}_2}}\Big[|u(x)-(u)_{B_{\tau_2}}||u(y)-(u)_{B_{\tau_2}}|+|u(x)-(u)_{B_{\tau_2}}|^2\Big]\,\frac{dx\dy}{|x-y|^{n+2s}}\notag\\
& \stackleq{est:dist}c\,\nu\Big(\frac{\tau_2}{\tau_2-\tau_1}\Big)^{n+2s} \tau_2^{-2s}\bigg[\tail\left(u-(u)_{B_{\tau_2}};B_{\tau_2}\right)\int_{B_{\tau_2}}\big|u-(u)_{B_{\tau_2}}\big|\dx\notag\\
&\pushright{+\,\tail\left(1;B_{\tau_2}\right)\int_{B_{\tau_2}}\big|u-(u)_{B_{\tau_2}}\big|^2\dx\bigg]}\notag\\
&\stackleq{tail.constant} \frac{c\,\nu\tau_2^n}{(\tau_2-\tau_1)^{n+2s}}\bigg[\tail\big(u;B_{\tau_2}\big)\int_{B_{\tau_2}} |u|\dx+\bigg(\int_{B_{\tau_2}}|u|\dx\bigg)^2+\int_{B_{\tau_2}}|u|^2\dx\bigg]\notag\\
&\stackrel{\mbox{Jensen}}{\leq} \frac{c\,\nu\tau_2^n}{(\tau_2-\tau_1)^{n+2s}}\bigg[\tail\big(u;B_{\tau_2}\big)\int_{B_{\tau_2}} |u|\dx+\int_{B_{\tau_2}}|u|^2\dx\bigg]\notag\\
& \stackleq{tail:4} \frac{c\,\nu\tau_2^n}{(\tau_2-\tau_1)^{n+2s}}\,\bigg[\Big[\tail\big(u; B_r\big)+\frac{r^n}{\tau_2^n}\int_{B_r}|u|\dx\Big]\int_{B_r}|u|\dx+\int_{B_r}|u|^2\dx \bigg]\notag\\
&\leq  \frac{c\,\nu r^n}{(\tau_2-\tau_1)^{n+2s}}\bigg[ \tail\big(u; B_r\big)\,\int_{B_r} |u|\dx+\int_{B_r}|u|^2\dx \bigg]\,.
\end{align}
Inserting \eqref{t1}--\eqref{t4} into \eqref{t0}, choosing $\delta=\Lambda^{-1}/8$, multiplying both members by $r^n$ and recalling that $\eta\equiv 1$ on $B_{\tau_1}$ and $r\leq r_0$, we find
\begin{multline*}
\int_{B_{\tau_1}}|Du|^2\dx\leq  \bar c\,\nu\,r_0^{2(1-s)}\int_{B_{\tau_2}} |Du|^2\dx+c\,\Big(\frac1{(\tau_2-\tau_1)^2}+\frac\nu{(\tau_2-\tau_1)^{2s}}\Big)\int_{B_{r}}|u|^2\dx\\
+\frac{c\,\nu\,r^n}{(\tau_2-\tau_1)^{n+2s}}\bigg[ \tail\big(u; B_r\big)\,\int_{B_r} |u|\dx+\int_{B_r}|u|^2\dx \bigg]\\
+c\,r^{n+2(1-\frac n\chi)}\bigg(\int_{B_r}|f|^\chi\dx\bigg)^{2/\chi}+c\int_{B_r} \big[|G|^2+1\big]\dx
\end{multline*}
with $c$ and $\bar c$ depending only on $\data$; now we choose $r_0\leq 1$ such that
\[
 \bar c\,\nu\,r_0^{2(1-s)}\leq \frac12
\]
(which yields \eqref{cond:r0}) and we apply the well-known ``simple, but fundamental Lemma'' by Giaquinta \& Giusti \cite[Lemma 1]{GG} (see also \cite[Lemma 6.1]{giusti}) to reabsorb the first integral on the right-hand side, getting
\begin{multline*}
    \int_{B_{\s r}}|Du|^2\dx\leq c\,\Big(\frac1{(1-\sigma)^2r^2}+\frac\nu{(1-\sigma)^{2s}r^{2s}}\Big)\int_{B_r}|u|^2\dx\\
    +\frac{c\,\nu}{(1-\s)^{n+2s}r^{2s}}\tail\big(u; B_r\big)\int_{B_r} |u|\dx+c\,r^{n+2(1-\frac n\chi)}\bigg(\int_{B_r}|f|^\chi\dx\bigg)^{2/\chi}+c\int_{B_r} \big[|G|^2+1\big]\dx\,.
\end{multline*}
After taking averages, this is the desired estimate.
    \end{proof}

Our next goal is to show that, under suitable integrability assumptions on $f,G$, we have higher integrability of the gradient of solutions to \eqref{eq2}.
Before stating such result, we anticipate a useful lemma to rewrite the right-hand side of \eqref{eq2} in divergence form by using classical potential theory for the Poisson equation. More precisely, we have
\begin{lemma}\label{lem:pot}
Let $f\in L^{\chi}(B_R(x_0))$, with $\chi>2_*$. Then there exist an exponent $q>2$ and a vector field $F\in L^q(B_R(x_0);\mathbb R^n)$ such that  $f=\mathrm{div}\, F$ in distributional sense, and satisfying
\begin{equation}\label{pot:est}
\bigg(\mint_{B_R(x_0)} |F|^q\dx\bigg)^{1/q} \leq c(n,\chi,q)\, R \bigg(\mint_{B_R(x_0)} |f|^\chi\dx\bigg)^{1/\chi}\,.
\end{equation}
If $2_*<\chi<n$, one may take
\[
q=\chi^*:=\frac{n\chi}{n-\chi}>2\,,
\]
while if $\chi\geq n$, one may take any number $q>2$ (actually, for $\chi>n$ one might even take $q=+\infty$).
\end{lemma}
\begin{proof}
We extend $f$ to the whole space by setting $f\equiv 0$ in $\R^n\setminus B_{R}(x_0)$. Consider the function 
    \begin{equation*}
        w=\int_{\R^n} \Gamma(x-y)f(y)\dy\,,
    \end{equation*}
    where $\Gamma$ is the fundamental solution of the Laplace equation in $\R^n$. It follows that $-\Delta w=f$, and classical potential estimates (\cite[Section 7.8]{gt}, \cite[Section 2]{min:ku}) give $|Dw| \leq c(n)\,{\bf I}_1(|f|)$. The assertion follows now by setting $F=-Dw$, and taking into account the Hardy--Littlewood--Sobolev inequality and its standard localized consequences. If $1<\chi<n$, this is immediate as $\chi^*>2$ by $\chi>2_*$. If $\chi\geq n$, we fix $q>2$ and choose 
\[ 
\chi_0=q_*=\frac{nq}{n+q}\in (2_*,n),\qquad q=\chi_0^*
\]
and using the previous case with exponent $\chi_0$ and H\"older's inequality gives \eqref{pot:est}. Finally, if $\chi>n$, one can use H\"older's inequality applied to the kernel $|x-y|^{1-n}$. 
\end{proof}
The main result of this section is the following:

\begin{proposition}[Higher integrability of the gradient]\label{prop:highint}
Let $u\in W^{1,2}_{\rm loc}(\Omega;\R^N)\cap L^1_{2s}$ be a weak solution to \eqref{eq2}, under assumptions \eqref{phi:gr}--\eqref{sym:phi}, \eqref{gr:Agen} in force, and $f=\mathrm{div}\,F $ as above.  Then there exist constants
\[
\zeta=\zeta(\data,\max\{\nu,1\})>0\quad\text{and}\quad c=c\big(\data,\max\{\nu,1\}\big)\geq 1
\]
 such that, if $F,G\in L^{2+\zeta}_{\rm loc}(\Omega;\R^{nN})$, then $|Du|\in L^{2+\zeta}_{\rm loc}(\Omega)$, with quantitative estimate
{\rm
\begin{multline}\label{high:int}
\bigg(\mint_{B_{R_0/2}}|Du|^{2+\zeta}\dx\bigg)^{1/(2+\zeta)} \leq c\bigg(\mint_{B_{2R_0}}|Du|^2\dx\bigg)^{1/2}+c\bigg(\mint_{B_{2R_0}}\big[|F|+ |G|+1\big]^{2+\zeta}\dx\bigg)^{1/(2+\zeta)}\\
        +c\,\nu R_0^{1-2s}\tail\big(u-(u)_{B_{2R_0}};B_{2R_0}\big)\,.
\end{multline}}
This holds for all $B_{2R_0}\equiv B_{2R_0}(x_0)\Subset \Omega$ with $R_0\leq r_0(\nu)\leq 1$ satisfying \eqref{cond:r0}.
\end{proposition}
The proof of Proposition \ref{prop:highint} is very similar to that of \cite[Proposition 3.4]{byun}, our starting point being \eqref{caccioppoli}. For the sake of completeness, we sketch the full proof.

\vs

We start by proving a reverse H\"older inequality, in the same spirit as \cite[Lemma 3.3]{byun}. To this end, we fix a parameter  $\tau>0$, depending only on $s$, satisfying
\begin{equation}\label{tau}
    \max\{s,1/2\}<\tau<1.
\end{equation}

\begin{proposition}[Reverse H\"older's inequality]
    Let $u\in W^{1,2}_{\rm loc}(\Omega;\R^N)\cap L^1_{2s}$ be a weak solution to \eqref{eq2}, under assumptions \eqref{phi:gr}--\eqref{sym:phi}, \eqref{gr:Agen} in force. Suppose that $f=\mathrm{div}\,F$, with $F,G\in L^2_{\rm loc}(\Omega;\R^{nN})$. Let $r<r_0(\nu)$, where $r_0(\nu)$ is the threshold appearing in Lemma \ref{lemma:cacc}, $x_0\in \Omega$, and let $\ell\in \N$ be such that 
\[
B_{2^{\ell+1}r}(x_0)\Subset \Omega\quad\text{and}\quad 2^{\ell+1}r\leq 1\,.
\]
\rm{
Then there exists a constant $c=c(\data)$ such that
\begin{multline}\label{rev:hold}
    \mint_{B_{r/2}}|Du|^2\dx \leq c\,\Big(1+\frac{\nu^2}{\delta}\Big)\bigg(\mint_{B_r}|Du|^{2n/(n+2)}\dx\bigg)^{(n+2)/n}\\
+\delta\bigg[\bigg(\mint_{B_{2^\ell r}}|Du|^{2n/(n+2)}\dx\bigg)^{(n+2)/n}+\int_{r}^{2^\ell r}\mu^{1-2\tau}\bigg(\mint_{B_\mu}|Du|^{2n/(n+2)}\dx\bigg)^{(n+2)/n}\,\frac{d\mu}{\mu}\bigg]\\
+\delta\bigg[\nu\,r(2^\ell r)^{-2s}\tail\big(u-(u)_{B_{2^\ell r}};B_{2^\ell r}\big)\bigg]^2+c\mint_{B_r}\big[|F|^2+|G|^2+1\big]\dx
\end{multline}
}
   holds for all $\delta\in (0,1]$.
\end{proposition}

\begin{proof}
In view of Lemma \ref{lem:pot} and \eqref{pot:est}, we may assume that $f\equiv 0$ by replacing $G$ with $F+G$, where $f=\mathrm{div}\, F$; we omit the center of the various balls under consideration. We want to estimate the terms on the right-hand side of \eqref{caccioppoli} with $\s=1/2$ and $u_0=(u)_{B_{r}}$. By Sobolev and H\"older inequalities, we have
\begin{equation}\label{a00}
    \mint_{B_{r}(x_0)}|u-(u)_{B_{r}(x_0)}|\dx\leq \Big(\mint_{B_{r}(x_0)}{\big|u-(u)_{B_{r}(x_0)}\big|}^2\dx\Big)^{1/2}\leq  c(n,N)\,r\Big( \mint_{B_{r}(x_0)}|Du|^{2n/(n+2)}\dx\Big)^{(n+2)/[2n]}\,;
\end{equation}
using \eqref{tail:1}, we can estimate the tail term as follows:
\begin{multline*}
    r^{-2s}\tail\big( u-(u)_{B_{r}};B_r\big)\leq c\,(2^\ell r)^{-2s}\tail\big(u-(u)_{B_{2^\ell r}};B_{2^\ell r}\big) \\
    +c\,(2^\ell r)^{-2s}\mint_{B_{2^\ell r}}|u-(u)_{B_{2^\ell r}}|\dx+ \,c\int_{r}^{2^\ell r}\mu^{-2s}\mint_{B_\mu}\big|u-(u)_{B_\mu}\big|\dx\,\frac{d\mu}{\mu}\,.
\end{multline*}
Using the previous estimate, Poincar\'e's inequality and then repeated use of weighted Young's inequality, we get
\begin{align}\label{a2}
\nu\,r^{-2s} &\tail\big( u-(u)_{B_{r}};B_r\big)\mint_{B_r}|u-(u)_{B_r}|\dx\notag\\
&\leq c\,\nu\bigg[r(2^\ell r)^{-2s}\tail\big(u-(u)_{B_{2^\ell r}};B_{2^\ell r}\big)\notag\\
&\pushright{+\,r(2^\ell r)^{1-2s}\mint_{B_{2^\ell r}}|Du|\dx + r \int_r^{2^\ell r}\mu^{1-2s}\mint_{B_\mu}|Du|\dx\,\frac{d\mu}{\mu}\bigg]\mint_{B_r}|Du|\dx}\notag\\
& \leq  \delta_1\Big[\nu\,r(2^\ell r)^{-2s}\tail\big(u-(u)_{B_{2^\ell r}};B_{2^\ell r}\big) \Big]^2+\delta_2\bigg(\mint_{B_{2^\ell r}}|Du|\dx\bigg)^2\notag\\
&\qquad+\delta_3\bigg[\int_{r}^{2^\ell r}\mu^{1-2s}\mint_{B_\mu}|Du|\dx\,\frac{d\mu}{\mu}\bigg]^2+c\,\bigg[\frac{1}{\delta_1}+\frac{\nu^2r^2(2^\ell r)^{2(1-2s)}}{\delta_2}+\frac{\nu^2 r^2}{\delta_3}\bigg]\bigg( \mint_{B_r}|Du|\dx\bigg)^2
\end{align}
for $\delta_i\in (0,1]$, $i=1,2,3$. Moreover, for $\mu\leq 2^\ell r\leq 1$ we may bound $\mu^{1-2s}\leq \mu^{1-2\tau}$ thanks to \eqref{tau}, so  by H\"older's inequality applied with measure $d\l(\mu)=\mu^{1-2\tau}\,\frac{d\mu}{\mu}$, we compute
\begin{align}\label{a3}
    \bigg(\int_{r}^{2^\ell\,r}\mu^{1-2s}\mint_{B_\mu}|Du|\dx\,\frac{d\mu}{\mu}\bigg)^2 & \leq \int_r^{2^\ell r}d\l(\mu) \,\int_r^{2^\ell r}\bigg(\mint_{B_\mu}|Du|\dx\bigg)^2\,d\l(\mu)\notag
    \\
    &\leq \frac{r^{1-2\tau}}{2\tau-1}\int_r^{2^\ell r}\mu^{1-2\tau}\bigg(\mint_{B_\mu}|Du|\dx\bigg)^2\,\frac{d\mu}{\mu}\,;
\end{align}
in the last inequality we used that $\tau\geq 1/2$ by \eqref{tau}. Combining \eqref{a2}--\eqref{a3}, choosing  $\delta_1=\delta_2=\delta$ and $\delta_3=(2\tau-1)^2r^{2(2\tau-1)}\delta$ and finally using the facts $r^2(2^\ell r)^{2(1-2s)}\leq (2^\ell r)^{2(2-2s)}\leq 1$ and $r^{1-\tau}\leq 1$ by \eqref{tau}, we infer
\begin{multline*}
\nu\,  r^{-2s} \tail\big( u-(u)_{B_{r}};B_r\big)\mint_{B_r}|u-(u)_{B_r}|\dx\leq \delta\Big[\nu\,r(2^\ell r)^{-2s}\tail\big(u-(u)_{B_{2^\ell r}};B_{2^\ell r}\big) \Big]^2\\
+\delta\bigg(\mint_{B_{2^\ell r}}|Du|\dx\bigg)^2+\delta\int_{r}^{2^\ell\,r}\mu^{1-2\tau}\bigg(\mint_{B_\mu}|Du|\dx\bigg)^2\,\frac{d\mu}{\mu}+\frac{c\,(1+\nu^2)}{\delta}\bigg(\mint_{B_r}|Du|\dx\bigg)^2
\end{multline*}
with $c$ depending on $n,N$ and $s$. To conclude, it suffices to use Caccioppoli's inequality \eqref{caccioppoli} with the choice $\sigma=1/2$, $u_0=(u)_{B_r}$, which, together with  \eqref{a00}, \eqref{a3} and H\"older's inequality gives
\begin{multline*}
\mint_{B_{r/2}}|Du|^2\dx  \leq  c\,\Big(1+\frac{\nu^2}{\delta}+\frac\nu{r^{2(s-1)}}\Big)\bigg(\mint_{B_{r}}|Du|^{2n/(n+2)}\dx\bigg)^{(n+2)/n}+c\mint_{B_r}\big[|F|^2+|G|^2+1\big]\dx\notag\\
+\delta\bigg[\bigg(\mint_{B_{2^\ell r}}|Du|^{2n/(n+2)}\dx\bigg)^{(n+2)/n}+\int_{r}^{2^\ell\,r}\mu^{1-2\tau}\bigg(\mint_{B_\mu}|Du|^{2n/(n+2)}\dx\bigg)^{(n+2)/n}\,\frac{d\mu}{\mu}\bigg]\notag\\
+\delta\Big[\nu\,r\,(2^\ell r)^{-2s}\tail\big(u-(u)_{B_{2^\ell r}};B_{2^\ell r}\big) \Big]^2\,,
\end{multline*}
and using that $\nu/r^{2(s-1)}\leq c(\data)$ by \eqref{cond:r0}, we get the desired outcome.
\end{proof}

We are now in the position to prove Proposition \ref{prop:highint}, whose proof is based on the reverse H\"older's inequality \eqref{rev:hold} and an exit time argument as in \cite{cm,kms} and \cite[Proposition 3.4]{byun}.

\begin{proof}[Proof of Proposition \ref{prop:highint}]
In the following, we set
\begin{equation*}
|\tilde{F}|\coloneqq |F|+|G|+1\,,
\end{equation*}
and we fix two constants $\zeta>0$ and $M\geq 1$ depending only on  $\data$ and $\max\{\nu,1\}$, to be determined later.
Given $R_0$ and $x_0$ as in the statement of the proposition, for all $R$ and $\bar x\in \Omega$ such that $B_R(\bar x)\subseteq B_{R_0}(x_0)$, we set
\begin{align*}
        \Upsilon(\bar x,R) & \coloneqq \bigg(\mint_{B_R(\bar x)}|\tilde{F}|^{2+\zeta}\dx\bigg)^{1/(2+\zeta)}\notag
        \\
        \Psi_M(\bar x,R) & \coloneqq \bigg( \mint_{B_R(\bar x)}|Du|^2\dx\bigg)^{1/2}+M\,\bigg(\mint_{B_R(\bar x)}|\tilde{F}|^2\dx\bigg)^{1/2}\,,
\end{align*}
    and
\begin{align}\label{def:nuove1}
        \mathcal{T}(\bar x,R)\coloneqq \bigg[\int_R^{2^\ell R}\mu^{1-2\tau}\bigg( \mint_{B_\mu(\bar x)}|Du|^{2\gamma}\dx\bigg)^{1/\g}\frac{d\mu}{\mu}\bigg]^{1/2}+\nu\,R\int_{\R^n\setminus B_{2^\ell R}(\bar x)}\frac{\big|u(y)-(u)_{B_{2^\ell R}(\bar x)}\big|}{{|y-\bar x|}^{n+2s}}\dy\,,
    \end{align}
    where $\gamma=n/(n+2)\in[1/2,1)$ and  $\ell=\ell(R)\in \N$ is such that 
    \begin{equation}\label{cond:ell}
        \frac{R_0}{2}\leq 2^\ell R<R_0\,;
    \end{equation}
we also set
    \begin{align}\label{def:nuove2}
        \Xi(\bar x,R)\coloneqq\Upsilon(\bar x,R)+\Psi_M(\bar x,R)+\mathcal{T}(\bar x,R)\,.
    \end{align}
We fix two radii $r_1,r_2>0$ such that 
    \begin{equation}\label{radii0}
        \frac{R_0}{2}\leq r_1<r_2\leq R_0
    \end{equation}
    and define the quantity
    \begin{equation}\label{def:l0}
        \l_0\coloneqq \sup\left\{\Xi(\bar x,R):\,\bar x\in B_{r_1}(x_0),\,\frac{r_2-r_1}{16}<R\leq \frac{R_0}{2}\right\}\,.
    \end{equation}
\noindent\textbf{Step 1: Upper bound for $\l_0$.} Let
\begin{equation}\label{fix:st1}
\bar x\in B_{r_1}(x_0)\qquad \text{and}\qquad \frac{r_2-r_1}{16}<R\leq \frac{R_0}{2}\,;
\end{equation}
we have $B_R(\bar x)\subseteq B_{2R_0}(x_0)$, and for this reason
\begin{equation}\label{a7}
 \Upsilon(\bar x,R)\leq \Big(\frac{2R_0}{R} \Big)^{n/(2+\zeta)}\Upsilon(x_0,2R_0)\leq \Big(\frac{32R_0}{r_2-r_1} \Big)^{n/2}\Upsilon(x_0,2R_0)\,.
\end{equation}
Moreover, by \eqref{cond:ell}--\eqref{fix:st1}, we have  $B_{2^{\ell+1}R}(\bar x)\subseteq B_{2R_0}(x_0)$, so that for all $R\leq \mu\leq 2^\ell R$, we similarly have
\begin{equation*}
    \bigg(\mint_{B_\mu(\bar x)}|Du|^{2\g}\dy\bigg)^{1/\g}\leq c(n)\Big( \frac{32 R_0}{r_2-r_1}\Big)^{{n}/{\g}}\bigg(\mint_{B_{2R_0}(x_0)}|Du|^{2\g}\dy\bigg)^{{1}/{\g}}\,.
\end{equation*}
In particular, by \eqref{tau} and $R\leq R_0\leq 1$, we get
\[
    \int_R^{2^\ell R}\mu^{1-2\tau}\bigg( \mint_{B_\mu(\bar x)}|Du|^{2\gamma}\dy\bigg)^{1/\g}\,\frac{d\mu}{\mu}\leq c(n,s)\Big(\frac{32R_0}{r_2-r_1} \Big)^{n+2}\mint_{B_{2R_0}(x_0)}|Du|^2\dy\,,
\]
where we also estimated the last integral using H\"older's inequality, since $\gamma=n/(n+2)<1$. Again, using the inclusion $B_{2^{\ell+1}R}(\bar x)$ $\subseteq B_{2R_0}(x_0)$, we may use \eqref{tail:3}, \eqref{fix:st1} and Poincar\'e's inequality to deduce
\begin{align*}
   \nu R\int_{\R^n\setminus B_{2^\ell R}(\bar x)}\frac{\big|u(y)-(u)_{B_{2^\ell R}(\bar x)}\big|}{{|y-\bar x|}^{n+2s}}\dy&= \nu R (2^\ell R)^{-2s}\tail\big(u-(u)_{B_{2^\ell R}(\bar x)};  B_{2^\ell R}(\bar x)\big)\notag
     \\
     &\leq c\,\nu R (2R_0)^{-2s}\Big( \frac{2R_0}{2R_0-|\bar x-x_0|}\Big)^{n+2s}\tail\big( u-(u)_{B_{2R_0}(x_0)};  B_{2R_0}(x_0)\big)\notag
     \\
     &\pushright{+\,c\,\nu R (2^\ell R)^{-2s}\Big( \frac{2R_0}{2^\ell R}\Big)^n\mint_{B_{2R_0}}|u-(u)_{B_{2R_0}(x_0)}|\dx}\notag
     \\
      &\leq c\,\nu R \Big( \frac{2R_0}{r_2-r_1}\Big)^{n+2s}(2R_0)^{-2s}\tail\big( u-(u)_{B_{2R_0}(x_0)};  B_{2R_0}(x_0)\big)\notag
      \\
      &\pushright{+\,c\,\nu R (2^\ell R)^{-2s}\frac{(2R_0)^{n+1}}{(2^\ell R)^n}\bigg(\mint_{B_{2R_0}(x_0)}|Du|^2\dx \bigg)^{1/2}\,;}
\end{align*}
here $c=c(n,N,s)$. Since $2^\ell R\leq R_0\leq 1$, by \eqref{tau} we estimate $(2^\ell R)^{-2s}\leq (2^\ell R)^{-2\tau}\leq R^{-2\tau}$, whence from the above inequality we find
\begin{align}\label{a9}
\nu\,R \int_{\R^n\setminus B_{2^\ell R}(\bar x)}\frac{|u(y)-(u)_{B_{2^\ell R}(\bar x)}|}{{|y-\bar x|}^{n+2s}} \dy &\leq c\, \nu\,R \Big( \frac{2R_0}{r_2-r_1}\Big)^{n+2s}  (2R_0)^{-2s}\tail\big( u-(u)_{B_{2R_0}(x_0)};  B_{2R_0}(x_0)\big)\notag\\
&\pushright{+ \,c\, \nu\,\frac{(2R_0)^{n+1}}{R^{n+2\tau-1}} \bigg(\mint_{B_{2R_0}(x_0)}|Du|^2 \dx \bigg)^{1/2}}\notag\\
&\leq  c\,\nu\Big( \frac{2R_0}{r_2-r_1}\Big)^{n+2s}  (2R_0)^{1-2s}\tail\big( u-(u)_{B_{2R_0}(x_0)};  B_{2R_0}(x_0)\big)\notag\\
&\pushright{+\, c\,\nu\Big( \frac{32R_0}{r_2-r_1}\Big)^{n+2\tau-1} \bigg(\mint_{B_{2R_0}(x_0)}|Du|^2 \dx \bigg)^{1/2}\,,}
\end{align}
with $c=c(n,N,s)$, where in the last inequality we also used \eqref{fix:st1} and $n+1>n+(2\tau-1)>n$ by \eqref{tau}.
Combining \eqref{a7}--\eqref{a9}, we thus find
\begin{equation}\label{bd:l0}
  \l_0\leq c\,M\Big(\frac{32R_0}{r_2-r_1}\Big)^{n+2}\Xi_0\,,
\end{equation}
where $c=c(n,N,s)$, and we have set
\begin{multline}\label{def:xi0}
    \Xi_0\coloneqq  \bigg(\mint_{B_{2R_0}(x_0)} |\tilde{F}|^{2+\zeta}\dx\bigg)^{1/(2+\zeta)}+\bigg(\mint_{B_{2R_0}(x_0)}|Du|^2\dx\bigg)^{1/2}\\
+\nu\,(2R_0)^{1-2s}\tail\big( u-(u)_{B_{2R_0}(x_0)};  B_{2R_0}(x_0)\big)\,.
\end{multline}

\noindent\textbf{Step 2: Exit time and Vitali covering.} Let $\lambda\geq \lambda_0$, and define
\[
    E_\lambda\coloneqq \bigg\{\bar x\in B_{r_1}(x_0):\,\sup_{0\leq R\leq \frac{r_2-r_1}{16}}\Psi_M(\bar x,R)>\lambda\bigg\}\,.
\]
By definition \eqref{def:l0}, we have
\[
    x\in B_{r_1}(x_0),\,R\in \left[\frac{r_2-r_1}{16},\frac{R_0}{2}\right]\implies \Psi_M(x,R)\leq \lambda_0\leq \lambda\,.
\]
On the other hand, for every $x\in E_\lambda$, there exists an exit radius $R_x\in \left(0,\frac{r_2-r_1}{16} \right) $ such that
\begin{equation}\label{exit}
    \Psi_M(x,R_x)=\lambda\qquad\text{and}\qquad \Psi_M(x,R)\leq \lambda \text{ for all $R\in \Big(R_x,\frac{R_0}{2}\Big)$}\,.
\end{equation}
By Vitali's covering argument, we may find a countable collection of pairwise disjoint balls $\{B_{2R_j}(x_j)\}_{j\in \mathcal J}$, with cardinality $\sharp \mathcal J\leq\sharp \N$, with centers $\{x_j\}_{j\in \mathcal J}\subseteq E_\l$ and radii $\{R_j\}_{j\in \mathcal J}$ such that
\begin{equation}\label{vit:1}
    E_\l\subseteq \bigcup_{j\in \mathcal J} B_{10 R_j}(x_j)\cup\mathcal N\,,
\end{equation}
with $\mathcal N$ negligible. For notational simplicity, we set $B_j=B_{R_j}(x_j)$, and denote by $KB_j=B_{KR_j}(x_j)$ for $K> 0$. For all $j\in\mathcal J$, let us pick $\ell_j=\ell(R_j)\in \N$ satisfying
\begin{equation*}
    \frac{R_0}{2}\leq  2^{\ell_j} R_j\leq  R_0\,.
\end{equation*}
Since $R_j<R_0/8$, it follows that $\ell_j\geq 3$ for all $j\in\mathcal J$, and $2^{\ell_j+1} B_j=B_{2^{\ell_j+1}R_j}(x_j)\subseteq B_{2R_0}(x_0)$.

\vs

We aim to estimate the measure $|B_j|$. By \eqref{exit}, we have that either
\begin{equation}\label{case1}
\bigg(\mint_{B_j} |Du|^2\dx\bigg)^{1/2}\geq \frac{\l}{2}\qquad\text{or}\qquad M\,\bigg(\mint_{B_j}|\tilde{F}|^2\dx \bigg)^{1/2}\geq \frac{\l}{2}
\end{equation}
holds true. Assume that the first condition in \eqref{case1} is in force. This information coupled with \eqref{rev:hold} (with $r=2R_j$) gives
\begin{multline}\label{a10}
    \lambda^2 \leq c\,\Big(1+\frac{\nu^2}{\delta}\Big)\bigg(\mint_{2B_j}|Du|^{2\gamma}\dx\bigg)^{1/\gamma} +\delta\Big[\nu\,2R_j(2^{\ell_j+1}R_j)^{-2s}\tail\big(u-(u)_{B_{2^{\ell_j+1}R_j}};B_{2^{\ell_j+1}R_j}\big)\Big]^2\\
+\delta\bigg[\bigg(\mint_{2^{\ell_j}B_j}|Du|^{2\gamma}\dx\bigg)^{1/\gamma}+\int_{2R_j}^{2^{\ell_j+1}R_j}\mu^{1-2\tau}\bigg(\mint_{B_\mu(x_j)}|Du|^{2\gamma}\dx\bigg)^{1/\gamma}\,\frac{d\mu}{\mu}\bigg]+c\mint_{2B_j}|\tilde{F}|^2\dx\,;
\end{multline}
furthermore, thanks to \eqref{exit}, for all $\mu\geq R_j$ we have
\begin{equation*}
    \mint_{B_\mu(x_j)}|Du|^{2}\dy\leq \l^2\,,
\end{equation*}
while, by \eqref{def:nuove1}, \eqref{def:nuove2}, \eqref{def:l0} and the fact that $\lambda\geq\lambda_0$, we have
\begin{multline*}
\nu\,2R_j\,(2^{\ell_j+1}R_j)^{-2s}\tail\big(u-(u)_{2^{\ell_j+1}B_j};2^{\ell_j+1}B_j\big)=\nu\,2R_j\int_{\R^n\setminus 2^{\ell_j+1}B_j}\frac{|u(y)-(u)_{2^{\ell_j+1}B_j}|}{{|y-x_j|}^{n+2s}}\dy\leq \mathcal{T}(x_j,2R_j)\\\leq \Xi(x_j, 2R_j)\leq \l\,,
\end{multline*}
and
\begin{equation*}
    \mint_{2B_j}|\tilde{F}|^2\dx\leq \frac{[\Xi(x_j,2R_j)]^2}{M^2}\leq \frac{\l^2}{M^2}\,.
\end{equation*}
Using the last three estimates in \eqref{a10} together with H\"older's inequality, yields
\begin{equation*}
    \l^2\leq c\,\Big(1+\frac{\nu^2}{\delta}\Big)\bigg(\mint_{2B_j}|Du|^{2\gamma}\dx \bigg)^{1/\g}+c\frac{\l^2}{M^2}+c\,\delta\,\l^2\,.
\end{equation*}
Hence we fix $\delta$ small enough and $M$ large enough, both depending only on $\data,\max\{\nu,1\}$, so as to obtain
\begin{equation*}
    \l^2\leq c(1+\nu)^2\bigg(\mint_{2B_j}|Du|^{2\gamma}\dx \bigg)^{1/\g}\,,
\end{equation*}
from which it follows, for any $\k>0$,
\begin{equation}\label{Bj1}
    |B_j|\leq \frac{c(1+\nu)^{2\gamma}}{\l^{2\gamma}}\int_{2B_j}|Du|^{2\gamma}\dx\leq \frac{c(1+\nu)^{2\gamma}}{\l^{2\gamma}}\int_{2B_j\cap \{|Du|>\k\l\}}|Du|^{2\g}\dx+c(1+\nu)^{2\gamma}\k^{2\g}|B_j|\,;
\end{equation}
by choosing $\k=\k(\data,\max\{\nu,1\})\in (0,1)$ so that $c(1+\nu)^{2\gamma}\k^{2\g}\leq 1/2$, from \eqref{Bj1} we get
\begin{equation}\label{Bj11}
    |B_j|\leq \frac{c(1+\nu)^{2\gamma}}{\l^{2\gamma}}\int_{2B_j\cap \{|Du|>\k\l\}}|Du|^{2\gamma}\dx
\end{equation}
with $c\equiv c(\data)$. If, on the other hand, the second condition in \eqref{case1} holds true, then trivially (recall that now $M$ is fixed depending on the $\data$)
\begin{equation*}
    |B_j|\leq \frac{c}{\l^2}\int_{B_j}|\tilde{F}|^2\dx\,,
\end{equation*}
and arguing as before we find $\k=\k(\data,\max\{\nu,1\})\in (0,1)$ small enough such that
\begin{equation}\label{Bj2}
    |B_j|\leq \frac{c}{\l^2}\int_{B_j\cap \{|\tilde{F}|>\k\l\}}|\tilde{F}|^2\dx
\end{equation}
(clearly we can suppose that the two small constants $\kappa$ in \eqref{Bj11} and \eqref{Bj2} are the same). In any case, from \eqref{Bj11} and \eqref{Bj2} we deduce
\begin{equation}\label{est:Bj}
    |B_j|\leq \frac{c(1+\nu)^{2\g}}{\l^{2\g}}\int_{2B_j\cap \{|Du|>\k\l\}}|Du|^{2\g}\dx+\frac{c}{\l^2}\int_{2B_j\cap\{|\tilde{F}|>\k\l\}}|\tilde{F}|^2\dx\,.
\end{equation}
From this point the proof is standard, thus we only sketch the main points; for the details see \cite{AM}, \cite[Section 8]{BB}, \cite[Section 5.2]{byun} with the appropriate modifications. We observe that \eqref{vit:1} implies
\begin{equation}\label{a11}
    B_{r_1}(x_0)\cap \{|Du|>\l\}\subseteq E_\l\subseteq \bigcup_{j\in\mathcal J} 10 B_j\quad\text{up to a negligible set,}
\end{equation}
and this can be easily proven considering Lebesgue points of the gradient in $B_{r_1}(x_0)\cap \{|Du|>\l\}$.

\vs

\noindent\textbf{Step 3: The $\l$-inequality.} From now on we denote $B_{r_1}, B_{r_2}$ for short, omitting their center $x_0$. As $\{2B_j\}$ is a pairwise disjoint family and $\bigcup_{j\in\mathcal J} 2B_j\subseteq B_{r_2}$, summing \eqref{est:Bj} over $j\in\mathcal J$ yields
\begin{equation}\label{est:sBj}
\sum_{j\in\mathcal J} |B_j|\leq \frac{c(1+\nu)^{2\g}}{\l^{2\g}}\int_{B_{r_2}\cap \{|Du|>\k\l\}}|Du|^{2\g}\dx+\frac{c}{\l^2}\int_{B_{r_2}\cap\{|\tilde{F}|>\k\l\}}|\tilde{F}|^2\dx\,.
\end{equation}
On the other hand, by \eqref{exit}, we have $\Psi_M(x_j,R)\leq \l$ if $R\geq R_j$, hence
\begin{equation}\label{a12}
    \int_{10B_j}|Du|^2\dx\leq \l^2|10B_j|\qquad \text{for all $j\in\mathcal J$\,;}
\end{equation}
from \eqref{a11}--\eqref{a12} we infer
\[
\int_{B_{r_1}\cap \{|Du|>\l\}}|Du|^2\dx\leq \sum_{j\in\mathcal J} \int_{10B_j}|Du|^2\dx\leq  10^n \l^2\,\sum_{j\in\mathcal J} |B_j| 
\]
and combining the latter inequality with \eqref{est:sBj}, we obtain
\begin{equation*}
    \int_{B_{r_1}\cap \{|Du|>\l\}}|Du|^2\dx\leq \frac{c(1+\nu)^2}{\l^{2\g-2}}\int_{B_{r_2}\cap \{|Du|>\k\l\}}|Du|^{2\g}\dx+c\int_{B_{r_2}\cap\{|\tilde{F}|>\k\l\}}|\tilde{F}|^2\dx\,,
\end{equation*}
that is 
\begin{align}\label{a14}
    \int_{\mathcal{U}_\l(r_1)}|Du|^2\dx\leq c\,\l^{2-2\g}\int_{\mathcal{U}_{\k\l}(r_2)}|Du|^{2\g}\dx+c\int_{\mathcal{F}_{\k\l}(r_2)}|\tilde{F}|^2\dx
\end{align}
(where now $c$ may also depend on $\max\{\nu,1\}$) where we have introduced the sets
\begin{align*}
\mathcal{U}_\l(r)=\big\{x\in B_r(x_0): |Du|(x)>\l\big\},\qquad \mathcal{F}_\l(r)=\big\{x\in B_r(x_0): |\tilde{F}|>\l\big\}\,.
\end{align*}

\noindent\textbf{Step 4: Integration and conclusion.} To ensure that all the quantities we are considering are finite, for any $t\geq 0$ we should define the truncated functions
\begin{equation*}
    |Du|_t=\min\big\{|Du|,t\big\} \quad \text{and}\quad|\tilde{F}|_t=\min\big\{|\tilde{F}|,t\big\}
\end{equation*}
and proceed as in \cite{BB}; we shall instead proceed formally and let the reader follow \cite[Section 8]{BB} to implement the necessary changes needed to make the proof rigorous.
Let us multiply  \eqref{a14} by $\l^{\zeta-1}$, where $\zeta>0$ is as in the statement, and then integrate it with respect to $\lambda$ over $(\lambda_0,+\infty)$ to obtain, after changing variables in the two integrals on the right-hand side,
\begin{align*}
\int_{\lambda_0}^\infty  \l^{\zeta}   \int_{\mathcal{U}_\l(r_1)}&|Du|^2\dx\,\frac{d\lambda}{\lambda} \\
&\leq \frac c{\kappa^{\zeta+2-2\g}}\int_{\k\lambda_0}^\infty \l^{\zeta+2(1-\g)}\int_{\mathcal{U}_{\l}(r_2)}|Du|^{2\g}\dx\,\frac{d\lambda}{\lambda}+\frac c{\kappa^\zeta}\int_{\k\lambda_0}^\infty \l^\zeta\int_{\mathcal{F}_{\l}(r_2)}|\tilde{F}|^2\dx \,\frac{d\lambda}{\lambda} \\
&\leq \frac c{[\zeta+2(1-\g)]\kappa^{\zeta}}\int_{B_{r_2}}|Du|^{2+\zeta}\dx+\frac c{\zeta\kappa^\zeta}\int_{B_{r_2}}|\tilde{F}|^{2+\zeta}\dx
\end{align*}
with $c,\kappa$ depending on $\data$ and $\max\{\nu,1\}$; in the last estimate we used Fubini's layer formula twice. On the other hand, recalling \eqref{bd:l0}--\eqref{def:xi0}, we have
\[
    \int_0^{\l_0}\l^\zeta \int_{\mathcal{U}^{r_1}_\l}|Du|^2\dx\,\frac{d\l}{\lambda} \leq\frac{\l_0^{\zeta}}{\zeta}\int_{B_{r_2}}|Du|^2\dx\leq \frac{c}{\zeta}\Big(\frac{R_0}{r_2-r_1} \Big)^{\zeta(n+2)}|B_{R_0}|\,\Xi_0^{\zeta+2}\,.
\]
Combining these two estimates and again using Fubini's formula yields
\begin{multline*}
\int_{B_{r_1}}|Du|^{2+\zeta}\dx=\zeta    \int_0^{+\infty}\l^\zeta \int_{\mathcal{U}^{r_1}_\l}|Du|^2\dx\,\frac{d\l}{\lambda}\\
\leq\frac {c\,\zeta}{[\zeta+2(1-\g)]\kappa^{\zeta}}\int_{B_{r_2}}|Du|^{2+\zeta}\dx+\frac c{\kappa^\zeta}\int_{B_{R_0}}|\tilde{F}|^{2+\zeta}\dx+c\left(\frac{R_0}{r_2-r_1} \right)^{\zeta(n+2)}|B_{R_0}|\,\Xi_0^{\zeta+2}\,;
\end{multline*}
this inequality is valid for any $r_1,r_2$ fulfilling \eqref{radii0}. Finally we conclude by choosing $\zeta>0$ depending on $\data,\max\{\nu,1\}$ such that
\begin{equation*}
    \frac{c\,\zeta}{[\zeta+2(1-\g)]\kappa^\zeta}\leq \frac{1}{2}\,,
\end{equation*}
as we are in the position to apply the aforementioned \cite[Lemma 6.1]{giusti} to reabsorb $\int_{B_{r_2}} |Du|^{2+\zeta}\dx$; this gives, after dividing by $|B_{R_0}|$, and recalling the definition of $\Xi_0$ in \eqref{def:xi0}, the desired estimate.
\end{proof}

\section{Linearization and excess decay}\label{linearization}
In this section, we fix $B_\vrho(x_0)\Subset \Omega$ and we assume that $u$ satisfies
\begin{equation}\label{excess:small}
    \mathcal{E}_{\g}\big(u;B_\vrho(x_0)\big)\leq \e_0
\end{equation}
for some $\e_0=\e_0(\data,\omegau(\,\cdot\,))\in (0,1)$ to be determined later, where $\mathcal{E}_\g(\,\cdot\,)$ is the $\gamma$-excess defined in \eqref{def:exc}. 
We also assume that $A(\cdot,\cdot\cdot)$ fulfills the $BMO$-condition
\begin{equation}\label{hp:A}
 E_{r_A}(A;x_0)\leq \delta,
\end{equation}
for some radius $r_A\in (0,1)$, and $\delta=\delta(\data)\in (0,1)$ to be chosen later, where $E_{r}(A;x_0)$ is defined in \eqref{E_A}.

\vs

We carry out a blow-up procedure at a point $x_0 \in \Omega$ within the ball $B_\vrho(x_0)$, where $\vrho \leq \min\{r_0,r_A\}$. Here $r_A$ is the radius from \eqref{hp:A} and $r_0 = r_0(\data, \nu) \in (0, 1)$ satisfies \eqref{cond:r0}; $r_0$ will be further reduced in due course of the proof.

In this section, unlike in the previous ones, we will always make the center of the balls explicit when it is $x_0$ and we will use the notation $B_{1/4},B_1$, etc., for balls centered at the origin. 

\vs

From now on, we fix a parameter 
\begin{equation}\label{dep:gamma}
    \text{$\gamma=\gamma(s)$ such that $\max\{2s,1\}<\gamma<2$\,,}
\end{equation}
 and for $\mathcal{E}_\gamma\equiv \mathcal{E}_\gamma(u;B_\vrho(x_0))$, we set
\begin{equation}\label{def:blow}
    u_\vrho(x):=\frac{u(x_0+\vrho x)-(u)_{B_\vrho(x_0)}}{\mathcal{E}_\gamma}\,,\qquad f_\vrho(x):=\frac{\vrho^2 f(x_0+\vrho x)}{\mathcal{E}_\gamma} \,.
\end{equation}
Note that
\begin{equation}\label{med:tail0}
    {(u_\vrho)}_{B_1}=0\,, \qquad    \mint_{B_1}{|u_\vrho|}^2\dx\leq 1
\end{equation}
and, by a change of variables, we have
\begin{equation}\label{tail:rho}
    \tail_\gamma\big(u_\vrho;B_1\big)=\tail\big(u_\vrho;B_1\big)=\frac{1}{\mathcal{E}_\g}\tail\big(u-(u)_{B_\vrho(x_0)};B_\vrho(x_0)\big)\leq \vrho^{2s-\gamma}\,;
\end{equation}
the last inequality is due to the definition of $\gamma$-excess in \eqref{def:exc} and \eqref{relaz:tails}.
By setting 
\[
A_\vrho(x,z)=A\big(x_0+\vrho x, {(u)}_{B_\vrho(x_0)}+\mathcal{E}_\g z\big)
\] 
and
\[
    \Phi_\vrho(x,y,v,w,t)=\frac{\Phi\big(x_0+\vrho x,x_0+\vrho y,(u)_{B_\vrho(x_0)}+\mathcal{E}_\gamma v,\, (u)_{B_\vrho(x_0)}+\mathcal{E}_\gamma w,\, \mathcal{E}_\gamma t\big)}{\mathcal{E}_\gamma}\,,\qquad Q^\mathrm{nl}_{\vrho}\equiv Q^\mathrm{nl}_{\Phi_\vrho}\,,
\]
straightforward computations (via change of variables) show that $u_\vrho$ is a weak solution to
\begin{equation}\label{sol:blow}
    -\mathrm{div}\big(A_\vrho(\cdot,u_\vrho)Du_\vrho \big)+\vrho^{2(1-s)}Q^\mathrm{nl}_{\vrho} u_\vrho=f_\vrho\qquad\text{in $B_1$}\,.
\end{equation}
Furthermore, it is straightforward to verify that $A_\vrho(\cdot,\cdot\cdot)$ still satisfies \eqref{A:gr}, and that $\Phi_\vrho$ continues to satisfy \eqref{phi:gr}--\eqref{sym:phi}, that is
\begin{equation}\label{A:grrho}
    \langle A_\vrho(x,z)\,\xi,\xi\rangle\geq \Lambda^{-1}|\xi|^2,\qquad |A_\vrho(x,z)|\leq \Lambda\quad\text{for a.e. $x\in B_1$ and $z\in \R^N$,}
\end{equation}
and, for a.e. $x,y\in \R^n$, all $u,v\in \R^N$ and $t\in \R^N$
\begin{equation}\label{phi:grrho}
\begin{cases}
    \Phi_\vrho(x,y,u,v,t) =\Phi_\vrho(y,x,v,u,t)
    \\
    \Phi_\vrho(x,y,u,v,-t)=-\Phi_\vrho(x,y,u,v,t)
    \end{cases},
\qquad     |\Phi_\vrho(x,y,v,w,t)|\leq \nu|t|\,.
\end{equation}
In addition, if $A(\cdot,\cdot\cdot)$ satisfies \eqref{hp:A}, then a simple change of variables shows that $A_\vrho(\cdot,\cdot\cdot)$ also fulfills
\begin{equation}\label{Ar:vani}
\bigg(\mint_{B_\sigma}{\big|A_\vrho(x,z)-(A_\vrho)_{B_\sigma}(z)\big|}^2 \dx \bigg)^{1/2}
\leq \delta,
\quad \text{for all $\sigma \leq r_A/\vrho$ and $z\in\mathbb{R}^N$}\,.
\end{equation}
Moreover, by \eqref{mod:cont.u}, we have that $u\mapsto A_\vrho(x,u)$ is uniformly continuous (uniformly in $x$), with modulus of continuity given by
\begin{equation}\label{new:modcont}
    \omega_{\mathsf{u},\vrho}(t)\coloneqq \omegau(\mathcal E_\gamma t)\quad \text{for $t\geq 0$}\,.
\end{equation}

\begin{remark}\label{rem:r}
{\rm From now on, possibly reducing $r_0$, we will always assume that
\begin{equation}\label{rpiccolo}
    c(\data)r_0^{2(1-s)}\nu\leq 1\quad\text{and}\quad  r_0^{2-\gamma}\nu\leq 1\,,
\end{equation}
where $c(\data)\geq1$ is the constant appearing in \eqref{cond:r0}. This choice is possible due to the bound in \eqref{dep:gamma}. Consequently, all quantitative constants from the previous sections (in particular, those in the Caccioppoli inequality of Lemma \ref{lemma:cacc} and the gradient higher integrability of Proposition \ref{prop:highint}) applied to $u_\vrho$ will be  independent of $\vrho$. Indeed, \eqref{A:grrho} holds, and the nonlocal term $\Phi_\vrho$ in \eqref{sol:blow} satisfies \eqref{phi:grrho} while being multiplied by $\vrho^{2(1-s)}$, yielding
\begin{equation}\label{Phi:gr1}
    \vrho^{2(1-s)}|\Phi_\vrho(x,y,u,v,t)|\leq r_0^{2(1-s)}\nu|t|\leq |t| \,,\quad\text{for all $t\in \R$}.
\end{equation}
Moreover, for the rescaled equation, the threshold radius appearing in \eqref{cond:r0} effectively becomes $1$ for any choice of $\vrho\leq r_0$; this allows us to apply \eqref{high:int} directly to the function $u_\vrho$.}
\end{remark}

\subsection{Small potential regime}
In this paragraph, we further assume that we are in the so-called \textit{small potential regime}, which means that
\begin{equation}\label{small:regime}
\vrho^2\bigg(\mint_{B_\vrho(x_0)}|f|^\chi\dx\bigg)^{1/\chi}<\e^*\mathcal E_\gamma\big(u;B_\vrho(x_0)\big)
\end{equation}
for some constant $\e^*\in (0,1)$ depending on $\data$ to be determined later. In particular, by definition of $f_\vrho$ in \eqref{def:blow}, this implies
\begin{equation}\label{s:reg}
   \bigg( \mint_{B_1}|f_\vrho|^\chi\dx\bigg)^{1/\chi}\leq \e^*<1 \,.
\end{equation}
We claim that, under assumption \eqref{small:regime}, there exists $\delta_1=\delta_1(\data)$ such that the higher integrability property
\begin{align}\label{utile}
\int_{B_{1/2}}|Du_\vrho|^2\dx+\int_{B_{1/2}}|Du_\vrho|^{2+\delta_1}\dx\leq c(\data)
\end{align}
holds true if $\vrho\leq r_0$. To prove it, let us fix $y\in B_{1/2}$; since $u_\vrho$ solves \eqref{sol:blow} with \eqref{A:grrho} and \eqref{phi:grrho}--\eqref{Phi:gr1} in force, from Caccioppoli inequality \eqref{caccioppoli} with $u_0=0$ and $\nu$ replaced by $\vrho^{2(1-s)}\nu\leq 1$, we obtain the energy inequality
\begin{multline}\label{a20}
    \mint_{B_{ {1}/{8}}(y)}|Du_\vrho|^2\dx\\
    \leq  c\mint_{B_{ {1}/{4}}(y)}|u_\vrho|^2\dx+c\,\vrho^{2(1-s)}\nu\tail\big(u_\vrho;B_{ {1}/{4}}(y)\big)\mint_{B_{ {1}/{4}}(y)}|u_\vrho|\dx+c\bigg(\mint_{B_{ {1}/{4}}(y)} |f_\vrho|^\chi\dx\bigg)^{2/\chi}+c\,,
\end{multline}
with $c=c(\data)$ due to Remark \ref{rem:r}.  Next observe that, for any $\sigma\in (1/32,1/4)$ and all $y\in B_{1/2}$, we have $B_\sigma(y)\subseteq B_1$. It then follows by \eqref{tail:4} and \eqref{med:tail0}--\eqref{tail:rho} that
\begin{align*}
    \vrho^{2(1-s)}\nu\tail\big(u_\vrho;B_{\sigma}(y)\big)&\leq c(n,s)\vrho^{2(1-s)}\nu\Big[\tail\big(u_\vrho;B_1\big)+\int_{B_1}|u_\vrho|\dx\Big]\\
    &\leq c(n,s)\vrho^{2-\gamma}\nu\leq c(n,s)r_0^{2-\gamma}\nu\leq c(n,s)\,,
\end{align*}
where the last inequality due to \eqref{rpiccolo}; a similar estimate holds if we replace $u_\vrho$ with $u_\vrho-{(u_\vrho)}_{B_{1/8}(y)}$, using this time \eqref{tail:3} and \eqref{tail.triangle}--\eqref{tail.constant}. Therefore, using H\"older's inequality, \eqref{s:reg}, \eqref{med:tail0} and the last estimate, from \eqref{a20} we infer
\begin{equation*}
    \mint_{B_{1/8}(y)}|Du_\vrho|^2dx\leq c(\data)\,.
\end{equation*}
Then, using \eqref{high:int} (with $\nu$ replaced by $\vrho^{2(1-s)}\nu$ by \eqref{Phi:gr1}) and \eqref{pot:est}, if we set $\delta_1=\min\{\zeta,\epsilon\}$ (which only depends on $\data$ thanks to Remark \ref{rem:r}), we have
\begin{multline*}
    \bigg(\mint_{B_{1/32}(y)}|Du_\vrho|^{2+\delta_1}\dx\bigg)^{1/(2+\delta_1)}\leq c\bigg(\mint_{B_{1/8}(y)}|Du_\vrho|^2\dx\bigg)^{1/2} +c+c\bigg(\mint_{B_1}|f_\vrho|^\chi dx\bigg)^{1/\chi}\\
+c\,\vrho^{2(1-s)}\nu\tail\big(u_\vrho-(u_\vrho)_{B_{1/8}(y)};B_{1/8}(y)\big)\,,
\end{multline*}
which, estimating as before, yields
\begin{equation}\label{a21}
    \mint_{B_{1/32}(y)}|Du_\vrho|^{2+\delta_1}\dx\leq c(\data)\,.
\end{equation}
By the arbitrariness of $y\in B_{1/2}$, \eqref{utile}) can be obtained from \eqref{a21} and a standard covering argument.
\vs

 It is useful to introduce the following notation: for $\Phi$ a kernel as in \eqref{phi:gr}, and $\varphi\in C^\infty_c(\R^n;\R^N)$, we write 
\begin{equation}\label{bracket}
{\llangle Q_\Phi^\mathrm{nl}u, \varphi\rrangle}_{\mathrm{nl}}:=\int_{\R^n}\int_{\R^n}\Phi(x,y,u(x),u(y), u(x)-u(y))\cdot (\vphi(x)-\vphi(y))\,\frac{dx\dy}{|x-y|^{n+2s}}\,. 
\end{equation}
In view of the quantitative approximation lemma, the next lemma shows the core of our paper: namely, it shows why the nonlocal term can be treated as a perturbative term, and it can be made arbitrarily small by taking $\vrho$ suitably small.
\begin{lemma}\label{lemma:smnloc}
There exists a constant $c$, depending only on $\data$, such that if $\gamma=\gamma(s)$ is given by \eqref{dep:gamma}, and \eqref{s:reg} is in force, then
    \begin{equation}\label{a23}
        \left|\vrho^{2(1-s)}{\llangle Q^\mathrm{nl}_{\vrho} u_\vrho,\vphi\rrangle}_{\mathrm{nl}} \right|\leq c\,\nu\vrho^{2-\gamma}{\|D\vphi\|}_{L^\infty(B_{1/4})}\,,
    \end{equation}
    for all functions $\vphi\in C^\infty_c(B_{1/4};\R^N)$.
\end{lemma}

\begin{proof}
Let $\vphi$ be as in the statement: by the symmetry property in \eqref{sym:phi},
\begin{multline*}
{\llangle Q^\mathrm{nl}_{\vrho} u_\vrho,\vphi\rrangle}_{\mathrm{nl}}=\,\int_{B_{1/4}}\int_{B_{1/4}} \Phi_{\vrho}\big(x,y,u_\vrho(x),u_\vrho(y),u_\vrho(x)-u_\vrho(y)\big)\cdot \big( \vphi(x)-\vphi(y)\big)\frac{dx\dy}{|x-y|^{n+2s}}\\
+2\int_{\R^n\setminus B_{1/4}}\int_{B_{1/4}}\Phi_{\vrho}\big(x,y,u_\vrho(x),u_\vrho(y),u_\vrho(x)-u_\vrho(y)\big)\cdot \vphi(x)\,\frac{dx\dy}{|x-y|^{n+2s}}\coloneqq (I)+(II)\,.
\end{multline*}
Using \eqref{phi:grrho}, the triangle inequality and \eqref{frac:emb}, we have
\[
    |(I)| \leq \nu{[u_\vrho]}_{W^{s,2}(B_{1/4})}{[\vphi]}_{W^{s,2}(B_{1/4})}\leq c(n,s)\nu{\|Du_\vrho\|}_{L^2(B_{1/4})}\,{\|D\vphi\|}_{L^2(B_{1/4})}\leq  c\,\nu{\|D\vphi\|}_{L^\infty(B_{1/4})}\,,
\]
where the last inequality follows from \eqref{utile} and the constant $c$ depends only on $\data$. Next we observe that for $x\in B_{1/4}$ and $y\in \R^n\setminus B_{1/2}$, we have $2|x-y|\geq |y|$; by using this information, the triangle inequality and again \eqref{phi:grrho}, we get, for a constant $c$ ultimately depending only on $\data$,
\begin{align*}
    |(II)|&\leq c\,\nu\bigg[\int_{\R^n\setminus B_{1/2}}\frac{dy}{|y|^{n+2s}}\int_{B_{1/4}}|u_\vrho(x)|\,|\vphi(x)|\dx+\int_{\R^n\setminus B_{1/2}}\frac{|u_\vrho(y)|}{|y|^{n+2s}}\dy\int_{B_{1/4}}|\vphi(x)|\dx\bigg]\\
    &\leq c\,\nu\bigg[\Big(\int_{B_1}|u_\vrho|^2\dx \Big)^{1/2}\Big(\int_{B_{1/4}}|\vphi|^2\dx \Big)^{1/2}+\int_{\R^n\setminus B_{1/2}}\frac{|u_\vrho(y)|}{|y|^{n+2s}}\dy\int_{B_{1/4}}|\vphi(x)|\dx\bigg]\\
        &\leq c\,\nu\bigg[\Big(\int_{B_1}|u_\vrho|^2\dx \Big)^{1/2}+\int_{\R^n\setminus B_{1/2}}\frac{|u_\vrho(y)|}{|y|^{n+2s}}\dy\bigg]{\|D\vphi\|}_{L^\infty(B_{1/4})}\,.
\end{align*}
We conclude the proof of \eqref{a23} by estimating the first term between parentheses by \eqref{med:tail0} and the second one as follows, using \eqref{tail:rho}:
\begin{align*}
    \int_{\R^n\setminus B_{1/2}}\frac{|u_\vrho(y)|}{|y|^{n+2s}}\dy & =\tail\big(u_\vrho;B_1\big)+\int_{B_1\setminus B_{1/2}}\frac{|u_\vrho(y)|}{|y|^{n+2s}}\dy
    \\
    &\leq \vrho^{2s-\gamma}+2^{n+2s}\,\int_{B_1}|u_\vrho(y)|dy\leq c(n)\big(\vrho^{2s-\gamma}+1\big)\,.
\end{align*}
Multiplying the  estimates above by $\vrho^{2(1-s)}$ and recalling that $\vrho^{2s-\gamma}>1$, $\vrho^{2(1-s)}<\vrho^{2-\gamma}$ by \eqref{dep:gamma} yields \eqref{a23}.
\end{proof}

\subsection{Approximate harmonicity}\label{par:apprharm}
With Lemma \ref{lemma:smnloc} at our disposal , we now want to apply Lemma \ref{lem:harapprox} with
\begin{equation}\label{def:Arhorho}
    \mathsf{A}\equiv \bar{A_\vrho} \coloneqq (A_{\vrho})_{B_1}(0),
\end{equation}
 where $(A_{\vrho})_{B_\sigma}(z)$ was defined in \eqref{A:mean}; $\bar{A_\vrho}$ clearly satisfies \eqref{Legendre.A} by \eqref{A:grrho}. To this end, we first prove the following
\begin{proposition}[Quantified approximate harmonicity]\label{prop:a}
Let $u_\vrho, A_\vrho $ be given by \eqref{def:blow} and \eqref{def:Arhorho}, respectively. Assume that  \eqref{excess:small}, \eqref{dep:gamma} and \eqref{small:regime} are in force. Then we have 
\begin{equation}\label{b:3}
        \bigg|  \int_{B_{1/4}}\big\langle \bar{A_\vrho} Du_\vrho, D\vphi\big\rangle\dx \bigg| \leq c\,s_1(\vrho, \e^*,\e_0,\delta){\|D\vphi\|}_{L^\infty(B_{1/4})}
\end{equation}
for all $\vphi\in C^\infty_c(B_{1/4};\R^N)$, where $c=c(\data)$ and
\begin{equation}\label{ess}
s_1(\vrho, \e^*,\e_0,\delta)= \nu\vrho^{2-\gamma}+\e^*+ [\omegau(\e_0)]^{1/2}+\delta\,.
\end{equation}
\end{proposition}
\begin{proof}
   Given $\vphi\in C^\infty_c(B_{1/4};\R^N)$, we have for $X\in B_{1/4}$
\[
\big\langle \bar{A_\vrho} Du_\vrho(x), D\vphi(x)\big\rangle= \big\langle A_{\vrho}(x,u_\vrho(x))Du_\vrho(x), D\vphi(x)\big\rangle +\big\langle\big(\bar{A_\vrho} -A_{\vrho}(x,u_\vrho(x))\big)Du_\vrho(x), D\vphi(x)\big\rangle\,,
\]
so that, integrating over $B_{1/4}$ and using the weak formulation \eqref{sol:blow}, we obtain
\begin{multline*}
\bigg| \int_{B_{1/4}}\big\langle \bar{A_\vrho}Du_\vrho, D\vphi\big\rangle\dx\bigg| \leq    \Big|\vrho^{2(1-s)}{\llangle Q^\mathrm{nl}_{\vrho} u_\vrho,\vphi\rrangle}_{\mathrm{nl}}\Big|+\bigg|\int_{B_{1/4}}f_\vrho\cdot\vphi\dx\bigg|\\
+\int_{B_{1/4}}\big|\bar{A_\vrho}-A_{\vrho}(x,u_\vrho)\big||Du_\vrho| |D\vphi|\dx=(I)+(II)+(III)\,.
\end{multline*}
The term $(I)$ is estimated using Lemma \ref{lemma:smnloc}. Moreover, by H\"older's inequality, \eqref{s:reg}, and since ${\|\vphi\|}_{L^\infty(B_{1/4})}\leq {\|D\vphi\|}_{L^\infty(B_{1/4})}$ being $\varphi$ compactly supported, we have
\begin{equation}
(II)\leq {\|f_\vrho\|}_{L^\chi(B_1)}\,{\|\vphi\|}_{L^\infty(B_{1/4})}\leq c(n)\,\e^*{\|D\vphi\|}_{L^\infty(B_{1/4})}\,;
\end{equation}
for the last term we use the triangle inequality and split
\begin{equation*}
(III)\leq \int_{B_{1/4}}\big|\bar{A_\vrho}-A_\vrho(x,0)\big||Du_\vrho||D\vphi|\dx   +\int_{B_{1/4}}\big|A_\vrho(x,0)-A_{\vrho}(x,u_\vrho)\big||Du_\vrho||D\vphi|\dx\eqqcolon (IV)+(V)\,.
\end{equation*}
By  H\"older's inequality, \eqref{Ar:vani} and \eqref{def:Arhorho} (recalling that $1/4\leq r_A/\vrho$ as $\vrho\leq r_0(\data,\nu,r_A)\leq r_A$), and by \eqref{utile}  we infer
\begin{align*}
    (IV)& \leq \bigg(\int_{B_{1/4}}{|\bar{A_\vrho}-A_\vrho(x,0)|}^2\dx\bigg)^{1/2}{\|Du_\vrho\|}_{L^2(B_{1/4})}{\|D\vphi\|}_{L^\infty(B_{1/4})}\notag
    \\
    &\leq c\,\bigg(\int_{B_{1/4}}{|\bar{A_\vrho}-A_\vrho(x,0)|}^2\dx\bigg)^{1/2}{\|D\vphi\|}_{L^\infty(B_{1/4})}\leq c\,\d\,{\|D\vphi\|}_{L^\infty(B_{1/4})}\,,
\end{align*}
with $c=c(\data)$. Then, by the uniform continuity of $A_\vrho$ \eqref{new:modcont}, H\"older's and Jensen's inequalities, \eqref{utile}, the boundedness of $\omega_{\mathsf{u},\vrho}(\,\cdot\,)\leq 1$, \eqref{med:tail0} and then \eqref{excess:small}, we deduce
\begin{align*}
(V)&\leq c\,{\|D\vphi\|}_{L^\infty(B_{1/4})}\int_{B_{1/4}}\omega_{\mathsf{u},\vrho}\big(|u_\vrho|\big)|Du_\vrho|\dx\notag\\
& \leq c\bigg(\mint_{B_1}\omega_{\mathsf{u},\vrho}(|u_\vrho|)^2\dx\bigg)^{1/2}\,{\|Du_\vrho\|}_{L^2(B_{1/4})}\,{\|D\vphi\|}_{L^\infty(B_{1/4})}\notag\\
&\leq c\bigg[\omega_{\mathsf{u},\vrho}\bigg(\mint_{B_1} |u_\vrho|\dy\bigg)\bigg]^{1/2}{\|D\vphi\|}_{L^\infty(B_{1/4})}\notag\\
&\leq c\,\Big[\omegau\Big(\mathcal{E}_{\g}\big(u;B_\vrho(x_0)\big)\Big)\Big]^{1/2}\,{\|D\vphi\|}_{L^\infty(B_{1/4})}\leq c\big[\omegau(\e_0)\big]^{1/2}{\|D\vphi\|}_{L^\infty(B_{1/4})}\,,
\end{align*}
where $c=c(\data)$. Combining all these estimates, we get the desired estimate.
\end{proof}

By the previous proposition, we are ready to apply the approximate harmonicity lemma \ref{lem:harapprox}: owing to \eqref{utile} and \eqref{b:3}, there exist a small exponent $\vartheta=\vartheta(\data)\in(0,1]$ and $c=c(\data)$ such that
\begin{equation}\label{small:vr}
    \int_{B_{1/4}}|Du_\vrho-Dv_\vrho|^2\dx\leq c\big[s_1(\vrho, \e^*,\e_0,\delta)\big]^{\vartheta}\,,
\end{equation}
for all $0<\vrho<r_0(\data,\nu,r_A)$, where $v_\vrho\in u_\vrho+W_0^{1,2}(B_{1/4};\R^N)$ is the unique solution to 
\begin{equation}\label{b:6}
\begin{cases}
    -\mathrm{div}\big(\bar{A_\vrho}Dv_\vrho \big)=0\quad & \text{in $B_{1/4}$}
    \\
    v_\vrho=u_\vrho\quad &\text{on $\partial B_{1/4}$}
    \end{cases}\,;
\end{equation}
here $\bar{A_\vrho}$ is given by \eqref{def:Arhorho}.
By standard elliptic regularity theory (see \cite[Remark 2.2, Chapter III]{giaq}) and \eqref{utile}, for all $k\in \N$ we have
\begin{equation}\label{utile1}
    {\|D^kv_\vrho\|}_{L^\infty(B_{1/8})}\leq c(k,\Lambda){\|Dv_\vrho\|}_{L^2(B_{1/4})}\leq c(k,\Lambda){\|Du_\vrho\|}_{L^2(B_{1/4})}\leq c(k,\data)\,;
\end{equation}
the second inequality comes from testing the weak formulation of \eqref{b:6}$_1$  with $u_\vrho-v_\vrho$, and performing standard computations, using the ellipticity and growth assumptions \eqref{A:grrho}.

\vs

\noindent Now let $\tau\in (0,1/64)$ to be fixed later, and estimate $\mathcal E_{\g}(u_\vrho;B_\tau)$ by  considering separately its two components. By  \eqref{small:vr} and \eqref{utile1}, we find
\begin{align}\label{av:1}
\bigg(\mint_{B_\tau} |Du_\vrho|^2\dx\bigg)^{1/2}    &\leq \bigg(\mint_{B_\tau} |Du_\vrho-Dv_\vrho|^2\dx\bigg)^{1/2}+\,\bigg(\mint_{B_\tau} |Dv_\vrho|^2\dx\bigg)^{1/2}\notag\\
    &\leq c(n)\tau^{-n/2}\bigg(\mint_{B_{1/4}} |Du_\vrho-Dv_\vrho|^2\dx\bigg)^{1/2}+{\|Dv_{\vrho}\|}_{L^\infty(B_{1/8})}\notag
    \\
    &\leq c\,\tau^{-n/2}\,\big[s_1(\vrho, \e^*,\e_0,\delta)\big]^{\vartheta/2}+c\,,
\end{align}
with $c=c(\data)$. Moreover, by Poincar\'e's inequality and a change of variables,
\[
\frac{1}{\mathcal E_{\g}(u;B_\vrho(x_0))}\bigg(\mint_{B_{\tau\vrho}(x_0)} {\big|u-{(u)}_{B_{\tau\vrho}(x_0)}\big|}^2\dx\bigg)^{1/2}=\bigg(\mint_{B_\tau} {\big|u_\vrho-{(u_\vrho)}_{B_\tau}\big|}^2\dx\bigg)^{1/2}\leq c\,\tau\bigg(\mint_{B_\tau} |Du_\vrho|^2\dx\bigg)^{1/2}\,,
\]
$c=c(n)$, which, coupled with \eqref{av:1} implies
\begin{equation}\label{av:2}
 \bigg(\mint_{B_{\tau\vrho}(x_0)} {\big|u-{(u)}_{B_{\tau\vrho}(x_0)}\big|}^2\dx\bigg)^{1/2}\leq c\,\tau\Big\{\tau^{-n/2}\,\big[s_1(\vrho, \e^*,\e_0,\delta)\big]^{\vartheta/2}+1\Big\}\,\mathcal E_{\g}\big(u;B_\vrho(x_0)\big)\,,
\end{equation}
again with $c$ depending on $\data$.
Next, using \eqref{tail:1} with $w=u_\vrho$ (remember that ${(u_\vrho)}_{B_1}=0$), \eqref{med:tail0}--\eqref{tail:rho}, we may estimate the $\gamma$-tail term as follows:
\begin{align}\label{it:tail1}
\tail_\gamma\big(u_\vrho-(u_\vrho)_{B_\tau};B_\tau\big) &\leq  c\,\tau^\g\tail_\g(u_\vrho;B_1)+c\,\tau^\g\mint_{B_1}|u_\vrho|\dx+c\int_{\tau}^1\frac{\tau^\g}{\mu^{2s}}\mint_{B_\mu}\big|u_\vrho-{(u_\vrho)}_{B_\mu}\big|\dx\,\frac{d\mu}{\mu}\notag
    \\
    &\leq  c\,\tau^\g\vrho^{2s-\g}+c\,\tau^\g\bigg[1+\int_{\tau}^1\mu^{-2s}\mint_{B_\mu}\big|u_\vrho-{(u_\vrho)}_{B_\mu}\big|\dx\,\frac{d\mu}{\mu}\bigg]\,.
\end{align}
We estimate separately the last integral, thanks to the triangle, H\"older's and Poincar\'e's inequalities, \eqref{med:tail0}, \eqref{small:vr} and \eqref{utile1}:
\begin{align}\label{it:tail2}
\int_{\tau}^1\mu^{-2s}\mint_{B_\mu}\big|u_\vrho-{(u_\vrho)}_{B_\mu}\big|\dx\,\frac{d\mu}{\mu}&\leq \int_{\tau}^{1/16}\mu^{-2s}\mint_{B_\mu}\big|u_\vrho-{(u_\vrho)}_{B_\mu}\big|\dx\,\frac{d\mu}{\mu}+2\int_{1/16}^1\mu^{-2s}\mint_{B_\mu}|u_\vrho|\dx\,\frac{d\mu}{\mu}\notag\\
&\leq c(n)\int_{\tau}^{1/16}\frac{1}{\mu^{2s-1}}\bigg(\mint_{B_\mu}|Du_\vrho-Dv_\vrho|^2\dx\bigg)^{1/2}\,\frac{d\mu}{\mu}\notag\\
&\qquad\qquad\qquad+c(n)\int_{\tau}^{1/16}\frac{1}{\mu^{2s-1}}\bigg(\mint_{B_\mu}|Dv_\vrho|^2\dx\bigg)^{1/2}\,\frac{d\mu}{\mu}\notag\\
&\pushright{+\,2\int_{1/16}^1\mu^{-2s}\bigg(\mint_{B_\mu}|u_\vrho|^2\dx\bigg)^{1/2}\,\frac{d\mu}{\mu}}\notag\\
    &\leq c\,\big[s_1(\vrho, \e^*,\e_0,\delta)\big]^{\vartheta/2} \int_\tau^{1/16}\mu^{1-2s-n/2}\,\frac{d\mu}{\mu}\notag\\
    &\pushright{+\,c{\|Dv_\vrho\|}_{L^\infty({B_{1/8}})}\,\int_{\tau}^{1/16}\mu^{1-2s}\,\frac{d\mu}{\mu}+c\int_{1/16}^1\mu^{-2s-n/2}\,\frac{d\mu}{\mu}}\notag\\
    &\leq  c(\data)\Big\{\tau^{1-2s-n/2}[s_1(\vrho, \e^*,\e_0,\delta)]^{\vartheta/2}+\tau^{1-2s}+1\Big\}\,.
\end{align}
On the other hand, using \eqref{relaz:tails} and a change of variables, we have
\begin{align}\label{it:tail3}
     \tail_\gamma\big(u_\vrho-(u_\vrho)_{B_\tau};B_\tau\big)=\tau^{\gamma-2s} \tail\big(u_\vrho-(u_\vrho)_{B_\tau};B_\tau\big)&=\frac{\tau^{\gamma-2s}}{\mathcal E_{\g}(u;B_\vrho(x_0))} \tail\big(u-(u)_{B_{\tau\vrho}(x_0)};B_{\tau\vrho}(x_0)\big)\notag\\
     &=\frac{\vrho^{2s-\g}}{\mathcal E_{\g}(u;B_\vrho(x_0))}\tail_\g\big(u-(u)_{B_{\tau\vrho}(x_0)};B_{\tau\vrho}(x_0)\big)\,.
\end{align}
Plugging \eqref{it:tail2} into \eqref{it:tail1} and in turn merging with \eqref{it:tail3} yields
\begin{align}\label{it:tail4}
    \tail_\g\big(u-(u)_{B_{\tau\vrho}(x_0)};B_{\tau\vrho}(x_0)\big)&=\frac{\mathcal E_{\g}(u;B_\vrho(x_0))}{\vrho^{2s-\g}} \tail_\gamma\big(u_\vrho-(u_\vrho)_{B_\tau};B_\tau\big)\notag\\
&\leq c\,\frac{\mathcal E_{\g}(u;B_\vrho(x_0))}{\vrho^{2s-\g}} \bigg[\tau^\g\vrho^{2s-\g}+\tau^\g\Big\{\tau^{1-2s-n/2}[s_1(\vrho, \e^*,\e_0,\delta)]^{\vartheta/2}+\tau^{1-2s}+1\Big\}\bigg]\notag\\    
&= c\,\tau^\gamma \Big\{1+\tau^{1-2s-n/2}\,[s_1(\vrho, \e^*,\e_0,\delta)]^{\vartheta/2}\,\vrho^{\g-2s}+ \tau^{1-2s}\vrho^{\g-2s}\Big\}\,\mathcal{E}_{\g}\big(u;B_\vrho(x_0)\big)\,,
\end{align}
recalling that $1\leq \vrho^{\g-2s}$ from \eqref{dep:gamma}. Finally, combining \eqref{av:2} and \eqref{it:tail4}, and using that $\tau^\gamma\leq \tau$ by \eqref{dep:gamma} and $\tau\leq 1$, we arrive at
\begin{equation}\label{avuuu}
\mathcal E_\gamma(u;B_{\tau\vrho}(x_0))\leq c\,\Big\{\tau^{1-n/2}\,[s_1(\vrho, \e^*,\e_0,\delta)]^{\vartheta/2} \big(1+\tau^{\g-2s}\vrho^{\g-2s}\big)+\tau+ \tau^{\g+1-2s}\vrho^{\g-2s} \Big\}\,\mathcal{E}_{\g}\big(u;B_\vrho(x_0)\big)\,,
\end{equation}
where $c=c(\data)$.
As a consequence of the above estimate, we can prove the following
\begin{proposition}[Excess decay in the small potential regime]\label{dec:small}
Let $u\in W^{1,2}_{\rm loc}(\Omega;\R^N)\cap L^1_{2s}$ be a weak solution to \eqref{eq1} under assumptions \eqref{A:gr} to \eqref{sym:phi}, let $f$ satisfy \eqref{ass:f} and let $B_\vrho(x_0)\Subset\Omega$. Then for every $\beta\in (0,1)$, there exist small constants
\begin{equation}\label{forevery:constants}
    \e^*=\e^*(\data,\beta),\quad \delta_0=\delta_0(\data,\beta),\quad \e_0=\e_0\big(\data,\beta,\omegau(\,\cdot\,)\big),\quad r_1=r_1(\data,\beta,\nu)\in(0,1),\quad \tau=\tau(\data, \beta)
\end{equation}
  such that if 
\begin{itemize}
\item    $A(\cdot,\cdot\cdot)$ satisfies \eqref{mod:cont.u} and the  $BMO$-condition \eqref{hp:A} for  some $\delta\leq \delta_0$ and radius $r_A>0$,
\item  the ball has radius $\vrho\leq \min\{r_0,r_1,r_A\}$, with $r_0$ satisfying \eqref{cond:r0} and \eqref{rpiccolo},
\item the excess is small at the initial scale, that is \eqref{excess:small} holds,
\item the system is in the small potential regime, that is \eqref{small:regime} holds; 
\end{itemize}
then there holds the $\gamma$-excess decay estimate
\begin{equation}\label{impr:excess}
    \mathcal E_{\g}\big(u;B_{\tau\vrho}(x_0)\big)\leq \tau^\beta\mathcal E_{\g}\big(u;B_{\vrho}(x_0)\big)\leq \frac12\mathcal E_{\g}\big(u;B_{\vrho}(x_0)\big)\,.
\end{equation}
\end{proposition}

\begin{proof}
We start from \eqref{avuuu}: factoring out $\tau^\beta$ and grouping the terms appropriately, we obtain
\begin{equation}\label{so:0}
         \mathcal E_{\g}\big(u;B_{\tau\vrho}(x_0)\big)\leq \tau^\beta\Big\{\bar c\,\tau^{1-\beta}\Big(\tau^{-n/2}[s_1(\vrho, \e^*,\e_0,\delta)]^{\vartheta/2}+1\Big)\Big(1+\tau^{\g-2s}\vrho^{\g-2s}\Big)\Big\}\,\mathcal E_{\g}\big(u;B_\vrho(x_0)\big) 
\end{equation}
with $\bar c$ depending on $\data$. 
We now fix $\tau=\tau(\data,\beta)>0$ so small that
\begin{equation}\label{so:1}
\tau^{1-\beta}\leq  \frac1{4 \,\bar c}\qquad \tau^\beta\leq \frac{1}{2},\qquad \tau\leq \frac{1}{64}
\end{equation}
then, given the dependence $\gamma=\gamma(s)$ and $\vartheta=\vartheta(\data)\in (0,1)$, we may choose a small radius $r_1=r_1(\data,\beta,\nu)\in (0,1)$ so that
\begin{equation}\label{so:11}
\tau^{\g-2s}r_1^{\g-2s}+\tau^{-n/2}\,\big(\nu r_1^{2-\gamma}\big)^{\vartheta/2}\leq \frac14
 \end{equation}
 then we choose $\e_0=\e_0(\data,\beta,\omegau(\,\cdot\,))$ small enough such that
\begin{equation}\label{chiamoomegau}
\tau^{-n/2}\,\omegau(\e_0)^{\vartheta/4}\leq \frac{1}{4}\,,
\end{equation}
and finally we pick small parameters $\e^*,\delta_0$ as in \eqref{forevery:constants} fulfilling
\begin{equation}\label{so:2}
\tau^{-n/2}\,\big[\e^*+\delta_0\big]^{\vartheta/2}\leq \frac1{2}\,.
\end{equation}
Combining the content of  \eqref{so:11} to \eqref{so:2} into \eqref{ess} yields, by sub-additivity,
\[
\tau^{-n/2}\big[s_1(\vrho,\e^*,\e_0,\delta) \big]^{\vartheta/2}\leq 1
\]
for all $\vrho\leq r_1$ and $\delta\leq \delta_0$. Inserting this estimate and \eqref{so:1}$_1$ into \eqref{so:0}, yields \eqref{impr:excess}.
\end{proof}

\subsection{Large potential regime} 

Let us now assume that \eqref{small:regime} is violated, that is
\begin{equation}\label{larg:pot}
    \vrho^2\bigg(\mint_{B_\vrho(x_0)}|f|^\chi\dx\bigg)^{1/\chi}>\e^*\mathcal E_{\g}\big(u;B_\vrho(x_0)\big) 
\end{equation}
holds; for $\beta\in(0,1)$ now $\e^*$ is the fixed constant, depending on $\data$ and $\beta$, given by Proposition \ref{dec:small}. As a simple application of \eqref{ut:exc}, we infer the following
\begin{proposition}[Excess decay in the large potential regime]\label{dec:large}
    Let $u\in W^{1,2}_{\rm loc}(\Omega;\R^N)\cap L^1_{2s}$ be as in Proposition \ref{dec:small}, and assume that $B_{\vrho}(x_0)\Subset \Omega$, $\vrho\leq r_0$, $r_0$ as in the beginning of Section \ref{linearization} and satisfying \eqref{rpiccolo}, are such that \eqref{larg:pot} is in force. Then, for every $\tau\in (0,1]$, we have
\begin{equation}\label{b:12}
\mathcal E_{\g}\big(u;B_{\tau \vrho}(x_0)\big) \leq \frac{c(n,s)\tau^{-n/2}}{\e^*}\vrho^2\bigg(\mint_{B_\vrho(x_0)}|f|^\chi\dx\bigg)^{1/\chi}\,.
\end{equation}
\end{proposition}

\subsection{Conclusion}
We now combine Propositions \ref{dec:small} and \ref{dec:large} and implement a standard iteration procedure to deduce the following conditional excess decay estimate:

\begin{proposition}[Excess decay]\label{dec:one}
Let $u\in W^{1,2}_{\rm loc}(\Omega;\R^N)\cap L^1_{2s}$ be a weak solution to \eqref{eq1} under assumptions \eqref{A:gr} to \eqref{sym:phi} and let $f$ satisfy \eqref{ass:f}. Fix $\beta\in(0,1)$, $\gamma$ as in \eqref{dep:gamma}, and let $B_\vrho(x_0)\Subset\Omega$.  There exists a smallness parameter  $\e_s=\e_s(\data,\omegau(\,\cdot\,),\beta)$, a small radius $r_1=r_1(\data,\nu,\beta)\in(0,1)$, a small constant $\delta_0=\delta_0\,(\data,\beta)$ and a large constant $c=c(\data,\beta)\geq 1$ such that if
\begin{itemize}
\item $A(\cdot,\cdot\cdot)$ satisfies \eqref{mod:cont.u} and the $BMO$ condition \eqref{hp:A} for some $\delta\leq \delta_0$ and $r_A>0$;
\item the ball has radius $\vrho\leq \min\{r_0,r_1,r_A\}$, with $r_0$ satisfying \eqref{cond:r0} and \eqref{rpiccolo}; 
\item the following smallness condition at the initial scale holds:
\begin{equation}\label{eps-small}
\mathcal E_{\g}\big(u;B_\vrho(x_0)\big)+\sup_{t\leq \vrho}\,t^2\bigg(\mint_{B_t(x_0)}|f|^\chi\dx\bigg)^{1/\chi}\leq \e_s\,;
\end{equation}
\end{itemize}
then there exists $\tau=\tau(\data,\beta)$ such that $\tau^\beta\leq 1/2$ and
\begin{equation}\label{alla:fineuso}
    \mathcal E_{\g}\big(u;B_{\tau^{k+1}\vrho}(x_0)\big)\leq \tau^\beta \mathcal E_{\g}\big(u;B_{{\tau^k\vrho}}(x_0)\big)+c\big({\tau^k\vrho}\big)^2\bigg(\mint_{B_{\tau^k\vrho}(x_0)}|f|^\chi\dx\bigg)^{1/\chi}
\end{equation}
for all $k\in \N_0$, and
\begin{equation}\label{impr:excess.general}
    \mathcal E_{\g}\big(u;B_{\sigma}(x_0)\big)\leq c\Big(\frac{\sigma}{\rho}\Big)^\beta \mathcal E_{\g}\big(u;B_{\rho}(x_0)\big)+c\sup_{\sigma\leq t\leq\rho}\, t^2\bigg(\mint_{B_t(x_0)}|f|^\chi\dx\bigg)^{1/\chi}
\end{equation}
for every $0<\sigma\leq\rho\leq \vrho$. Moreover
\begin{equation}\label{unif:EXCS}
\sup_{t\leq \vrho}\,\mathcal E_{\g}\big(u;B_t(x_0)\big)\leq c(\data,\beta)\,\e_s\,.
\end{equation}
\end{proposition}

\begin{proof}
 For $\beta\in (0,1)$ as in the statement, we take $\e^*(\data,\beta),\e_0(\data,\beta, \omegau(\,\cdot\,)),\tau(\data,\beta),r_1(\data,\beta,\nu)$ the corresponding constants from Proposition \ref{dec:small} and we fix $\e_s$ as
\[
\e_s=\e_s\big(\data,\omegau(\,\cdot\,),\beta\big)=\frac{\e_0}4\Big[\frac{c(n,s)\tau^{-n/2}}{\e^*}\Big]^{-1}\,,
\]
where $c(n,s)$ is the constant in \eqref{b:12}; notice that $\e_s\leq \e_0/4$. Now, for $\vrho\leq \min\{r_0,r_1,r_A\}$, combining the estimates from Propositions \ref{dec:small} and \ref{dec:large}, we get
\begin{equation}\label{exc.decay.onescale}
   \mathcal E_{\g}\big(u;B_{\tau\vrho}(x_0)\big)\leq \tau^\beta \mathcal E_{\g}\big(u;B_{\vrho}(x_0)\big)+\frac{c\,\tau^{-n/2}}{\e^*}\,\vrho^2\bigg(\mint_{B_\vrho(x_0)}|f|^\chi\dx\bigg)^{1/\chi} 
\end{equation}
for $c$ depending on the $\data$. Using \eqref{eps-small}, we can estimate the right-hand side as
\[
\tau^\beta \mathcal E_{\g}\big(u;B_\vrho(x_0)\big)+\frac{c\,\tau^{-n/2}}{\e^*}\vrho^2\bigg(\mint_{B_\vrho(x_0)}|f|^\chi\dx\bigg)^{1/\chi}\leq \Big[\tau^\beta+\frac{c\,\tau^{-n/2}}{\e^*}\Big]\e_s \leq \frac{\e_0}4+\frac{\e_0}4\leq \frac{\e_0}2\,,
\]
we are in a position to use again the decay estimate \eqref{exc.decay.onescale} with $\tau \vrho$ replacing $\vrho$, as \eqref{excess:small} holds in $B_{\tau\vrho}(x_0)$ and, repeating the argument, we get that for every $k\in\N_0$,
\begin{equation*}
\begin{cases}
\displaystyle{ \mathcal E_{\g}\big(u;B_{{\tau^k\vrho}}(x_0)\big)\leq \frac{\e_0}2}\\[2mm]
\displaystyle{\mathcal E_{\g}\big(u;B_{\tau^{k+1}\vrho}(x_0)\big)\leq \tau^\beta \mathcal E_{\g}\big(u;B_{{\tau^k\vrho}}(x_0)\big)+c\big({\tau^k\vrho}\big)^2\bigg(\mint_{B_{\tau^k\vrho}(x_0)}|f|^\chi\dx\bigg)^{1/\chi}}
\end{cases},
\end{equation*}
so that  \eqref{alla:fineuso} is proved; moreover, by iterating this estimate we deduce that, for every $i,j \in\N_0$ with $j\geq i$,
\begin{align*}
\mathcal E_{\g}\big(u;B_{\tau^{j}\vrho}(x_0)\big)&\leq \tau^{\beta(j-i)} \mathcal E_{\g}\big(u;B_{{\tau^i \vrho}}(x_0)\big)+c\sum_{k=i}^{j-1}\tau^{\beta(j-1-k)}\big({\tau^k\vrho}\big)^2\bigg(\mint_{B_{\tau^k\vrho}(x_0)}|f|^\chi\dx\bigg)^{1/\chi} \\
&\leq \tau^{\beta(j-i)} \mathcal E_{\g}\big(u;B_{\tau^i\vrho}(x_0)\big)+\frac{c}{1-\tau^\beta}\sup_{\tau^{j-1}\vrho\leq t \leq\tau^i\vrho}\,t^2\bigg(\mint_{B_t(x_0)}|f|^\chi\dx\bigg)^{1/\chi}
\end{align*}
(if $i=j$ we agree that the summation and the supremum are both zero). Now the conclusion is standard: to prove \eqref{impr:excess.general}, we first consider when $\sigma<\tau \rho$, and in this case there are $j,k\in\mathbb N$, with $j\geq k+1$, such that $\tau^{j+1}\vrho<\sigma\leq \tau^j\vrho\leq \tau^{k+1}\vrho<\rho\leq \tau^k\vrho$: at this point by \eqref{ut:exc}, using the previous estimate with $i=k+1$, with the same agreement on the supremum,
\begin{align*}
\mathcal E_{\g}\big(u;B_\sigma(x_0)\big)&\leq c(n,s)\tau^{-\frac n2}\mathcal E_{\g}\big(u;B_{\tau^j\vrho}(x_0)\big)\\
&\leq c(n,s)\tau^{-\frac n2}\bigg[\tau^{\beta(j-k-1)}\mathcal E_{\g}\big(u;B_{{\tau^{k+1}\vrho}}(x_0)\big)+\frac{c}{1-\tau^\beta}\sup_{\tau^{j-1}\vrho\leq t\leq \tau^{k+1}\vrho}\,t^2\bigg(\mint_{B_t(x_0)}|f|^\chi\dx\bigg)^{1/\chi}\bigg]\\
&\leq c\bigg[\Big(\frac{\sigma}{\rho}\Big)^\beta\mathcal E_{\g}\big(u;B_{\rho}(x_0)\big)+\sup_{\sigma\leq t\leq \rho}t^2\bigg(\mint_{B_t(x_0)}|f|^\chi\dx\bigg)^{1/\chi}\bigg]\,,
\end{align*}
with now $c=c(\data,\beta)$, since $\tau^{j-k-1}\leq \tau^{-2}(\sigma/\rho)\leq c\,(\sigma/\rho)$. In the case $\sigma\in[\tau\rho,\rho]$ the conclusion is again standard by \eqref{ut:exc}. Using \eqref{eps-small} in \eqref{impr:excess.general} immediately leads to \eqref{unif:EXCS}.
\end{proof}

\section{Regular set and therein regularity of solutions}\label{sec:regset}
For $u\in W^{1,2}_{\rm loc}(\Omega;\R^N)\cap L^1_{2s}$ a fixed local solution to \eqref{eq1}, under the assumptions of Theorem \ref{part.BMO}, and $\gamma$ as in \eqref{dep:gamma}, let us define the $\gamma$-singular set
\[
\Omega_{{\rm sing},\gamma}=\big\{x_0\in\Omega:\limsup_{\vrho\to0^+}\mathcal E_\gamma(u;B_\vrho(x_0))>0\big\}\,.
\]
Taking into account \eqref{exc.vanishing} from Lemma \ref{exc.vanishing}, with $\zeta$ defined in Proposition \ref{prop:highint}, we have
\[
\Omega_{{\rm sing},\gamma}\subseteq\big\{x_0\in\Omega:\limsup_{\vrho\to0^+}E(u;B_\vrho(x_0))>0\big\}
\]
and
\[
\Omega_{{\rm sing},\gamma}\subseteq\Big\{x_0\in\Omega:\limsup_{\vrho\to0^+}\vrho^{-[n-(2+\zeta)]}\int_{B_\vrho(x_0)} |Du|^{2+\zeta}\,dx>0
\Big\}\,.
\]
The latter inclusion follows from Poincar\'e's inequality
\[
\bigg(
\mint_{B_\vrho(x_0)} \big|u-(u)_{B_\vrho(x_0)}\big|^2\,dx \bigg)^{(2+\zeta)/2}
\leq c(n,\zeta)\vrho^{2+\zeta-n} \int_{B_\vrho(x_0)} |Du|^{2+\zeta}\,dx\,.
\]
Therefore $\Omega_{{\rm sing},\gamma}=\emptyset$ in dimension $n=2$, while,
if $n\ge3$, the classical Giusti lemma \cite[Chapter IV, Theorem 2.2]{giaq}
gives
\[
\dim_{\mathcal H}(\Omega_{{\rm sing},\gamma}) \leq n-2-\zeta\,.
\]
 
\subsection{A general construction}\label{general.construction}
We now record a construction that will be used with different choices of the parameter $\beta\in(0,1)$. Fix $\beta\in(0,1)$ and let $\varepsilon_s = \varepsilon_s(\data,\omegau(\,\cdot\,),\beta)$, $r_1=r_1(\data,\beta,\nu)$, $\delta_0=\delta_0(\data)$ and $\tau=\tau(\data,\beta)$ be the quantities provided by Proposition \ref{dec:one}. We assume that the smallness hypotheses \eqref{x0:BMOd} and \eqref{ass.BMO.f0} hold with this value of $\delta_0$ and with $\varepsilon_b:=\varepsilon_s/2$. Let also $r_0$ satisfy \eqref{cond:r0} and \eqref{rpiccolo}. 

\vs

We fix a compact exhaustion $(K_i)_{i\in\mathbb N}$ of $\Omega$ such that $K_i\Subset {\rm int}(K_{i+1})$ for every $i\in\mathbb N$. For every $i\in\mathbb N$ there exist radii $r_A(K_i),r_f(K_i)\in(0,{\rm dist}(K_i,\partial\Omega))$ such that \eqref{x0:BMOd} and \eqref{ass.BMO.f0} hold on $K_i$; we set
\begin{equation}\label{bar.r}
\bar r_\beta(K_i):=\min\Big\{r_0,r_1,r_f(K_i),r_A(K_i),\frac14{\rm dist}\big(K_i,\partial\Omega\big)\Big\}\,,
\end{equation}
and then
\begin{equation}\label{Om.u.1}
K_{u,i,\beta} := \Big\{x\in{\rm int}(K_i): \exists \vrho_x\in\big(0,\bar r_\beta(K_i)\big)\text{ such that } B_{\vrho_x}(x)\Subset\Omega \text{ and } \mathcal E_\gamma(u;B_{\vrho_x}(x))<\varepsilon_s/2 \Big\}
\end{equation}
and
\begin{equation}\label{Om.u.2}
\Omega_{u,\beta} := \bigcup_{i\in\mathbb N}K_{u,i,\beta}.
\end{equation}
By the continuity of $x\mapsto \mathcal E_\gamma(u;B_\vrho(x))$ for fixed $\vrho$ (see the discussion after \eqref{def:exc})  the set $\Omega_{u,\beta}$ is open. Moreover, if $x_0\in\Omega\setminus\Omega_{u,\beta}$, then
\[
\limsup_{\vrho\to0^+} \mathcal E_\gamma(u;B_\vrho(x_0)) \geq \frac{\varepsilon_s}{2}.
\]
Indeed, otherwise, since $x_0\in{\rm int}(K_i)$ for some $i$, one could choose $\vrho<\bar r_\beta(K_i)$ sufficiently small so that $B_\vrho(x_0)\Subset\Omega$ and $\mathcal E_\gamma(u;B_\vrho(x_0))<\varepsilon_s/2$, which would imply $x_0\in K_{u,i,\beta}\subset\Omega_{u,\beta}$, a contradiction. Hence $\Omega\setminus\Omega_{u,\beta}\subseteq\Omega_{{\rm sing},\gamma}$; consequently, $\Omega\setminus\Omega_{u,\beta}=\emptyset$ if $n=2$ while, if $n\ge3$,
\begin{equation}\label{hausdorff.estimate}
\dim_{\mathcal H}(\Omega\setminus\Omega_{u,\beta}) \leq n-2-\zeta\,. 
\end{equation}
The next proofs rely on the following observation. Let $x_0\in\Omega_{u,\beta}$; then, there exist $i=i(x_0)\in\mathbb N$ and a radius $\vrho_{x_0}\in(0,\bar r_\beta(K_i))$ such that $B_{\vrho_{x_0}}(x_0)\Subset\Omega$ and 
\begin{equation}\label{non:cancellare}
\mathcal E_\gamma\big(u;B_{\vrho_{x_0}}(x_0)\big)<\frac{\varepsilon_s}{2}\,. 
\end{equation}
By the already mentioned continuity of the map $x\mapsto\mathcal E_\gamma(u;B_{\vrho_{x_0}}(x))$, there exists a small radius $r_{x_0}\leq \vrho_{x_0}/8$ such that $B_{r_{x_0}}(x_0)\Subset{\rm int}(K_i)$, $B_{\vrho_{x_0}}(\bar x)\Subset\Omega$ for every $\bar x\in B_{r_{x_0}}(x_0)$ and
\begin{equation}\label{local.Egamma.bdd}
  \mathcal E_\gamma\big(u;B_{\vrho_{x_0}}(\bar x)\big) < \frac{\varepsilon_s}{2}\qquad\text{for every $\bar x\in B_{r_{x_0}}(x_0)$}\,;
\end{equation}
since $\vrho_{x_0}<r_f(K_i)$ and $\bar x\in K_i$, assumption \eqref{ass.BMO.f0}, applied with $K=K_i$, gives 
\begin{equation}\label{small.nou.f}
\sup_{\substack{\bar x\in B_{r_{x_0}}(x_0)\\0<t\leq \vrho_{x_0}}} t^2\bigg(\mint_{B_t(\bar x)}|f|^\chi\dx \bigg)^{1/\chi}< \frac{\varepsilon_s}{2}\,. 
\end{equation}
Therefore the assumption \eqref{eps-small} of Proposition \ref{dec:one} holds on every ball $B_{\vrho_{x_0}}(\bar x)$ with $\bar x\in B_{r_{x_0}}(x_0)$. Moreover, \eqref{x0:BMOd} gives the required coefficient smallness and, since also $\vrho_{x_0}<\min\{r_0,r_1\}$, all assumptions of Proposition \ref{dec:one} are fulfilled on such balls. Therefore Proposition \ref{dec:one} yields
\begin{equation}\label{pointwise.BMO}
\sup_{0<t\leq \vrho_{x_0}}
\mathcal E_\gamma(u;B_t(\bar x)) \leq c(\data,\beta)\,\varepsilon_s\,,
\end{equation}
\begin{equation}\label{local.Egamma.decay}
\mathcal E_{\g}\big(u;B_{\tau^{k+1}\vrho_{x_0}}(\bar x)\big)\leq \tau^\beta \mathcal E_{\g}\big(u;B_{{\tau^k\vrho_{x_0}}}(\bar x)\big)+c\big({\tau^k\vrho_{x_0}}\big)^2\bigg(\mint_{B_{\tau^k\vrho_{x_0}}(\bar x)}|f|^\chi\dx\bigg)^{1/\chi}\qquad\text{for all $k\in \N_0$}\,,
\end{equation}
and
\begin{equation}\label{local.Egamma.decay2}
\mathcal  E_{\g}\big(u;B_{\sigma}(\bar x)\big)\leq c\,\Big(\frac{\sigma}{\rho}\Big)^\beta \mathcal E_{\g}\big(u;B_{\rho}(\bar x)\big)+c\,\sup_{t\leq \rho}\,t^2\bigg(\mint_{B_t(\bar x)}|f|^\chi\dx\bigg)^{1/\chi}
\end{equation}
for all $0<\s\leq \rho\leq \vrho_{x_0}$ and all $\bar x\in B_{r_{x_0}}(x_0)$; all the constants $c$ depend on the $\data$ and $\beta$, but not on $\bar x$. 

%

\begin{proof}[Proof of Theorem \ref{part.BMO}]
The proof is essentially a consequence of Proposition \ref{dec:one}. However, due to the local nature of assumptions \eqref{x0:BMOd}--\eqref{ass.BMO.f0}, the argument is slightly more delicate than in the standard setting, and we therefore provide the details.

\vs

We use the construction above with $\beta=1/2$. Thus $\varepsilon_s=\varepsilon_s(\data,\omegau(\,\cdot\,))$, $r_1=r_1(\data,\nu)$, $\delta_0=\delta_0(\data)$, and the smallness parameter in the statement of the theorem is chosen as $\varepsilon_b:=\varepsilon_s/2$; we also set $\Omega_u:=\Omega_{u,1/2}$. By the construction above, $\Omega_u$ is open, $\Omega\setminus\Omega_u=\emptyset$ if $n=2$, while, if $n\ge3$, estimate \eqref{hausdorff.estimate} holds. It remains to prove \eqref{local.BMO}; the case
$A(x,u)\equiv A(x)$ will be treated separately in Section \ref{sec:everywhere}.

\vs

Let now $K\Subset\Omega_u$. For every $x_0\in K$ we have radii $\vrho_{x_0}>0$ and $r_{x_0}\leq \vrho_{x_0}/8$ satisfying \eqref{pointwise.BMO} for every $\bar x\in B_{r_{x_0}}(x_0)$. The radii $r_{x_0}$ may depend on $x_0$ and may degenerate along $K$; we therefore use compactness: there exist points $x_1,\ldots,x_m\in K$ such that the union of $B_{r_{x_j}}(x_j)$ covers $K$; we set then $r_K:=\min_{1\leq j\leq m}r_{x_j}>0$. We claim that
\[
E_{r_K}(u;\bar x)\leq c(\data)\,\varepsilon_s \qquad\text{for every } \bar x\in K\,.
\]
Indeed, if $\bar x\in K$, then $\bar x\in B_{r_{x_j}}(x_j)$ for some $j\in\{1,\ldots,m\}$. Since $r_K\leq r_{x_j}\leq\vrho_{x_j}$, estimate \eqref{pointwise.BMO}, applied with $x_j$ in place of $x_0$, gives
\[
E_{r_K}(u;\bar x) = \sup_{0<t\leq r_K}E\big(u;B_t(\bar x)\big) \leq \sup_{0<t\leq r_K}\mathcal E_\gamma\big(u;B_t(\bar x)\big) \leq c(\data)\varepsilon_s\,;
\]
thus, by the definition in Paragraph \ref{FS}, $u\in BMO(K;\R^N)$. Since $K\Subset\Omega_u$ was arbitrary, this proves \eqref{local.BMO}.
\end{proof}

\begin{proof}[Proof of Theorem \ref{teo.partVMO}]
We again fix $\beta\in(0,1)$, say $\beta=1/2$, and apply the construction described at the beginning of the section. Indeed, assumption \eqref{ass.BMO.f} implies the validity of \eqref{ass.BMO.f0} on every compact subset of $\Omega$, for some radius $r_f>0$ depending on the compact set. In particular, estimate \eqref{local.Egamma.decay2} holds for all $0<\sigma\leq \rho\leq \vrho_{x_0}$, with a constant $c=c(\data)$ independent of $\bar x\in B_{r_{x_0}}(x_0)$. Hence
\[
\mathcal E_{\gamma}\big(u;B_{\sigma}(\bar x)\big) \leq c\Big(\frac{\sigma}{\rho}\Big)^\beta \mathcal E_{\gamma}\big(u;B_{\rho}(\bar x)\big)+ c\,\sup_{t\leq \rho} t^2 \bigg( \mint_{B_t(\bar x)}|f|^\chi\,dx \bigg)^{1/\chi}\,.
\]
Taking the supremum with respect to $\bar x\in B_{r_{x_0}}(x_0)$ and then the limsup as $\sigma\to0^+$, with $\rho$ fixed, we obtain
\[
\limsup_{\sigma\to0^+} \sup_{\bar x\in B_{r_{x_0}}(x_0)} \mathcal E_\gamma\big(u;B_\sigma(\bar x)\big) \leq c\sup_{\substack{\bar x\in B_{r_{x_0}}(x_0), \\ t\leq\rho}} t^2 \bigg(\mint_{B_t(\bar x)}|f|^\chi\,dx\bigg)^{1/\chi}\,.
\]
Letting $\rho\to0^+$ and using \eqref{ass.BMO.f}, we obtain
\begin{equation}\label{thesis:VMO} 
\lim_{\sigma\to0^+} \mathcal E_\gamma\big(u;B_\sigma(\,\cdot\,)\big) = 0 \qquad \text{uniformly in } B_{r_{x_0}}(x_0)\,.
\end{equation}
This implies $u\in VMO_{\rm loc}(\Omega_u;\R^N)$ by a covering argument analogous to the one used in the proof of Theorem \ref{part.BMO}.
\end{proof}

\begin{proof}[Proof of Theorem \ref{cont.t}]
Again, we fix $\beta\in (0,1)$, say $\beta=1/2$. Thanks to  \eqref{est.potential.onescale}, we see that \eqref{I1.uniformly} implies
\begin{equation*}
    \lim_{\vrho\to0^+}\sup_{t\leq \vrho}\,t^2\bigg(\mint_{B_t(\,\cdot\,)}|f|^\chi\dx\bigg)^{1/\chi}=0 \qquad\text{locally uniformly in $\Omega$}\,,
\end{equation*}
that in turn implies the validity of \eqref{ass.BMO.f0} for some radius $r_f>0$, so that \eqref{local.Egamma.decay} holds true.

Now, let $\tau=\tau(\data)\in(0,1)$ be as given at the beginning of the Section for the choice $\beta=1/2$. We set $\vrho_j = \tau^j \vrho_{x_0}$, and we write, for short, $B_j(\bar x)=B_{\vrho_j}(\bar x)$, $\mathcal E_j(\bar x)=\mathcal E_\gamma(u; B_{\vrho_j}(\bar x))$ for $\bar x\in B_{r_{x_0}}(x_0)$; then by using \eqref{local.Egamma.decay} and recalling the bound before \eqref{alla:fineuso} we deduce  
\[
\mathcal E_{j+1}(\bar x)  \leq \frac12 \mathcal E_j(\bar x) + c\, \vrho_j^2 \bigg( \mint_{B_j(\bar x)} |f|^\chi \dx\bigg)^{1/\chi} 
\]
for all $j\in\N_0$ and all $\bar x\in B_{r_{x_0}}(x_0)$. For $i,k\in\N$, $k > i$, summing over $j=i,i+1,\dots,k$, changing the summation variable and enlarging the summation on the right-hand side, we obtain
\[
\sum_{j=i}^{k+1} \mathcal E_j(\bar x) \leq \mathcal E_i(\bar x) + \frac12\sum_{j=i}^{k+1} \mathcal E_j(\bar x)+ c \sum_{j=i}^k \vrho_j^2\bigg( \mint_{B_j(\bar x)} |f|^\chi \dx\bigg)^{1/\chi} 
\]
and, after reabsorbing, 
\[
\sum_{j=i}^{k+1} \mathcal E_j(\bar x) \leq 2\mathcal E_i(\bar x) + c \sum_{j=i}^\infty \vrho_j^2 \bigg( \mint_{B_j(\bar x)} |f|^\chi \dx\bigg)^{1/\chi}\,;
\]
the constant $c$ depends here only on $\data$. Now, for the same indices, by telescopic summation, using the above estimate and \eqref{est.potential.sum}, we get
\begin{align}\label{uk:cauchy}
\big|(u)_{B_k(\bar x)} - (u)_{B_i(\bar x)}\big| &\leq \sum_{j=i}^k \big|(u)_{B_{j+1}(\bar x)} - (u)_{B_j(\bar x)}\big| \leq \tau^{-n} \sum_{j=i}^{k+1} \bigg(\mint_{B_j(\bar x)} {\big|u - (u)_{B_j(\bar x)}\big|}^2 \dx \bigg)^{1/2} \\
&\leq c\, \mathcal E_i(\bar x)+ c \sum_{j=i}^k \vrho_j^2 \bigg(\mint_{B_j(\bar x)} |f|^\chi\dx \bigg)^{1/\chi} \leq c\, \mathcal E_i(\bar x)+c \,{\bf I}^f_{2,\chi}(\bar x,2\vrho_{i-1})\,.\notag
\end{align}
As the right-hand side of \eqref{uk:cauchy} vanishes uniformly in $B_{r_{x_0}}(x_0)$ thanks to \eqref{thesis:VMO} and \eqref{I1.uniformly}, it follows that the sequence of maps $\{\bar x\mapsto (u)_{B_i(\bar x)}\}_{i\in \N}$ satisfies the Cauchy condition with respect to the uniform convergence in $B_{r_{x_0}}(x_0)$.  As $(u)_{B_i(\bar x)}\rightarrow u(\bar x)$ for $i\to+\infty$, for a.e. $\bar x$ by Lebesgue's point theorem, this implies that $u$ agrees a.e. with a continuous function.
\end{proof}

The next result is a slightly more general version of Theorem \ref{thm:hold}.

\begin{theorem}[Partial H\"older continuity]\label{partial.Holder}
Let $\beta_0\in (0,1)$, and assume that for every compact set $K\Subset\Omega$, there exists a radius $R_f=R_f(K)>0$ such that $f$ satisfies
\begin{equation}\label{new:fbound}
\sup_{\substack{x_0\in K,\\ 0<t\leq R_f}} t^{2-\beta_0}\,\bigg(\mint_{B_t(x_0)} |f|^\chi\dx\bigg)^{1/\chi}<\infty\,;
\end{equation}
suppose moreover that $A(\cdot,\cdot\cdot)$ satisfies \eqref{mod:cont.u}. There exists a small constant  $\delta_0 = \delta_0(\data,\beta_0)$ such that if for every compact $K\Subset\Omega$, there exists a radius $r_A=r_A(K)>0$ for which \eqref{hp:A} holds uniformly for $x_0\in K$, then 
\begin{equation*}
   u \in C^{0,\beta_0}_{\rm loc}(\Omega_u; \R^N)\,,
\end{equation*}
 where $\Omega_u$ is the open set given by \eqref{Om.u.1}--\eqref{Om.u.2} for the choice $\beta=(1+\beta_0)/2$. Moreover, for every $x_0\in \Omega_u$, there exist radii $0<r_{x_0}\leq \vrho_{x_0}$ such that $B_{r_{x_0}}(x_0)\Subset \Omega_u$, $B_{2\vrho_{x_0}}(x_0)\Subset \Omega$, and
{\rm \begin{multline}\label{vera:localHolder}
{[u]}_{C^{0,\beta_0}(B_{r_{x_0}}(x_0))}\leq c\,\vrho_{x_0}^{-\beta_0}\bigg( \mint_{B_{\vrho_{x_0}}(x_0)} {\big|u - (u)_{B_{\vrho_{x_0}}(x_0)}\big|}^2 \dx \bigg)^{1/2}\\
+c\,\vrho_{x_0}^{\gamma-2s-\beta_0}\tail\big(u-(u)_{B_{\vrho_{x_0}}(x_0)};B_{\vrho_{x_0}}(x_0)\big)+ c\,{\big\|\mathrm{M}^{2-\beta_0}_{\vrho_{x_0},\chi}[f]\big\|}_{L^\infty(B_{r_{x_0}}(x_0))}\,,
\end{multline}}
$\!\!$where $c=c(\data,\beta_0)$, $\gamma=\gamma(s)$ is fixed as in \eqref{dep:gamma}, and we set
\begin{equation}\label{maximal:f}
     \mathrm{M}^{2-\beta_0}_{\vrho,\chi}[f](\bar x):=\sup\limits_{t \in (0, \vrho]} t^{2-\beta_0}\bigg( \mint_{B_t(\bar x)}|f|^\chi\dx\bigg)^{1/\chi}.
\end{equation}
Assumption \eqref{new:fbound} is verified if $f \in \mathcal M^{n/(2-\beta_0)}_{\rm loc}(\Omega;\R^N)$, in which case the quantitative estimate \eqref{local.Holder1} holds true.
\end{theorem}

\begin{proof}
We let $\beta=(1+\beta_0)/2$ and define $\Omega_u=\Omega_{u,(1+\beta_0)/2}$ as in \eqref{Om.u.1}--\eqref{Om.u.2}; we fix, according to this value of $\beta$, $\varepsilon_s = \varepsilon_s(\data,\omegau(\,\cdot\,),\beta_0)$, $r_1=r_1(\data,\beta_0,\nu)$ and $\delta_0=\delta_0(\data,\beta_0)$ as at the beginning of this Section. We then fix a generic $x_0\in\Omega_u$; as described in Paragraph \ref{general.construction}, there exists an index $i=i(x_0)$ such that $x_0\in {\rm int}(K_i)$ and we let now $r_A$ to be the constant $r_A(K_i)$ given from the assumption of the Theorem. Moreover, from  \eqref{new:fbound} it follows immediately that there exists a radius $0<r_f\leq R_f$ such that  \eqref{ass.BMO.f0} is satisfied on $K_i$ with $\e_b=\e_s/2$; moreover we can suppose $r_f\leq {\rm dist}(K_i,\partial\Omega)/4$. Using the general construction of  Paragraph \ref{general.construction},  we may find radii $r_{x_0}\leq \vrho_{x_0}/8$ such that \eqref{local.Egamma.decay2} holds and $B_{\vrho_{x_0}}(\bar x)\subset B_{2\vrho_{x_0}}(x_0)\Subset\Omega$ for all $\bar x\in B_{r_{x_0}}(x_0)$. 

\vs

We now choose $\vartheta=\vartheta(\data,\beta_0)\in (0,1/16)$ such that $c\,\vartheta^{\beta-\beta_0}=c\,\vartheta^{(1-\beta_0)/2}\leq 1/2$, with $c=c(\data,\beta_0)$ exactly the constant appearing in \eqref{local.Egamma.decay2}; thus, by choosing $\s=\vartheta\rho$ for $0<\rho\leq \vrho_{x_0}$ in \eqref{local.Egamma.decay2}, we find
\begin{equation}\label{6}
       \mathcal E_\g\big(u; B_{\vartheta\rho}(\bar x)\big)\leq \frac{1}{2}\vartheta^{\b_0} \mathcal E_\g\big(u; B_{\rho}(\bar x)\big)+ c\,\rho^{\beta_0}\sup_{t\leq \rho}t^{2-\beta_0}\bigg(\mint_{B_t(\bar x)}|f|^\chi\dx\bigg)^{1/\chi}\quad\text{for all $\bar x\in B_{r_{x_0}}(x_0)$}\,,
\end{equation}
with $c=c(\data,\beta_0)$. Now, for $0<\tilde \rho\leq \rho$, we define the fractional maximal operators
\begin{equation}\label{def:maximu}
\begin{cases}
\mathrm{M}^{\beta_0}_{\tilde \rho,\rho}(\bar x) := \sup\limits_{t \in (\tilde \rho,\rho]} t^{-\beta_0} \mathcal E_\g\big(u; B_t(\bar x)\big)
\\
 \mathrm{M}^{\beta_0}_{\rho}(\bar x) := \sup\limits_{t \in (0, \rho]} t^{-\beta_0} \mathcal E_\g\big(u; B_t(\bar x)\big)
\end{cases}\,,
\end{equation}
and $ \mathrm{M}^{2-\beta_0}_{\vrho,\chi}[f](\bar x)$ is defined by \eqref{maximal:f}. Multiplying \eqref{6} by $(\vartheta\rho)^{-\beta_0}$, and by taking the supremum for $\rho\in(\vrho_{x_0}/[2k],\vrho_{x_0}/2]$ for $k \in \N\setminus\{0,1\}$, we deduce
\begin{align*}
\mathrm{M}^{\beta_0}_{\vartheta\vrho_{x_0}/[2k],\vartheta\vrho_{x_0}/2}(\bar x) &\leq \frac{1}{2} \mathrm{M}^{\beta_0}_{\vrho_{x_0}/[2k],\vrho_{x_0}/2}(\bar x) + c\,\vartheta^{-\beta_0}\mathrm{M}^{2-\beta_0}_{\vrho_{x_0},\chi}[f](\bar x)\\
&\leq \frac12\mathrm{M}^{\beta_0}_{\vartheta\vrho_{x_0}/[2k],\vartheta\vrho_{x_0}/2}(\bar x)+ \frac{1}{2} \mathrm{M}^{\beta_0}_{\vartheta\vrho_{x_0}/2,\vrho_{x_0}/2}(\bar x) + c\,\mathrm{M}^{2-\beta_0}_{\vrho_{x_0},\chi}[f](\bar x)\,,
\end{align*}
for all $\bar x\in B_{r_{x_0}}(x_0)$, where we also used that $\vartheta$ depends on $\data$ and $\beta_0$. By reabsorbing terms, we get
\begin{equation*}
    \mathrm{M}^{\beta_0}_{\vartheta\vrho_{x_0}/[2k],\vartheta\vrho_{x_0}/2}(\bar x)\leq \mathrm{M}^{\beta_0}_{\vartheta\vrho_{x_0}/2,\vrho_{x_0}/2}(\bar x) + c\,\mathrm{M}^{2-\beta_0}_{\vrho_{x_0},\chi}[f](\bar x)\,;
\end{equation*}
since, using the definition in \eqref{def:maximu} and the fact that $\vartheta=\vartheta(\data,\beta_0)$, we have
\[
\mathrm{M}^{\beta_0}_{\vartheta \vrho_{x_0}/2,\vrho_{x_0}/2}(\bar x) \leq c(\vrho_{x_0}/2)^{-\beta_0}\mathcal E_{\g}\big(u;B_{\vrho_{x_0}/2}(\bar x)\big)= c\,\vrho_{x_0}^{-\beta_0}\mathcal E_{\g}\big(u;B_{\vrho_{x_0}/2}(\bar x)\big)\,,
\]
combining the two estimates above, we obtain
\[
\mathrm{M}^{\beta_0}_{\vartheta\vrho_{x_0}/[2k],\vrho_{x_0}/2}(\bar x)\leq  c\,\vrho_{x_0}^{-\beta_0}\mathcal E_{\g}\big(u;B_{\vrho_{x_0}/2}(\bar x)\big)+ c{\big\|\mathrm{M}^{2-\beta_0}_{\vrho_{x_0},\chi}[f]\big\|}_{L^\infty(B_{r_{x_0}}(x_0))}\,,
\]
with $c=c(\data,\beta_0)$, for all $\bar x\in B_{r_{x_0}}(x_0)$. On the other hand, recalling that $|\bar x-x_0|\leq r_{x_0}\leq \vrho_{x_0}/8\leq 1/8$, we have $B_{\vrho_{x_0}/2}(\bar x)\subseteq B_{\vrho_{x_0}}(x_0)$; hence we may use \eqref{media:nonconc} and \eqref{tail:3} in \eqref{def:exc} and get
\begin{equation*}
    \mathcal{E}_\gamma\big(u;B_{\vrho_{x_0}/2}(\bar x)\big)\leq c\,\mathcal{E}_\g\big(u;B_{\vrho_{x_0}}(x_0)\big)\qquad\text{for all $\bar x\in B_{r_{x_0}}(x_0)$}
\end{equation*}
-- recall that $\gamma-2s\geq0$. Combining the two estimates above, letting $k\to\infty$, and expanding the definition of $\mathcal E_\gamma(\,\cdot\,)$, we get
\begin{align*}
\mathrm{M}^{\beta_0}_{\vrho_{x_0}/2}(\bar x)&\leq  c\,\vrho_{x_0}^{-\beta_0}\mathcal E_{\g}\big(u;B_{\vrho_{x_0}}(x_0)\big)+ c{\big\|\mathrm{M}^{2-\beta_0}_{\vrho_{x_0},\chi}[f]\big\|}_{L^\infty(B_{r_{x_0}}(x_0))}\notag
    \\
    &\leq c\,\vrho_{x_0}^{-\beta_0}\bigg( \mint_{B_{\vrho_{x_0}}(x_0)} {\big|u - (u)_{B_{\vrho_{x_0}}(x_0)}\big|}^2 \dx \bigg)^{1/2}\notag
\\
&\qquad\qquad+c\,\vrho_{x_0}^{\gamma-2s-\beta_0}\tail\big(u-(u)_{B_{\vrho_{x_0}}(x_0)};B_{\vrho_{x_0}}(x_0)\big)+ c{\big\|\mathrm{M}^{2-\beta_0}_{\vrho_{x_0},\chi}[f]\big\|}_{L^\infty(B_{r_{x_0}}(x_0))}\,,
    \end{align*}
which is valid for every $\bar x\in B_{r_{x_0}}(x_0)$, with $c=c(\data,\beta_0)$.

\vs 

This estimate implies the desired $\beta_0$-H\"older continuity of $u$ and the local estimate \eqref{vera:localHolder} via Campanato's characterization of H\"older continuity \cite[Chapter 3.1]{giaq} since, by the very definition of $\mathrm{M}^{\beta_0}_{\vrho}$ and $\mathcal E_\g(\,\cdot\,)$ in \eqref{def:maximu} and \eqref{def:exc}, we have
\begin{equation*}
\mathrm{M}^{\beta_0}_{\vrho_{x_0}/2}(\bar x) \ge \s^{-\beta_0}\bigg( \mint_{B_\s(\bar x)} {\big|u - (u)_{B_\s(\bar x)}\big|}^2 \dx \bigg)^{1/2}
\end{equation*}
for every $t\in(0,\vrho_{x_0}/2]$, and $\bar x\in B_{r_{x_0}}(x_0)$. Finally, thanks to \eqref{density.Marc}, if \eqref{ass:fhold} holds, then we have, for $\chi$ as chosen after \eqref{density.Marc}
\begin{equation*}
    t^{2-\beta_0} \bigg(\mint_{B_t(\bar x)} |f|^\chi \dx\bigg)^{1/\chi}\leq c(n,\beta_0){\|f\|}_{\mathcal M^{n/(2-\beta_0)}(B_{t}(\bar x))}\leq c(n,\beta_0){\|f\|}_{\mathcal M^{n/(2-\beta_0)}(B_{2\vrho_{x_0}}(x_0))}
\end{equation*}
for all $\bar x\in B_{r_{x_0}}(x_0)$ and all $0<t\leq\varrho_{x_0}$, hence \eqref{new:fbound} is valid, and we have
\begin{equation*}
    {\big\|  \mathrm{M}^{2-\beta_0}_{\vrho_{x_0},\chi}[f]\big\|}_{L^\infty(B_{r_{x_0}}(x_0))}\leq c(n,\beta_0){\|f\|}_{\mathcal M^{n/(2-\beta_0)}(B_{2\vrho_{x_0}}(x_0))}\,.
\end{equation*}
From this inequality,
\eqref{vera:localHolder} and using that $\vrho_{x_0}^{\gamma}\leq \vrho_{x_0}$ by \eqref{dep:gamma} as $\vrho_{x_0}\leq 1$, estimate \eqref{local.Holder1} follows; local H\"older regularity follows using a covering argument similar to that employed in the proof of Theorem \ref{part.BMO}.
\end{proof}

\section{Gradient regularity of solutions}\label{sec:gradreg}

 We start with some preliminary remarks; in particular, let us verify the validity of Theorem \ref{partial.Holder} whenever assumptions  \eqref{mod:cont.x} and \eqref{fin:riesz} are in force, and consequently derive a H\"older estimate in the reference ball $B_{r_{x_0}}(x_0)$ for every exponent $\beta_0\in(0,1)$.

\vs

Let us now fix $\beta_0=\beta_0(s)\in(0,1)\setminus\{2s\}$ such that
\begin{equation}\label{cond:beta}
2s<1+\beta_0 \qquad\Longleftrightarrow\qquad 1+\beta_0-2s>0\,,
\end{equation}
and let $x_0\in\Omega_u$, where $\Omega_u$ is the open set given by \eqref{Om.u.1}--\eqref{Om.u.2} for the choice $\beta=(1+\beta_0)/2$. Let $\delta_0=\delta_0(\data,\beta_0)$ be the small constant appearing in Theorem \ref{partial.Holder}. By assumption \eqref{mod:cont.x} and the fact that $\omegax(\,\cdot\,)$ is a modulus of continuity, we may choose $r_A=r_A\big(\data,\beta_0,\omegax(\,\cdot\,), {\rm dist}(K,\partial\Omega)\big)\in(0,1)$ such that \eqref{hp:A} holds uniformly for $x_0\in K$, as
\begin{equation}\label{small.ass61}
\sup_{x_0\in K}E_{r}(A;x_0)=\sup_{x_0\in K}\sup_{\substack{\vrho\in(0,r]:B_{\vrho}(x_{0})\subseteq\Omega,\\ z\in\R^N,z\neq 0}} \mint_{B_{\vrho}(x_{0})}\big|{A(x,z)-{(A)}_{B_{\vrho}(x_{0})}(z)}\big|\dx\leq \Lambda\omegax(r)\xrightarrow{r\to 0} 0. 
\end{equation}
 Moreover, by \eqref{fin:riesz} and \eqref{est.potential.onescale}, condition \eqref{new:fbound} is satisfied, with $R_f=\min\{1,{\rm dist}(K,\partial\Omega)\}/2$; hence all the hypotheses of Theorem \ref{partial.Holder} are fulfilled. Let now $r_{x_0}$ and $\vrho_{x_0}$ be the radii provided by Theorem \ref{partial.Holder}, with $0<r_{x_0}\leq \vrho_{x_0}/8$ such that $B_{r_{x_0}}(x_0)\Subset \Omega_u$, $B_{2\vrho_{x_0}}(x_0)\Subset \Omega$. Using again \eqref{est.potential.onescale}, we have
\[
{\big\| \mathrm{M}^{2-\beta_0}_{\vrho_{x_0},\chi}[f]\big\|}_{L^\infty(B_{r_{x_0}}(x_0))}\leq c(n,\beta_0,\chi){\big\|{\bf I}_{1,\chi}^f(\cdot,2\vrho_{x_0})\big\|}_{L^\infty(B_{r_{x_0}}(x_0))}.
\]
Therefore, estimate \eqref{vera:localHolder} gives
\begin{align}\label{usholder}
[u]_{C^{0,\beta_0}(B_{r_{x_0}}(x_0))}
&\leq c\,\vrho_{x_0}^{-\beta_0}\bigg(\mint_{B_{\vrho_{x_0}}(x_0)}\big|u-(u)_{B_{\vrho_{x_0}}(x_0)}\big|^2 \dx\bigg)^{1/2}+c\,\vrho_{x_0}^{\gamma-\beta_0-2s}\tail\big(u-(u)_{B_{\vrho_{x_0}}(x_0)};B_{\vrho_{x_0}}(x_0)\big)\notag\\
&\pushright{+\,c\,{\big\|{\bf I}_{1,\chi}^f(\cdot,2\vrho_{x_0})\big\|}_{L^\infty(B_{r_{x_0}}(x_0))}}\notag\\
&\leq c\,\vrho_{x_0}^{-n/2}{\|Du\|}_{L^2(B_{\vrho_{x_0}}(x_0))}+c\,\vrho_{x_0}^{\gamma-\beta_0-2s}\tail\big(u-(u)_{B_{\vrho_{x_0}}(x_0)};B_{\vrho_{x_0}}(x_0)\big)\notag\\
&\pushright{+\,c\,{\big\|{\bf I}_{1,\chi}^f(\cdot,2\vrho_{x_0})\big\|}_{L^\infty(B_{r_{x_0}}(x_0))}}\,,
\end{align}
where $c=c(\data,\beta_0)$. In the last inequality we used Poincar\'e's inequality and the fact that $\vrho_{x_0}^{1-\beta_0}\leq 1$.

\subsection{A gradient comparison estimate.} Let us take a point $\bar x\in B_{r_{x_0}/2}(x_0)$ and a radius $0<r\leq r_{x_0}/4$. Accordingly, we define
\begin{equation}\label{def:Ax}
\bar A=A\big(\bar x, u(\bar x)\big)\,.
\end{equation}
Here and in what follows, $u$ denotes its continuous representative in $\Omega_u$, so that $A(\bar x,u(\bar x))$ is well-defined for every $\bar x\in B_{r_{x_0}/2}(x_0)$, thanks to what proved at the beginning of the Section. Accordingly, with a slight abuse of notation, we shall write ${[u]}_{C^{0,\beta_0}}$ for the H\"older seminorm of this representative. In particular, the pointwise values $u(x)$ and $u(\bar x)$ appearing below are well defined for every $x,\bar x\in B_{r_{x_0}}(x_0)$.

\medskip 

Consider $h\in W^{1,2}(B_r(\bar x);\R^N)$ solving
\begin{equation}\label{new:eqv}
    \begin{cases}
        -\mathrm{div}\big(\bar ADh \big)=0\quad & \text{in $B_r(\bar x)$}
        \\
        h=u\quad & \text{on $\partial B_r(\bar x)$}
    \end{cases},
\end{equation}
and we extend $h$ outside $B_r(\bar x)$ by  setting
\begin{equation}\label{v:extension}
    h\equiv u\quad\text{in $\R^n\setminus B_r(\bar x)$.}
\end{equation}

Let us now prove the following comparison estimate.
\begin{proposition}[Comparison estimate]\label{prop:comparison}
Let $u\in W^{1,2}_{\rm loc}(\Omega;\R^N)\cap L^1_{2s}$ be a local weak solution to \eqref{eq1} under the assumptions \eqref{A:gr} and \eqref{phi:gr}--\eqref{sym:phi}, and let $x_0,\bar x,r_{x_0}, r$ and $h\in W^{1,2}(B_r(\bar x);\R^N)$ be as above.  Then the following comparison estimate holds:{\rm 
\begin{multline}\label{st:comp5}
    \bigg(\mint_{B_r(\bar x)}|Du-Dh|^2\dx \bigg)^{1/2}\leq  c\Big[{\big\|A(\cdot,u(\,\cdot\,))-\bar A\big\|}_{L^\infty(B_r(\bar x))}+r^{2(1-s)}\Big]\bigg(\mint_{B_{2r}(\bar x)}|Du|^2\dx\bigg)^{1/2}\\
    +c\,r\bigg( \mint_{B_r(\bar x)}|f|^{2_*}\dx\bigg)^{1/2_*}+c\,r^{1-2s}\tail\big(u-(u)_{B_r(\bar x)};B_r(\bar x)\big)\,,
\end{multline}}
$\!\!$where $c=c(\data,\max\{\nu,1\})$, and $\bar A$ is given by \eqref{def:Ax}.
\end{proposition}

\begin{proof}
Using the coercivity and growth assumptions \eqref{A:gr} (which are also inherited by $\bar A$), and testing \eqref{eq1} and \eqref{new:eqv} with $u-h$, we get
\begin{align}\label{st:comp0}
    \frac{1}{\Lambda}\int_{B_r(\bar x)}|Du & -Dh|^2\dx
    \leq \int_{B_r(\bar x)}\langle\bar A(Du-Dh), Du-Dh\rangle\dx=\int_{B_r(\bar x)}\langle\bar ADu,Du-Dh\rangle\dx\notag\\
    &=\int_{B_r(\bar x)} \Big\langle\big[\bar A-A(x,u(x))\big]Du,Du-Dh\Big\rangle\dx+\int_{B_r(\bar x)}\big\langle A(x,u(x))Du, Du-Dh\big\rangle\dx\notag\\
    &= \int_{B_r(\bar x)} \Big\langle\big[\bar A-A(x,u(x))\big]Du,Du-Dh\Big\rangle\dx+\int_{B_r(\bar x)}f\cdot (u-h)\dx-{\llangle Q^\mathrm{nl}_\Phi u, u-h\rrangle}_{\mathrm{nl}} \notag\\
    &\eqqcolon (I)+(II)+(III)\,;
\end{align}
we recall the notation introduced in \eqref{bracket}. By H\"older's inequality, we get
\begin{align}\label{st:comp1}
    |(I)|&\leq {\big\|A(\cdot,u(\,\cdot\,))-\bar A\big\|}_{L^\infty(B_r(\bar x))}\,\int_{B_r(\bar x)}|Du|\,|Du-Dh|\dx\notag\\
    &\leq c(n){\|A\big(\cdot,u(\,\cdot\,) \big)-\bar A\|}_{L^\infty(B_r(\bar x))}\,r^{n/2}\bigg(\mint_{B_r(\bar x)} |Du|^2\dx\bigg)^{1/2}\,\bigg(\int_{B_r(\bar x)} |Du-Dh|^2\dx\bigg)^{1/2}\,,
\end{align}
and by H\"older's and Sobolev's inequalities
\begin{align}\label{st:comp2}
    |(II)| & \leq c(n)r^{n/2_*}\bigg(\mint_{B_r(\bar x)}  |f|^{2_*}\dx\bigg)^{1/2_*}\bigg(\int_{B_r(\bar x)} |u-h|^{2^*}\dx\bigg)^{1/2^*}\notag\\
    &\leq c(n)r^{n/2_*}\bigg(\mint_{B_r(\bar x)}  |f|^{2_*}\dx\bigg)^{1/2_*}\bigg(\int_{B_r(\bar x)} |Du-Dh|^{2}\dx\bigg)^{1/2} \,.
\end{align}
Then, by using \eqref{sym:phi} and recalling that $u-h\equiv 0$ in $\R^n\setminus B_r(\bar x)$ by \eqref{v:extension}, we get
\begin{multline}\label{est:nonlocal}
    (III)= -\int_{B_{2r}(\bar x)}\int_{B_{2r}(\bar x)}\Phi\big(y,z,u(y),u(z),u(y)-u(z)\big)\cdot \big[ (u-h)(y)-(u-h)(z)\big] \,\frac{dy\,dz}{|y-z|^{n+2s}}\\
-\,2\int_{B_r(\bar x)}\int_{\R^n\setminus B_{2r}(\bar x)}\Phi\big(y,z,u(y),u(z),u(y)-u(z)\big)\cdot(u-h)(z)\,\frac{dy\,dz}{|y-z|^{n+2s}}\eqqcolon (III)_1+(III)_2\,.
\end{multline}
Using \eqref{phi:gr}, H\"older's  inequality and the fractional Sobolev inequality \eqref{frac:emb}, we have
\begin{align}\label{st:comp3}
    |(III)_1|&\leq c(n)\nu[u]_{W^{s,2}(B_{2r}(\bar x))}\,[u-h]_{W^{s,2}(B_{2r}(\bar x))}\notag\\
    &\leq c(n,s)\nu\, r^{2(1-s)}\bigg(\int_{B_{2r}(\bar x)}|Du|^2\dx \bigg)^{1/2}\bigg(\int_{B_r(\bar x)}|Du-Dh|^2\dx \bigg)^{1/2}\,,
\end{align}
where in the last line we also used that $Du=Dh$ in $B_{2r}(\bar x)\setminus B_r(\bar x)$ by \eqref{v:extension}. Then, by using the elementary inequalities
\begin{equation}\label{zbarx}
    |y-\bar x|\leq 4|y-z|,\qquad \text{for $y\in \R^n\setminus B_{2r}(\bar x)$ and $z\in B_r(\bar x)$}\,,
\end{equation}
together with \eqref{phi:gr}, \eqref{tail.constant}, H\"older's and Poincar\'e's inequalities, we get
\begin{align}\label{st:comp4}
    |(III)_2| &\leq c\,\nu\int_{B_r(\bar x)}|(u-h)(z)|\,\bigg\{\int_{\R^n\setminus B_{2r}(\bar x)}\big(|u(y)-(u)_{B_r(\bar x)}|+|u(z)-(u)_{B_r(\bar x)}|\big)\frac{dy}{|y-\bar x|^{n+2s}} \bigg\}\,dz\notag
    \\
    &= c\,r^{-2s}\tail\big(u-(u)_{B_r(\bar x)};B_{2r}(\bar x)\big)\int_{B_r(\bar x)}|(u-h)(z)|\,dz+c\,r^{-2s}\int_{B_r(\bar x)}|u-h||u-(u)_{B_r(\bar x)}|\,dz\notag
    \\
    &\leq c\,r^{n/2-2s}\tail\big(u-(u)_{B_r(\bar x)};B_{2r}(\bar x)\big)\bigg(\int_{B_r(\bar x)}|u-h|^2\,dz\bigg)^{1/2}\notag
    \\
    &\pushright{+\,c\,r^{-2s}\bigg(\int_{B_r(\bar x)}|u-h|^2\,dz\bigg)^{1/2}\,\bigg(\int_{B_r(\bar x)}|u-(u)_{B_r(\bar x)}|^2\,dz\bigg)^{1/2}\notag}\\
    &\leq c\,r^{n/2+1-2s}\tail\big(u-(u)_{B_r(\bar x)};B_{2r}(\bar x)\big)\,\bigg(\int_{B_r(\bar x)}|Du-Dh|^2\,dz\bigg)^{1/2}\notag
    \\
    &\pushright{+\, c\,r^{n/2+2(1-s)}\bigg(\int_{B_r(\bar x)}|Du-Dh|^2\,dz\bigg)^{1/2}\bigg(\mint_{B_{2r}(\bar x)}|Du|^2\,dz\bigg)^{1/2}\,,}
\end{align}
with $c=c(n,s,\max\{\nu,1\})$.
Combining \eqref{st:comp1}--\eqref{st:comp3} and \eqref{st:comp4} with \eqref{st:comp0}, and dividing both sides of the resulting inequality by $\displaystyle{r^{n/2}\Big(\int_{B_r(\bar x)}|Du-Dh|^2\dx \Big)^{1/2}}$ (which may be assumed to be nonzero) yields our desired estimate
\begin{multline*}
\bigg(\mint_{B_r(\bar x)}|Du-Dh|^2\dx \bigg)^{1/2}\leq  c\,\bigg\{{\big\|A\big(\cdot,u(\,\cdot\,)\big)-\bar A\big\|}_{L^\infty(B_r(\bar x))}+r^{2(1-s)}\bigg\}\bigg(\mint_{B_{2r}(\bar x)}|Du|^2\,dz\bigg)^{1/2}\\
+c\,r\bigg( \mint_{B_r(\bar x)}|f|^{2_*}\dx\bigg)^{1/2_*}+c\,r^{1-2s}\tail\big(u-(u)_{B_r(\bar x)};B_r(\bar x)\big)\,,
\end{multline*}
where $c=c(n, s,\max\{\nu,1\},\Lambda)$, where we also used \eqref{counter.monotonicity}, and that
\begin{equation*}
    \frac{n}{2_*}-\frac{n}{2}=\begin{cases}
        1\quad&\text{for $n\geq 3$}
        \\[2mm]
        \text{any number $>1$ }\quad&\text{for $n=2$}
    \end{cases};
\end{equation*}
this concludes the proof.
\end{proof}

Before proving the above-stated theorems on the gradient, we establish  a few preparatory lemmas. Using the comparison estimate \eqref{st:comp5}, we can prove the following excess decay estimate for $Du$.

\begin{proposition}[Excess decay estimate]
Under the notation and assumptions of Proposition \ref{prop:comparison}, we have
{\rm
\begin{multline}\label{st:comp7}
E\big(Du;B_\vrho(\bar x)\big)\leq c\bigg[\frac\vrho r+\Big( \frac{r}\vrho\Big)^{n/2}\Big[ {\big\|A(\cdot,u(\,\cdot\,))-\bar A\big\|}_{L^\infty(B_r(\bar x))}+r^{2(1-s)}\Big]\bigg] E\big(Du;B_{2r}(\bar x)\big)\\
+c\,\Big(\frac r{\vrho}\Big)^{n/2}\Big[ {\big\|A(\cdot,u(\,\cdot\,))-\bar A\big\|}_{L^\infty(B_r(\bar x))}+r^{2(1-s)}\Big]\big|{(Du)}_{B_{2r}(\bar x)}\big|\\
+c\,\Big( \frac r{\vrho}\Big)^{n/2}r\bigg(\mint_{B_r(\bar x)}|f|^{2_*}\dx\bigg)^{1/2_*}+c\,\Big(\frac r{\vrho}\Big)^{n/2}r^{1-2s}\tail\big(u-(u)_{B_r(\bar x)};B_r(\bar x)\big)\,,
\end{multline}
}
$\!\!$for every $0<\vrho\leq r\leq r_{x_0}/2$, where $c=c(\data,\max\{\nu,1\})$, and $E(Du;B_r(x))$ is the excess functional defined by \eqref{defL2:excess}.
\end{proposition}

\begin{proof}
As $h\in W^{1,2}(B_r(\bar x))$ is a solution to a system with constant coefficients, a standard decay estimate is
\begin{equation*}
\mint_{B_\vrho(\bar x)}{\big|Dh-{(Dh)}_{B_\vrho(\bar x)}\big|}^2\dx   \leq c\Big(\frac{\vrho}r\Big)^2\mint_{B_r(\bar x)}{\big|Dh-{(Dh)}_{B_r(\bar x)}\big|}^2\dx\,,
\end{equation*}
for every $0<\vrho\leq r$, where $c=c(n,N,\Lambda)$, see \cite[Chapter III, Theorem 2.1 \& Remark 2.3]{giaq}. Using standard algebraic manipulations taking into account \eqref{media} and the subadditivity several times, the previous decay estimate  and standard manipulations yield
\begin{equation}\label{st:comp6}
\mint_{B_\vrho(\bar x)}{\big|Du-{(Du)}_{B_\vrho(\bar x)}\big|}^2\dx\leq c\,\Big(\frac{\vrho}r\Big)^2\mint_{B_r(\bar x)}{\big|Du-{(Du)}_{B_r(\bar x)}\big|}^2\dx+c\,\Big(\frac r{\vrho}\Big)^n \mint_{B_r(\bar x)}{|Du-Dh|}^2\dx\,,
\end{equation}
with $c=c(n,N,\Lambda)$ (see, for instance, the proof of Proposition 3.2 for $\beta=1$ in \cite{BCV}, although the problem considered there is different). Combining \eqref{st:comp6} with \eqref{st:comp5}, we get
\begin{align*}
E\big(Du;B_\vrho(\bar x)\big)&\leq  c\, \frac\vrho r\bigg(\mint_{B_{r}(\bar x)}{\big|Du-{(Du)}_{B_r(\bar x)}\big|}^2\dx \bigg)^{1/2}
\\
&\qquad+c\,\Big( \frac{r}\vrho\Big)^{n/2}\,\Big[ {\big\|A(\cdot,u(\,\cdot\,))-\bar A\big\|}_{L^\infty(B_r(\bar x))}+r^{2(1-s)}\Big]\bigg(\mint_{B_{2r}(\bar x)}|Du|^2\dx \bigg)^{1/2} \\
&\qquad\qquad+c\,\Big( \frac r{\vrho}\Big)^{n/2}r\bigg(\mint_{B_r(\bar x)}|f|^{2_*}\dx\bigg)^{1/2_*}+c\,\Big(\frac r{\vrho}\Big)^{n/2}r^{1-2s}\tail\big(u-(u)_{B_r(\bar x)};B_r(\bar x)\big)\,,
\end{align*}
for all $0<\vrho\leq r$; using \eqref{media} and the elementary inequality
\begin{equation*}
    \bigg(\mint_{B_{2r}(\bar x)}{|Du|}^2\dx\bigg)^{1/2}\leq c\,E\big(Du;B_{2r}(\bar x)\big)+c\,\big|{(Du)}_{B_{2r}(\bar x)}\big|
\end{equation*}
then yields our desired result \eqref{st:comp7}.
\end{proof}

Now we take advantage of the H\"older continuity of $u$ in order to refine the tail estimate: this is the content of the next proposition.

\begin{proposition}[Tail estimate for H\"older continuous functions]\label{prop:newtail}
Let $w\in C^{0,\beta}(B_{R}(x_0))\cap L^1_{2s}$ for some $\beta \in (0,1]$. Then for every $\bar x\in B_{R/2}(x_0)$, and every radius $0<r\leq R/4$, we have
\rm{
\begin{multline*}
r^{-2s}\tail\big(w-(w)_{B_r(\bar x)};B_r(\bar x)\big)\leq  c\,R^{-2s} \tail\big(w-(w)_{B_{R}(x_0)};B_{R}(x_0)\big)\\
+c\,[w]_{C^{0,\beta}(B_{R}(x_0))}\bigg[R^{\beta-2s}+\int_r^R\mu^{\beta-2s} \,\frac{d\mu}{\mu}\bigg]\,,
\end{multline*}
}
for a constant $c=c(n,s)$.
\end{proposition}

\begin{proof}
By \eqref{tail:1}, we have for $c=c(n,s)$
\begin{multline*}
r^{-2s}\tail\big(w-(w)_{B_r(\bar x)};B_r(\bar x)\big)\leq c\,R^{-2s} \tail\big(w-(w)_{B_{R/4}(\bar x)};B_{R/4}(\bar x)\big)\\
+c\,R^{-2s}\mint_{B_{R/4}(\bar x)}{\big|w-(w)_{B_{R/4}(\bar x)}\big|}\dx+c\int_r^{R/4}\mu^{-2s}\mint_{B_\mu(\bar x)}{\big|w-(w)_{B_\mu(\bar x)}\big|}\dx\,\frac{d\mu}{\mu}\,.
\end{multline*}
By the $\beta$-H\"older continuity of $w$, and since $B_{R/4}(\bar x)\subseteq B_R(x_0)$, we have
\[
    \mint_{B_{R/4}(\bar x)}{\big|w-(w)_{B_{R/4}(\bar x)}\big|}\dx\leq [w]_{C^{0,\beta}(B_{R}(x_0))}R^\beta\,,
\]
and
\[
   \int_r^{R/4}\mu^{-2s}\mint_{B_\mu(\bar x)}\big|w-(w)_{B_\mu(\bar x)}\big|\dx\,\frac{d\mu}{\mu}  \leq  [w]_{C^{0,\beta}(B_{R}(x_0))}\int_r^R\mu^{\beta-2s} \,\frac{d\mu}{\mu}\,.
\]
As $B_{R/4}(\bar x)\subseteq B_{R}(x_0)$, and $R-|\bar x-x_0|\geq R/2$, we further estimate via \eqref{tail:3}
\begin{align*}
      \tail\big(w-(w)_{B_{R/4}(\bar x)};B_{R/4}(\bar x)\big) &\leq  c\tail\big(w-(w)_{B_{R}(x_0)};B_{R}(x_0)\big)+c\mint_{B_{R}(x_0)}\big|w-(w)_{B_{R}(x_0)}\big|\dx\notag\\
    &\leq  c\tail\big(w-(w)_{B_{R}(x_0)}; B_{R}(x_0)\big)+c\,[w]_{C^{0,\beta}(B_{R}(x_0))}\,R^\beta\,;
\end{align*}
merging the content of these displays, we get the desired estimate.
\end{proof}

\subsection{Lipschitz continuity of solutions} 
We are now ready to prove the Lipschitz regularity of solutions to \eqref{eq1}. The estimates established above provide the two ingredients needed for the iteration: the excess decay estimate \eqref{st:comp8}, which controls the oscillation of $Du$ at smaller scales in terms of its oscillation at larger scales, and the tail estimate of Proposition \ref{prop:newtail}, which allows us to handle the nonlocal contribution with sufficient precision by using the local H\"older estimate already obtained in Theorem \ref{partial.Holder}. We shall combine these estimates with a suitable choice of the H\"older exponent $\beta_0$ in order to absorb the lower-order terms and obtain a uniform bound for the gradient excess.

\begin{proof}[Proof of Theorem \ref{thm:lip}]
As in the previous subsections, we shall consider $\bar x\in B_{r_{x_0}/2}(x_0)$ and a radius $0<r\leq r_{x_0}/4$ so that  $B_{2r}(\bar x)\subset B_{r_{x_0}}(x_0)\Subset \Omega_u$. As remarked at the beginning of this section (see the discussion preceding equation~\eqref{usholder}), we are working under assumptions that are stronger than those leading to Theorem \ref{partial.Holder}; thus by \eqref{mod:cont.u}, \eqref{mod:cont.x} and the monotonicity of $\omegau$
\begin{align}\label{AAx:0}
{\big\|A(\cdot,u(\,\cdot\,))-\bar A\big\|}_{L^\infty(B_r(\bar x))}&\leq \sup_{x,y \in B_r(\bar x)} \big|A(x, u(x)) - A(y, u(y))\big|\notag \\
&\leq  \sup_{x,y \in B_r(\bar x)} \big|A(x, u(x)) - A(y, u(x))\big| +  \sup_{x,y \in B_r(\bar x)}\big|A(y, u(x)) - A(y, u(y))\big|  \notag\\
&\leq \Lambda \bigg[\omegax(r)+\omegau\Big(\sup_{x,y \in B_r(\bar x)}|u(y)-u(x)|\Big) \bigg]\,,
\end{align}
and then using that $u\in C^{0,\beta_0}(B_r(\bar x))$, we get
\begin{equation}\label{AAx:2}
    {\big\|A(\cdot,u(\,\cdot\,))-\bar A\big\|}_{L^\infty(B_r(\bar x))}\leq \Lambda\Big[ \omegax(r) + \omegau\big( [u]_{C^{0,\beta_0}(B_{r_{x_0}}(x_0))}r^{\beta_0}\big)\Big]\,.
\end{equation}
If we now define
\begin{equation}\label{def:omegabeta}
    \omega_{\beta_0}(r):=\omegax(r)+\omegau\big([u]_{C^{0,\beta_0}(B_{r_{x_0}}(x_0))}r^{\beta_0}\big)+r^{2(1-s)}\,,
\end{equation}
then combining \eqref{st:comp7} and \eqref{AAx:2} yields
\begin{multline}\label{st:comp8}
E\big(Du;B_\vrho(\bar x)\big)\leq c\bigg[ \frac\vrho r+\Big( \frac{r}{\vrho}\Big)^{n/2}\omega_{\beta_0}(r)\bigg] E\big(Du;B_{2r}(\bar x)\big)+c\,\Big(\frac r{\vrho}\Big)^{n/2}\omega_{\beta_0}(r)\big|{(Du)}_{B_{2r}(\bar x)}\big|\\
+c\,\Big( \frac r{\vrho}\Big)^{n/2}r\bigg(\mint_{B_r(\bar x)}|f|^{2_*}\dx\bigg)^{1/2_*}+c\,\Big(\frac r{\vrho}\Big)^{n/2}r^{1-2s}\tail\big(u-(u)_{B_r(\bar x)};B_r(\bar x)\big)\,,
\end{multline}
for all $0<\vrho\leq r\leq r_{x_0}/4$, where $c=c(\data,\max\{\nu,1\})$. 

\vs

Let us now consider a radius $\bar r\leq r_{x_0}/4$ (that we use as smallness parameter) and a parameter $\tau\in (0,1/16)$ to be chosen later and we specialize the previous estimate to points $\bar x\in B_{\bar r}(x_0)$. We build the shrinking sequence of radii $\vrho_i=\tau^i r$ for $r\leq \bar r$ and $i\in\N_0$, and subsequently set
\begin{equation}\label{def:neew}
   B_i\equiv B_{\vrho_i}(\bar x),\quad k_i\coloneqq\big|{(Du)}_{B_i}\big|\,.
\end{equation}
First we observe that for $m\in\N_0$, we have by telescopic summation (see \cite[Equation~(4.5)]{BCV})
\begin{equation}\label{kmi:1}
    k_{m+1}\leq k_0+\sum_{i=0}^m \mint_{B_{i+1}}\big|Du-{(Du)}_{B_i}\big|\dx\leq k_0+ \tau^{-n}\sum_{i=0}^m E(Du; B_i)\,,
\end{equation}
enlarging the domain of integration, using H\"older's inequality and the definition of $E(Du; B_i)$. Next, by applying \eqref{st:comp8} with the choice $\vrho=\vrho_{i+1}$ and $r=\vrho_i/2$, we obtain
\begin{multline}\label{Ai:temp}
E(Du; B_{i+1})\leq c\Big[\tau+\tau^{-n/2}\,\omega_{\beta_0}(\bar r) \Big]E(Du; B_i) \\
    + c\, \tau^{-n/2}\bigg[\omega_{\beta_0}(\vrho_{i})k_{i}+\vrho_{i}\bigg( \mint_{B_{i}}|f|^{2_*}\dx\bigg)^{1/2_*}+\vrho_{i}^{1-2s}\tail\big(u-(u)_{B_{\vrho_i/2}(\bar x)};B_{\vrho_i/2}(\bar x)\big)\bigg]
\end{multline}
where $c=c(\data,\max\{\nu,1\})$.   We now first choose  $\tau=\tau(\data,\max\{\nu,1\})\in (0,1/16)$, and then choose the radius $\bar r\leq r_{x_0}/4$, depending on $\data,\max\{\nu,1\}$ and also on $\omegax(\,\cdot\,)$, $\omegau(\,\cdot\,)$, $[u]_{C^{0,\beta_0}(B_{r_{x_0}}(x_0))}$ sufficiently small so that
\begin{equation}\label{lorichiamo}
    c\Big[\tau+\tau^{-n/2}\,\omega_{\beta_0}(\bar r) \Big]\leq \frac{1}{2}\,.
\end{equation}
In particular, recalling that $\beta_0=\beta_0(s)$ is fixed, by definition of $\omega_{\beta_0}$ in \eqref{def:omegabeta}, the monotonicity of $\omegau$ and the H\"older estimate \eqref{usholder}, the radius $\bar r$ depends on $\data,\max\{\nu,1\},\omegax(\,\cdot\,),\omegau(\,\cdot\,)$,  $\vrho_{x_0}$, ${\|I_{1,\chi}^f(\cdot,1)\|}_{L^\infty(\Omega)}$ and on an upper bound on $E(u;B_{\vrho_{x_0}}(x_0))$ and the tail term $\tail(u-(u)_{B_{\vrho_{x_0}}(x_0)};B_{\vrho_{x_0}}(x_0))$.

\vs

Next, by using \eqref{lorichiamo}, equation~\eqref{Ai:temp} becomes
\begin{align}\label{st:comp9}
E(Du; B_{i+1}) &\leq \frac12E(Du; B_i) + c\bigg[\omega_{\beta_0}(\vrho_{i})k_{i}+\vrho_{i}\bigg( \mint_{B_{i}}|f|^{2_*}\dx\bigg)^{1/2_*}+\vrho_{i}^{1-2s}\tail\big(u-(u)_{B_{\vrho_i/2}(\bar x)};B_{\vrho_i/2}(\bar x)\big)\bigg]\notag\\
    &\leq \frac12E(Du; B_i) + c\bigg[\omega_{\beta_0}(\vrho_{i})k_{i}+\vrho_{i}\bigg( \mint_{B_{i}}|f|^{2_*}\dx\bigg)^{1/2_*}+\vrho_{i}\Big[ r_{x_0}^{-2s}\tail\big(u-(u)_{B_{r_{x_0}}(x_0)};B_{r_{x_0}}(x_0)\big)\notag\\
    &\pushright{+\,\frac1{|\beta_0-2s|}\vrho_i^{\beta_0-2s}{[u]}_{C^{0,\beta_0}(B_{r_{x_0}}(x_0))}+\Big(1+\frac1{|\beta_0-2s|}\Big)r_{x_0}^{\beta_0-2s}{[u]}_{C^{0,\beta_0}(B_{r_{x_0}}(x_0))}\Big]\bigg]}
\end{align}
for all $i\in\mathbb N_0$, with $c=c(\data,\max\{\nu,1\})$. In the last line we used Proposition~\ref{prop:newtail} with $r=\vrho_i/2$ and $R=r_{x_0}$; moreover, since $\beta_0\neq 2s$, we computed
\[
\int_{\vrho_i/2}^{r_{x_0}} \mu^{\beta_0-2s}\,\frac{d\mu}{\mu} \leq \frac{r_{x_0}^{\beta_0-2s}+\vrho_i^{\beta_0-2s}}{|\beta_0-2s|}\,.
\] 
By summing \eqref{st:comp9} for $i\in\{0,1,\dots,m-1\}$ for $m\in\N$ and performing simple algebraic manipulations, similar to those from Proof of Theorem \ref{cont.t}, we get
\begin{multline*}
\sum_{i=0}^mE(Du; B_i)\leq E(Du; B_0)+ \frac{1}{2}\sum_{i=0}^{m}E(Du; B_i)+c\sum_{i=0}^{m-1} \omega_{\beta_0}(\vrho_i)k_i+c\sum_{i=0}^\infty \vrho_i\bigg( \mint_{B_i}|f|^{2_*}\dx\bigg)^{1/2_*}
\\
+c\,{[u]}_{C^{0,\beta_0}(B_{r_{x_0}}(x_0))}\sum_{i=0}^\infty\vrho_i^{1+\beta_0-2s} +c\Big[r_{x_0}^{-2s} \tail\big(u-(u)_{B_{r_{x_0}}(x_0)} ;B_{r_{x_0}}(x_0)\big)+r_{x_0}^{\beta_0-2s}{[u]}_{C^{0,\beta_0}(B_{r_{x_0}}(x_0))}\Big]\sum_{i=0}^\infty \vrho_i\,,
\end{multline*}
with $c=c(\data,\max\{\nu,1\})$. Reabsorbing the first term of the right-hand side, and using that 
\begin{equation*}
    \sum_{i=0}^\infty \vrho_i=\sum_{i=0}^\infty\tau^ir=c(\data,\max\{\nu,1\})r,\quad\text{and similarly}\quad \sum_{i=0}^\infty \vrho_i^{1+\beta_0-2s}=c(\data,\max\{\nu,1\})r^{1+\beta_0-2s}
\end{equation*}
 due to $\beta_0=\beta_0(s)$ and \eqref{cond:beta}, we find
\begin{multline}\label{st:comp11}
\sum_{i=0}^mE(Du; B_i)\leq 2E(Du; B_0)+c\sum_{i=0}^{m-1} \omega_{\beta_0}(\vrho_i)k_i+c\sum_{i=0}^\infty \vrho_i\bigg( \mint_{B_i}|f|^{2_*}\dx\bigg)^{1/2_*}\\
+c\,{[u]}_{C^{0,\beta_0}(B_{r_{x_0}}(x_0))}r\big[r^{\beta_0-2s}+r_{x_0}^{\beta_0-2s}\big]+c\,r\,r_{x_0}^{-2s} \tail\big(u-(u)_{B_{r_{x_0}}(x_0)}; B_{r_{x_0}}(x_0)\big)\,,
\end{multline}
with constant $c=c(\data,\max\{\nu,1\})$. Combining \eqref{st:comp11} and \eqref{kmi:1} with \eqref{est.potential.sum} for $\bar\jmath=0$, $\alpha=1$, we get
\begin{align}\label{iter:km}
k_{m+1}&\leq  k_0+c\,E\big(Du; B_r(\bar x)\big)+c\sum_{i=0}^{m-1} \omega_{\beta_0}(\vrho_i)k_i+c{\big\|{\bf I}^f_{1,\chi}(\,\cdot\,,r_{x_0})\big\|}_{L^\infty (B_{r_{x_0}}(x_0))}\notag\\
&\qquad\qquad\qquad+c\,{[u]}_{C^{0,\beta_0}(B_{r_{x_0}}(x_0))}r\big[\bar r^{\beta_0-2s}+r_{x_0}^{\beta_0-2s}\big]+c\,r\,r_{x_0}^{-2s}  \tail\big(u-(u)_{B_{r_{x_0}}(x_0)} ;B_{r_{x_0}}(x_0)\big)\notag\\
&\eqqcolon M_0+c\sum_{i=0}^{m-1} \omega_{\beta_0}(\vrho_i)\,k_i\,,
\end{align}
with the obvious definition for $M_0$, and with $c=c(\data,\max\{\nu,1\})$. We now claim that 
\begin{equation}\label{fin:claim}
k_m\leq 2M_0\quad \text{for all $m\in \N_0$}\,,
\end{equation}
provided that we choose the radius $\bar r>0$ sufficiently small.

\vs

We proceed by induction: the case $m=0$ is trivial by the definition of $M_0$. Let us suppose the claim is true for $m=0,1,\dots,m_0$ for some $m_0\in\N_0$, and prove it for $m_0+1$. By \eqref{iter:km} and the inductive hypothesis, 
\begin{align}\label{k:m0}
        k_{m_\smallzero+1}&\leq M_0+c\sum_{i=0}^{m_\smallzero} \omega_{\beta_0}(\vrho_i)k_i\leq M_0\Big[1+2\bar c\sum_{i=0}^\infty \omega_{\beta_0}(\vrho_i)\Big]\,,
\end{align}
with $\bar c=\bar c(\data,\max\{\nu,1\})$. Using the standard estimate \eqref{Dini.dyadic}, recalling that 
\[
\beta_0=\beta_0(s),\qquad \tau=\tau(\data,\max\{\nu,1\})\leq 1/16\,,
\]
by the expression of \eqref{def:omegabeta}, using \eqref{Dini.dyadic} and the Dini continuity of $\omegax(\,\cdot\,),\omegau(\,\cdot\,)$, we may reduce the size of $\bar r$ in such a way that
\begin{equation}\label{radius:chiamo}
\sum_{i=0}^\infty \omega_{\beta_0}(\vrho_i)\leq \frac{1}{\log(1/\tau)}\,\int_0^{2\bar r}\omegax(\vrho)\,\frac{d\vrho}{\vrho}+\frac{1}{\beta_0\log(1/\tau)}\,\int_0^{2[u]_{C^{0,\beta_0}(B_{r_{x_0}}(x_0))}\bar r^{\beta_0}}\omegau(\vrho)\,\frac{d\vrho}{\vrho}+ \frac{\bar r^{2(1-s)}}{1-\tau^{2(1-s)}}\leq  \frac{1}{2\bar c}\,.
\end{equation}
Combining \eqref{radius:chiamo} with \eqref{k:m0}, we deduce $k_{m_\smallzero+1}\leq 2M_0$, which proves the claim \eqref{fin:claim}.

\vs

Now, for all Lebesgue points $\bar x\in B_r(x_0)$ of $Du$, by \eqref{fin:claim} we have
\begin{equation}\label{almost:final}
    |Du(\bar x)|=\lim_{m\to \infty}k_m\leq 2M_0\qquad\implies \qquad {\|Du\|}_{L^\infty(B_{r}(x_0))}\leq 2M_0\,.
\end{equation}
Recalling \eqref{def:neew} and \eqref{media}, for all $\bar x\in B_r(x_0)$, we estimate 
\begin{equation*}
    k_0+E\big(Du;B_{r}(\bar x)\big)\leq c(n)\bigg(\mint_{B_{2r}(x_0)}|Du|^2\dx \bigg)^{1/2}
\end{equation*}
since $B_{r}(\bar x)\subset B_{2r}(x_0)$; we then insert the estimate \eqref{usholder}  into \eqref{almost:final},  expand the expression \eqref{iter:km} of $M_0$, and use that, since \eqref{dep:gamma} holds, $1+\beta_0-2s>0$ and $r\leq r_{x_0}\leq \vrho_{x_0}\leq 1$,
\[
r\big[r^{\beta_0-2s}+r_{x_0}^{\beta_0-2s}\big]\leq r^{1+\beta_0-2s}+r_{x_0}^{1+\beta_0-2s}\leq2
\]
and similarly
\[
   \vrho_{x_0}^{\gamma-\beta_0-2s}r\big[r^{\beta_0-2s}+r_{x_0}^{\beta_0-2s}\big]\leq2\vrho_{x_0}^{\gamma-4s} \leq 2\vrho_{x_0}^{-2s}\,.
\]
This yields
\begin{align*}
{\|Du\|}_{L^\infty(B_{r}(x_0))}&\leq c\,r^{-n/2}{\|Du\|}_{L^2(B_{2r}(x_0))}+c\,\vrho_{x_0}^{-n/2}{\|Du\|}_{L^2(B_{\vrho_{x_0}}(x_0))}
    \\
    &\qquad+c\,{\big\|{\bf  I}_{1,\chi}^f(\cdot,r_{x_0})\big\|}_{L^\infty(B_{r_{x_0}}(x_0))}+c\,r\,r_{x_0}^{-2s}  \tail\big(u-(u)_{B_{r_{x_0}}(x_0)} ;B_{r_{x_0}}(x_0)\big)
    \\
    &\qquad\qquad+c\,r\vrho_{x_0}^{-2s} \tail\big(u-(u)_{B_{\vrho_{x_0}}(x_0)};B_{\vrho_{x_0}}(x_0)\big)\,,
\end{align*}
with $c=c(\data,\max\{\nu,1\})$, and with $r\leq \bar r\leq r_{x_0}/4$ fulfilling \eqref{lorichiamo} and \eqref{radius:chiamo}. This completes the proof of \eqref{quant:lip} (with $\bar r_{x_0} := \bar r$).
\end{proof}

\begin{remark}
{\rm Taking advantage of the Lipschitz continuity of $u$, we can improve \eqref{st:comp8}. As we want to make use of Theorem \ref{thm:lip} in \eqref{AAx:0}, we start from points $\bar x\in B_{\bar r_{x_0}/2}(x_0)$ and radii $r\leq \bar r_{x_0}/4$ and get
\begin{equation}\label{AAx:3}
{\big\|A(\cdot,u(\,\cdot\,))-\bar A\big\|}_{L^\infty(B_r(\bar x))}\leq \Lambda\Big[\omegax(r)+\omegau\big({[u]}_{C^{0,1}(B_{\bar r_{x_0}}(x_0))}r\big)\Big]\,,
\end{equation}
with $\bar A$ as in \eqref{def:Ax}. Additionally,  we can apply Proposition \ref{prop:newtail} with $\beta=1$, thus getting
\begin{multline}\label{temp:liptail}
r^{1-2s}\tail\big(u-(u)_{B_r(\bar x)};B_r(\bar x)\big)\leq  c\,r\, \bar r_{x_0}^{-2s} \tail\big(u-(u)_{B_{\bar r_{x_0}}(x_0)} ;B_{\bar r_{x_0}}(x_0)\big)\\
+c\,r\bigg[\bar r_{x_0}^{1-2s}+\int_r^{\bar r_{x_0}}\mu^{1-2s} \,\frac{d\mu}{\mu}\bigg] {\|Du\|}_{L^\infty(B_{\bar r_{x_0}}(x_0))}\,.
\end{multline}
Substituting this into \eqref{st:comp7}, and taking into account that $\big|{(Du)}_{B_{2r}(\bar x)}\big|\leq {\|Du\|}_{L^\infty(B_{\bar r_{x_0}}(x_0))}$, gives
\begin{multline}\label{st:comp:lip}
E\big(Du;B_\vrho(\bar x)\big)\leq c\bigg[ \frac{\vrho}{r}+\Big(\frac{r}{\vrho}\Big)^{n/2}\omega_1(r)\bigg]E\big(Du;B_{2r}(\bar x)\big)+c\,\Big(\frac{r}{\vrho}\Big)^{n/2}\bigg\{\omega_1(r){\|Du\|}_{L^\infty(B_{\bar r_{x_0}}(x_0))}\\
+c\,r\bigg( \mint_{B_r(\bar x)}|f|^{2_*}\dx\bigg)^{1/2_*}+c\,r\, \bar r_{x_0}^{-2s} \tail\big(u-(u)_{B_{\bar r_{x_0}}(x_0)};B_{\bar r_{x_0}}(x_0)\big)\\
+c\,r\bigg[\bar r_{x_0}^{1-2s}+\int_r^{\bar r_{x_0}}\mu^{1-2s} \,\frac{d\mu}{\mu}\bigg]{\|Du\|}_{L^\infty(B_{\bar r_{x_0}}(x_0))}\bigg\}\,,
\end{multline}
for $0<\vrho\leq r\leq \bar r_{x_0}/4$, with $c=c(\data,\max\{\nu,1\})$, and where $\omega_1(\,\cdot\,)$ is defined in \eqref{def:omegabeta} with $\beta_0=1$, that is
\begin{equation*}
    \omega_1(r):=\omegax(r)+\omegau\big([u]_{C^{0,1}(B_{\bar r_{x_0}}(x_0))}r\big)+r^{2(1-s)}.
\end{equation*}
We recall that $u$ is the continuous representative of the Lebesgue class in $\Omega_u$.}
\end{remark}

\subsection{Gradient continuity of solutions} 
We prove here the continuity of the gradient under the additional assumption that the Riesz potential of the datum tends to zero locally uniformly as the radius vanishes. The proof is divided into two steps: first we show that the local excess of the gradient vanishes locally uniformly at small scales. Then, by summing the decay estimate \eqref{st:comp8} along a dyadic sequence of radii, we prove that the corresponding averages of the gradient form a Cauchy sequence; the conclusion then follows as in the proof of Theorem~\ref{cont.t}. Again, $x_0\in\Omega_u$ and $0<\bar r_{x_0}\leq r_{x_0}/4\leq \vrho_{x_0}/16$ are such that $B_{r_{x_0}}(x_0)\Subset \Omega_u$, $B_{2\vrho_{x_0}}(x_0)\Subset \Omega$ and the local estimate \eqref{quant:lip} holds.

\begin{proof}[Proof of Theorem \ref{thm:gradcont}]
{\em Step 1: A VMO-type result}. We first prove a local VMO-type estimate for $Du$, uniform in $B_{\bar r_{x_0}/4}(x_0)$. Specifically, we show that for every $\e\in(0,1)$, there exists a radius $r_\e\leq \bar r_{x_0}/4$ small enough and independent of $\bar x\in B_{\bar r_{x_0}/4}(x_0)$, such that
\begin{equation}\label{VMO}
E\big(Du;B_r(\bar x)\big)=\bigg(\mint_{B_r(\bar x)}{\big|Du-{(Du)}_{B_r(\bar x)}\big|}^2\dx \bigg)^{1/2}<\e
\end{equation}
for all $r\in(0, r_\e]$ and for all $\bar x\in B_{\bar r_{x_0}/4}(x_0)$. To this end, let us fix $R_\e>0$, to be determined in the course of the proof, and let $r>0$ be arbitrary with with $0<2r\leq R_\e \leq \bar r_{x_0}/2$. By this choice and \eqref{est.potential.sum}, we have
\begin{equation*}
    r\bigg( \mint_{B_r(\bar x)}|f|^{2_*}\dx\bigg)^{1/2_*} \leq c\,{\bf I}^f_{1,\chi}(\bar x,2r)\leq c\,{\bf I}^f_{1,\chi}(\bar x,R_\e) 
\end{equation*}
and clearly $E(Du; B_{2r}(\bar x))\leq 2{\|Du\|}_{L^\infty(B_{\bar r_{x_0}}(x_0))}$ for all $\bar x\in B_{\bar r_{x_0}/4}(x_0)$.  Let us now fix $\tau_\e\in (0,1)$ to be chosen later, 
and combine the last two estimates with \eqref{st:comp:lip}, choosing $\vrho=\tau_\e r$. Recalling the monotonicity of $\omega_1(\,\cdot\,)$ and that $r\leq R_\e\leq \bar r_{x_0}/2$, we thus obtain    
\begin{multline}\label{vmo:5}
E(Du; B_{\tau_\e r}(\bar x))
\leq \tilde c\bigg[\tau_\e+\tau_\e^{-n/2}\Big[\omega_1(R_\e)+R_\e \bar r_{x_0}^{1-2s}+\sup_{0<r\leq R_\e}r\int_r^{\bar r_{x_0}}\mu^{1-2s} \,\frac{d\mu}{\mu}\Big]\bigg]{\|Du\|}_{L^\infty(B_{\bar r_{x_0}}(x_0))}\\
+\tilde c\,\tau_\e^{-n/2}{\bf I}^f_{1,\chi}(\bar x,R_\e)+\tilde c\,R_\e\, \bar r_{{x_0}}^{-2s} \tail\big(u-(u)_{B_{\bar r_{x_0}}(x_0)}; B_{\bar r_{x_0}}(x_0)\big)\,,
\end{multline}
with $\tilde c$ depending on $\data$ and $\max\{\nu,1\}$. We can now fix $\tau_\e$ and $R_\e$: for $\e>0$ fixed, we choose  
\begin{equation}\label{fix:taue}
       \tau_\e=\frac{\e}{4\tilde c\,{\|Du\|}_{L^\infty(B_{\bar r_{x_0}}(x_0))}+1}\,,
\end{equation}
and then we select $R_\e>0$ independent of $\bar x$ such that
\[
\tau_\e^{-n/2}\Big[\omega_1(R_\e)+R_\e \bar r_{x_0}^{1-2s}+\sup_{0<r\leq R_\e}r\int_r^{\bar r_{x_0}}\mu^{1-2s} \,\frac{d\mu}{\mu}\Big]{\|Du\|}_{L^\infty(B_{\bar r_{x_0}}(x_0))}\leq  \frac{\e}{4\tilde c}
\]
(this is possible since $\sup_{0<r\leq R_\e}r\int_r^{\bar r_{x_0}}\mu^{1-2s} \,\frac{d\mu}{\mu}\searrow 0$ as $R_\e\to 0$) and
\[
R_\e\,\bar r_{x_0}^{-2s} \tail\big(u-(u)_{B_{\bar r_{x_0}}(x_0)} ;B_{\bar r_{x_0}}(x_0)\big)\leq  \frac{\e}{4\tilde c}\,.
\]
Moreover, thanks to our assumption \eqref{f:riesz1}, we may further reduce $R_\e$, independently of $\bar x$, in such a way that
\begin{equation}\label{fix:Ifeps}
    \tau_\e^{-n/2}{\bf I}^f_{1,\chi}(\bar x,R_\e)<\frac{\e}{4\tilde c}\,;
\end{equation}
merging the content of equations \eqref{fix:taue}--\eqref{fix:Ifeps} with \eqref{vmo:5}, we obtain
\[
E\big(Du; B_{\tau_\e r}(\bar x)\big)<\e\quad\text{for all $r\leq R_\e/2$\,, for all $\bar x\in B_{\bar r_{x_0}/4}(x_0)$}\,,
\]
that is \eqref{VMO} with $r_\e=\tau_\e R_\e/2$.

\vs

\noindent\textit{Step 2: $Du$ is continuous.} 
We let  $\bar x\in B_{\bar r_{x_0}/4}(x_0)$, and let $\tau\in (0,1/2)$ a parameter and $R_1\leq \bar r_{x_0}/2$  a radius both be determined later; we set, for $i\in\N$ $\vrho_i=\tau^iR_1$, $B_i(\bar x)=B_{\vrho_i}(\bar x)$. In order to prove that $Du$ is continuous, it suffices to show that 
\begin{equation}\label{claim:cauchy}
\bar x\mapsto    {\big\{{(Du)}_{B_i(\bar x)}\big\}}_{i\in\N_0}\text{ is a Cauchy sequence, uniformly in $B_{\bar r_{x_0}/4}(x_0)$}
\end{equation}
(see indeed the argument at the end of the proof of Theorem \ref{cont.t}, applied to $Du$).

\vs

 Let us prove \eqref{claim:cauchy}. First observe that, for $1\leq k<m$, using the definition of $\vrho_i$ and H\"older's inequality 
\begin{equation}\label{cauchy:1}
    \big|{(Du)}_{B_k(\bar x)}-{(Du)}_{B_m(\bar x)}\big| \leq \sum_{i=k}^{m-1}\mint_{B_{i+1}}\big|Du-{(Du)}_{B_i(\bar x)}\big|\dx\leq \tau^{-n}\sum_{i=k}^{m-1}E\big(Du;B_i(\bar x)\big)\,,
\end{equation}
see for instance \eqref{uk:cauchy}. We then make use of \eqref{st:comp:lip} with $\vrho=\vrho_{i+1}$ and $r=\vrho_i/2$ for $i\in\N_0$, together with $|{(Du)}_{B_i(\bar x)}|\leq {\|Du\|}_{L^\infty(B_{\bar r_{x_0}}(x_0))}$ and the fact that $\omega_1$ is monotone increasing, so that we find
\begin{align}\label{cauchy:2}
    E\big(Du;B_{i+1}(\bar x)\big)&\leq \tilde c\,\Big\{\tau+\tau^{-n/2}\omega_1(\vrho_i) \Big\}E\big( Du; B_i(\bar x)\big)+c\,\tau^{-n/2}\omega_1(\vrho_i){\|Du\|}_{L^\infty(B_{\bar r_{x_0}}(x_0))}\notag
     \\
     &\qquad+c\,\tau^{-n/2}\vrho_i\bigg( \mint_{B_i(\bar x)}|f|^{2_*}\dx\bigg)^{1/2_*}+c\,\tau^{-n/2}\vrho_i\,\bar r_{x_0}^{-2s}\, \tail\big(u-(u)_{B_{\bar r_{x_0}}(x_0)}; B_{\bar r_{x_0}}(x_0)\big)\notag
     \\
     &\qquad\qquad+c\,\tau^{-n/2}\vrho_i\, \bar r_{x_0}^{1-2s}{\|Du\|}_{L^\infty(B_{\bar r_{x_0}}(x_0))}+c\,\tau^{-n/2}\vrho_i\int_{\vrho_i/2}^{\bar r_{x_0}}\mu^{1-2s} \,\frac{d\mu}{\mu}\,{\|Du\|}_{L^\infty(B_{\bar r_{x_0}}(x_0))}\,,
\end{align}
where $c,\tilde c$ depend on $\data$ and $\max\{\nu,1\}$. We now fix $\tau=\tau(\data,\max\{\nu,1\})$  and then $R_1>0$, both independent of $\bar x$, such that 
\begin{equation*}
    \tilde c\,\Big\{\tau+\tau^{-n/2}\omega_1(R_1) \Big\}\leq \frac{1}{2}\,.
\end{equation*}
Combining this information with \eqref{cauchy:2}, using $\omega_1(\vrho_i)\leq \omega_1(R_1)$ and recalling that now $\tau$ depends on $\data$ gives
\begin{multline*}
    E\big(Du;B_{i+1}(\bar x)\big)\leq \frac{1}{2}\,E\big(Du;B_i(\bar x)\big)+c\,\omega_1(\vrho_i){\|Du\|}_{L^\infty(B_{\bar r_{x_0}}(x_0))}\\
    +c\,\vrho_i\bigg(\mint_{B_i(\bar x)}|f|^{2_*}\dx\bigg)^{1/2_*}+c\,\vrho_i\int_{\vrho_{i+1}}^{\bar r_{x_0}}\mu^{1-2s} \,\frac{d\mu}{\mu}\,{\|Du\|}_{L^\infty(B_{\bar r_{x_0}}(x_0))}\\
    +c\,\vrho_i\bar r_{x_0}^{-2s}\Big[\tail\big(u-(u)_{B_{\bar r_{x_0}}(x_0)} ;B_{\bar r_{x_0}}(x_0)\big)+\bar r_{x_0}{\|Du\|}_{L^\infty(B_{\bar r_{x_0}}(x_0))}\Big]\,.
\end{multline*}
We take the sum over $i=k,\dots,m$ and, after changing variables and enlarging summations on the right-hand side, we get
\begin{multline*}
\sum_{i=k+1}^{m+1} E\big(Du;B_i(\bar x)\big)\leq \frac{1}{2}\sum_{i=k}^{m}E\big(Du;B_i(\bar x)\big)+c{\|Du\|}_{L^\infty(B_{\bar r_{x_0}}(x_0))}\,\sum_{i=k}^\infty\,\omega_1(\vrho_i)\\
+c\sum_{i=k}^m\vrho_i\bigg( \mint_{B_i(\bar x)}|f|^{2_*}\dx\bigg)^{1/2_*}+c\,{\|Du\|}_{L^\infty(B_{\bar r_{x_0}}(x_0))} \sum_{i=k}^\infty\vrho_i\int_{\vrho_{i+1}}^{\bar r_{x_0}}\mu^{1-2s} \,\frac{d\mu}{\mu}\\
+c\,\bar r_{x_0}^{-2s}\Big[\tail\big(u-(u)_{B_{\bar r_{x_0}}(x_0)} ;B_{\bar r_{x_0}}(x_0)\big)+\bar r_{x_0}{\|Du\|}_{L^\infty(B_{\bar r_{x_0}}(x_0))}\Big]\sum_{i=k}^\infty\,\vrho_i\,;
\end{multline*}
reabsorbing the first term on the right-hand side, letting $m\to+\infty$, using the definition of the radii $\vrho_i$, the choice of $\tau=\tau(\data,\max\{\nu,1\})$ and the fact that $R_1\leq\bar r_{x_0}$, the above inequality entails
\begin{multline}\label{cauchy:4}
    \sum_{i=k}^\infty E\big(Du;B_i(\bar x)\big)\leq  2E\big( Du;B_k(\bar x)\big)+c\,{\|Du\|}_{L^\infty(B_{\bar r_{x_0}}(x_0))}\sum_{i=k}^\infty\omega_1(\vrho_i)+c\sum_{i=k}^\infty\vrho_i\bigg( \mint_{B_i(\bar x)}|f|^{2_*}\dx\bigg)^{1/2_*}\\
+\,c\,\tau^k\Big[\bar r_{x_0}^{1-2s} \tail\big(u-(u)_{B_{\bar r_{x_0}}(x_0)}; B_{\bar r_{x_0}}(x_0)\big)+\bar r_{x_0}^{2(1-s)}{\|Du\|}_{L^\infty(B_{\bar r_{x_0}}(x_0))}\Big]\\
+c\,{\|Du\|}_{L^\infty(B_{\bar r_{x_0}}(x_0))} \sum_{i=k}^\infty\vrho_i\int_{\vrho_{i+1}}^{\bar r_{x_0}}\mu^{1-2s} \,\frac{d\mu}{\mu}\,.
\end{multline}
Now we recall that, by \eqref{Dini.dyadic} ($\omega_1(\,\cdot\,)$ is indeed a modulus of continuity) and \eqref{est.potential.sum},
\[
    \sum_{i=k}^\infty\omega_1(\vrho_i) \leq c\int_0^{\vrho_{k-1}}\omega_1(t)\frac{dt}{t}\qquad\text{and}\qquad \sum_{i=k}^\infty\vrho_i\bigg( \mint_{B_i(\bar x)}|f|^{2_*}\dx\bigg)^{1/2_*}\leq c\,{\bf I}^f_{1,\chi}(\bar x,\vrho_{k-1})\,,
\]
and we notice that
\[
\sum_{i=k}^\infty\vrho_i\int_{\vrho_{i+1}}^{\bar r_{x_0}}\mu^{1-2s} \,\frac{d\mu}{\mu}\leq c(s,\tau)\tau^{k\min\{1,2(1-s)\}}\bar r_{x_0}^{2(1-s)}\Big[\log\frac{\bar r_{x_0}}{R_1}+k\Big]\,.
\]
Indeed, if $s\neq 1/2$, then
\[
\int_{\vrho_{i+1}}^{\bar r_{x_0}}\mu^{1-2s} \,\frac{d\mu}{\mu}=\frac{\bar r_{x_0}^{1-2s}-\vrho_{i+1}^{1-2s}}{1-2s}\leq c(s,\tau) \max\big\{\bar r_{x_0}^{1-2s},\vrho_i^{1-2s}\big\}
\]
so that, recalling that $\vrho_i\leq R_1\leq\bar r_{x_0}/2$ and that $2(1-s)>0$,
\[
\sum_{i=k}^\infty\vrho_i\int_{\vrho_{i+1}}^{\bar r_{x_0}}\mu^{1-2s} \,\frac{d\mu}{\mu}\leq c(s)\sum_{i=k}^\infty\vrho_i\max\big\{\bar r_{x_0}^{1-2s},\vrho_i^{1-2s}\big\}\leq c(s,\tau)\,\bar r_{x_0}^{2(1-s)}\max\big\{\tau^k,\tau^{2k(1-s)}\big\}\,;
\]
if $s=1/2$, then
\[
\int_{\vrho_{i+1}}^{\bar r_{x_0}} \mu^{-1}\,d\mu = \log\frac{\bar r_{x_0}}{\tau^{i+1}R_1} = \log\frac{\bar r_{x_0}}{R_1}+ (i+1)\log\frac1\tau\,, 
\]
so that
\[
\sum_{i=k}^\infty\vrho_i\int_{\vrho_{i+1}}^{\bar r_{x_0}}\mu^{-1} \,d\mu\leq c(\tau)R_1\sum_{i=k}^{\infty}\tau^i\Big(\log\frac{\bar r_{x_0}}{R_1}+ i+1\Big)\leq  c\,\bar r_{x_0}\tau^k\Big[\log\frac{\bar r_{x_0}}{R_1}+k+1\Big]\,.
\]
Inserting this information into \eqref{cauchy:4} yields
\begin{multline}\label{cauchy:4.5}
    \sum_{i=k}^\infty E\big(Du;B_i(\bar x)\big)\leq  2E\big(Du;B_k(\bar x)\big)+\tilde c\,{\|Du\|}_{L^\infty(B_{\bar r_{x_0}}(x_0))}\int_0^{\vrho_{k-1}}\omega_1(t)\,\frac{dt}{t}+\tilde c\,{\bf I}^f_{1,\chi}(\bar x,\vrho_{k-1})\\
    +c\,\bar r_{x_0}^{1-2s} \bigg[ \tau^k \tail\big(u-(u)_{B_{\bar r_{x_0}}(x_0)} ;B_{\bar r_{x_0}}(x_0)\big)+\tau^{k\min\{1,2(1-s)\}}\bar r_{x_0}\Big[\log\frac{\bar r_{x_0}}{R_1}+k\Big]{\|Du\|}_{L^\infty(B_{\bar r_{x_0}}(x_0))}\bigg]\,,
\end{multline}
with $c,\tilde c$ depending on $\data$ and $\max\{\nu,1\}$. By {\em Step 1}, see \eqref{VMO}, we have that
\[
    \lim_{k\to \infty} E\big(Du;B_k(\bar x)\big)=0\quad\text{uniformly in $\bar x\in B_{\bar r_{x_0}/4}(x_0)$}\,.
\]
Using this information, the Dini continuity of $\omega_1(\,\cdot\,)$, \eqref{f:riesz1}, and since $\tau=\tau(\data,\max\{\nu,1\})\in (0,1/16)$, from \eqref{cauchy:4.5} we deduce that for any $\e>0$ we may find $k_\e>0$, independent of $\bar x$, such that
\begin{equation*}
    \sum_{i=k}^\infty\sup_{\bar x\in B_{\bar r_{x_0}/4}(x_0)} E\big(Du;B_i(\bar x)\big)<\tau^{n}\e\quad \text{for all $k\geq k_\e$}\,;
\end{equation*}
substituting this estimate into \eqref{cauchy:1} gives
\begin{equation*}
    \big|{(Du)}_{B_k(\bar x)}-{(Du)}_{B_m(\bar x)}\big|<\e\quad \text{for all $k_\e\leq k<m$ and all $\bar x\in B_{\bar r_{x_0}/4}(x_0)$}\,,
\end{equation*}
so that \eqref{claim:cauchy} is proved, thus completing the proof.
\end{proof}

\subsection{Gradient H\"older continuity} Carefully tracking the dependences in the previous proof gives the proof of Theorem \ref{thm:c1b}, once the data are sufficiently regular.

\begin{proof}[Proof of Theorem \ref{thm:c1b}]
We fix an exponent \(\chi\) such that \(2_*<\chi<n\); since \(n<n/(1-\beta)\), this choice is admissible in \eqref{density.Marc} with \(m=n/(1-\beta)\). By \eqref{ass:f} and \eqref{density.Marc}, we have for $\bar x\in B_{r_{x_0}}(x_0)$ and $r\leq r_{x_0}$
\begin{equation}\label{f:mars}
    r\bigg( \mint_{B_r(\bar x)}|f|^{2_*}\dx\bigg)^{1/2_*}+r\bigg( \mint_{B_r(\bar x)}|f|^{\chi}\dx\bigg)^{1/\chi}\leq c(n,\beta)r^{\beta}{\|f\|}_{\M^{n/(1-\beta)}(B_r(\bar x))}\,,
\end{equation}
which in particular implies
\[
    {\big\|{\bf I}^f_{1,\chi}(\,\cdot\,,r_{x_0})\big\|}_{L^\infty(B_{r_{x_0}}(x_0))}\leq c(n,\beta){\|f\|}_{\M^{n/(1-\beta)}(B_{\vrho_{x_0}}(x_0))}\,;
\]
combining this inequality with \eqref{quant:lip}, the new Lipschitz continuity estimate \eqref{new:quantlip} follows. 

\vs

Let us now fix $\bar x\in B_{\bar r_{x_0}/4}(x_0)$, and let $r\leq \bar r_{x_0}/4$. Rewriting estimate \eqref{st:comp:lip} using \eqref{f:mars} and also $E\big(Du;B_{2r}(\bar x)\big)\leq 2{\|Du\|}_{L^\infty(B_{\bar r_{x_0}}(x_0))}$ due to the inclusion $B_{2r}(\bar x)\subset B_{\bar r_{x_0}}(x_0)$, yields 
\begin{multline}\label{st:comp:lip_2}
E\big(Du;B_\vrho(\bar x)\big)\leq c \,\frac{\vrho}{r}\,E\big(Du;B_{2r}(\bar x)\big)+c\,\Big(\frac{r}{\vrho}\Big)^{n/2}\bigg\{\omega_1(r){\|Du\|}_{L^\infty(B_{\bar r_{x_0}}(x_0))}+c\,r^\beta{\|f\|}_{\M^{n/(1-\beta)}(B_{\bar r_{x_0}}(x_0))}\\
+c\,r\, \bar r_{x_0}^{-2s} \tail\big(u-(u)_{B_{\bar r_{x_0}}(x_0)};B_{\bar r_{x_0}}(x_0)\big)+c\,r\bigg[\bar r_{x_0}^{1-2s}+\int_r^{\bar r_{x_0}}\mu^{1-2s} \,\frac{d\mu}{\mu}\bigg]{\|Du\|}_{L^\infty(B_{\bar r_{x_0}}(x_0))}\bigg\}\,,
\end{multline}
where we can redefine
\[
    \omega_1(r)=\Big[1+{\|Du\|}^\beta_{L^\infty(B_{\bar r_{x_0}}(x_0))}\Big]r^{\min\{\beta,2(1-s)\}}.
    \]
We now notice the elementary inequality
\begin{equation*}
r\left[ \bar r_{x_0}^{1-2s} + \int_r^{\bar r_{x_0}}\mu^{1-2s}\frac{d\mu}{\mu} \right] \leq c\, \big(1+\bar r_{x_0}^{1-2s}\big) r^{\min\{\beta,2(1-s)\}}.
\end{equation*}
valid for every $r\in(0,\bar r_{x_0}/4]$: the only non-trivial case in when $s=1/2$, and in this case we can estimate
\[
    r \int_r^{\bar{r}_{x_0}} \mu^{-1} \, d\mu = r \log \frac{\bar{r}_{x_0}}{r} \leq c(\beta)\bar r_{x_0}^{1-\beta} r^\beta
\]
as $\beta<1$. If we insert this into \eqref{st:comp:lip_2}, after multiplying both sides by $\vrho^{n/2}$ and factoring out the smallest power of $r$ appearing in the error terms, namely $r^{\min\{\beta,2(1-s)\}}\geq r$,  after a few algebraic manipulations, we are lead to
\begin{multline}\label{camp:caract}
        \bigg(\int_{B_\vrho(\bar x)}{\big|Du-{(Du)}_{B_{\vrho}(\bar x)}\big|}^2\dx \bigg)^{1/2}\\\leq  \widehat{C}r^{n/2+\min\{\beta,2(1-s)\}}+c\, \Big(\frac{\vrho}{r}\Big)^{n/2+1} \bigg(\int_{B_{2r}(\bar x)}{\big|Du-{(Du)}_{B_{2r}(\bar x)}\big|}^2\dx\bigg)^{1/2} \,,
\end{multline}
where we set
\begin{multline*}
     \widehat{C}=c(\data,\max\{\nu,1\},\beta)\Big[{\|Du\|}_{L^\infty(B_{\bar r_{x_0}}(x_0))}^{1+\beta}+\big(1+\bar r_{x_0}^{1-2s}\big){\|Du\|}_{L^\infty(B_{\bar r_{x_0}}(x_0))}\\
     +{\|f\|}_{\M^{n/(1-\beta)}(B_{\vrho_{x_0}}(x_0))}+\bar r_{x_0}^{-2s}\tail\big(u-(u)_{B_{\bar r_{x_0}}(x_0)} ;B_{\bar r_{x_0}}(x_0)\big)  +1 \Big]\,.
\end{multline*}
Therefore, by setting 
\begin{equation*}
    \phi(\vrho)\equiv \phi_{\bar x}(\vrho):=\bigg(\int_{B_\vrho(\bar x)}{\big|Du-{(Du)}_{B_\vrho(\bar x)}\big|}^2\dx \bigg)^{1/2}\,
\end{equation*}
 from \eqref{camp:caract} we have, absorbing the powers of $2$ into the constants,
\begin{equation}\label{camp:carat1}
    \phi(\vrho)\leq c\, \Big(\frac\vrho R\Big)^{n/2+1} \phi(R)+c\,\widehat{C}R^{n/2+\min\{\beta,2(1-s)\}}\,.
\end{equation}
for all $0<\vrho \leq R=2r\leq \bar r_{x_0}/2$, as the estimate is trivial for $\vrho\in (R/2,R]$. Additionally, by means of \eqref{media} it is immediate to verify that $\vrho\mapsto\phi(\vrho)$ is a monotone function.
This observation and \eqref{camp:carat1} allow us to apply a standard iteration lemma \cite[Lemma 2.1, Chapter III]{giaq}, which yields
\begin{equation*}
    \phi(r)\leq c\left[ \frac{\phi(\bar r_{x_0}/4)}{(\bar r_{x_0})^{n/2+\min\{\beta,2(1-s)\}}}+\widehat{C}\right] r^{n/2+\min\{\beta,2(1-s)\}}\,;
\end{equation*}
from this estimate and \eqref{media}, since $B_{\bar r_{x_0}/4}(\bar x)\subseteq B_{\bar r_{x_0}}(x_0)$, we get
\[
     E(Du;B_r(\bar x))\leq c\Big[\bar r_{x_0}^{-\min\{\beta,2(1-s)\}}\, E\big(Du;B_{\bar r_{x_0}}(x_0)\big) +\widehat{C}\Big]\,r^{\min\{\beta,2(1-s)\}}
\]
for all $0<r\leq \bar r_{x_0}/4$ and all $\bar x\in B_{\bar r_{x_0}/4}(x_0)$, where $c=c(\data,\max\{\nu,1\},\beta)$. This estimate then yields our desired result \eqref{C1:depend} by Campanato's characterization of H\"older continuity \cite[Chapter 3.1]{giaq}. 
\end{proof}
We now move on to the proof of Theorem \ref{thm:c1bfull}; to this end, we need to improve the exponent $2(1-s)$ appearing in the scale factor of the comparison estimate  \eqref{st:comp5}. This in turn boils down to improving the nonlocal estimates \eqref{est:nonlocal}, \eqref{st:comp3} and \eqref{st:comp4}, and this is where we will use the particular form of the nonlocal term $Q^\mathrm{nl}_\Phi$ given by \eqref{ass:phi2} and the newfound H\"older continuity of $Du$ obtained in Theorem \ref{thm:c1b}.

\vs 

Before proving Theorem \ref{thm:c1bfull}, we introduce the following notation and make a preliminary remark. Under assumption \eqref{ass:phi2}, we write $Q^\mathrm{nl}_\Phi=Q^\mathrm{nl}_{a}$.
\begin{remark}
    \rm{For any affine function $\ell:\R^n\to \R^N$, following \cite[Remark 3.4]{kns}, we show that
\begin{equation}\label{aff:zero}
    Q^\mathrm{nl}_a\ell=0\qquad \text{in $\R^n$}
\end{equation}
that holds in the $P.V.$ sense if $s\geq1/2$. In fact, write $\ell(x)=Mx+b$, for some matrix $M\in \R^{N\times n}$ and vector $b\in \R^N$.
Then, for all $x\in \R^n$, we have
\begin{align}\label{aff:uno}
    \int_{\R^n}a(x-y)\,\big(\ell(x)-\ell(y) \big)\,\frac{dy}{|y-x|^{n+2s}}&=\int_{\R^n}a(x-y)\,M(x-y)\,\frac{dy}{|y-x|^{n+2s}}
=\int_{\R^n}a(z)\,Mz\,\frac{dz}{|z|^{n+2s}}=0\,,
\end{align}
where the last equality stems from
\begin{equation*}
\int_{\R^n}a_{\a\b}(z)\,M_{\b j}z_j\,\frac{dz}{|z|^{n+2s}}=0\quad\text{for all $\a,\b=1,\dots,N$, $j=1,\dots,n$,}
\end{equation*}
since the integrand is odd by \eqref{ass:phi1}; these integrals are to be meant in the P.V. sense. From \eqref{aff:uno}, equation~\eqref{aff:zero} follows.
}
\end{remark}
 Let us now introduce the following shorthand notation. For $x_0\in \Omega$ and $\vrho>0$, we define the affine function
\begin{equation}\label{def:ellx}
    \ell_{x_0,\vrho}(x)\equiv\ell_{\vrho}(x):=(u)_{B_\vrho(x_0)}+{(Du)}_{B_\vrho(x_0)}\,(x-x_0),\qquad x\in \R^n\,.
\end{equation}
In the following, we will use several times the following inequality: for all $y\in \R^n$ and $0<\vrho\leq r$, we have
\begin{equation}\label{diff:affini}
     |\ell_{x_0,\vrho}-\ell_{x_0,r}|(x)\leq \Big(\frac{r}{\vrho} \Big)^{n}\mint_{B_r(x_0)}|u-\ell_{x_0,r}|\dx+\Big(\frac{r}{\vrho}\Big)^{n}|x-x_0|\mint_{B_r(x_0)}\big|Du-{(Du)}_{B_r(x_0)}\big|\dx\, .
\end{equation}
To prove it, without loss of generality we may assume that $x_0=0$, and we set $(u)_r=(u)_{B_r}$ and $\ell_\vrho=\ell_{x_0,\vrho}$; then observe that $\mint_{B_\vrho}{(Du)}_r\, x\dx=0$,  so we have
\begin{align*}
    |\ell_\vrho-\ell_r|(x)&=\left|\mint_{B_\vrho}\left[u(z)+Du(z)\,x-(u)_r-{(Du)}_r\,x \right]\,dz \right|\notag\\
   &= \left|\mint_{B_\vrho}\left[u(z)-(u)_r-{(Du)}_r z+Du(z)\,x-{(Du)}_r\,x \right]\,dz \right|\notag\\
     &\leq \mint_{B_\vrho} |u-\ell_r|(z)dz+\big|[{(Du)}_\vrho-{(Du)}_r]\,x\big|\notag\\
     &\leq  \Big(\frac{r}{\vrho} \Big)^{n}\mint_{B_r}|u-\ell_r|\,dz+\Big(\frac{r}{\vrho}\Big)^{n}|x|\mint_{B_r}|Du-{(Du)}_r|\,dz 
\end{align*}
and \eqref{diff:affini} follows. 

\begin{lemma}\label{lemma:affine}
Let $0<R_0\leq 1$, and let $u\in W^{1,1}(B_{R_0}(x_0);\R^N)\cap L^1_{2s}$; suppose that $s>1/2$. Then there exists a constant $c=c(n,N,s)$ such that
{\rm \begin{multline}\label{tailaffine:utile}
\tail\big(u-\ell_{x_0,r};B_{r}(x_0)\big)\leq c\,\left(\frac{r}{R} \right)^{2s}\tail\big(u-\ell_{x_0,R};B_R(x_0)\big)\\
+c\,\int_r^R\Big(\frac{r}{\mu}\Big)^{2s}\mint_{B_{\mu}(x_0)} |u-\ell_{x_0,\mu}|\dx\,\frac{d\mu}{\mu}+c\,r\,\int_r^R{\Big(\frac{r}{\mu}\Big)}^{2s-1}\mint_{B_{\mu}(x_0)} \big|Du-{(Du)}_{B_\mu(x_0)}\big|\dx\,\frac{d\mu}{\mu}\\
+c\,\left(\frac{r}{R}\right)^{2s}\mint_{B_R(x_0)} |u-\ell_{x_0,R}|\dx+c\left(\frac{r}{R}\right)^{2s}R\mint_{B_{R}(x_0)} \big|Du-{(Du)}_{B_R(x_0)}\big|\dx
\end{multline}}
for all $0<r\leq R\leq R_0$. In particular, by Poincar\'e's inequality, there holds
{\rm \begin{multline}\label{tailaffine1:utile}
    \tail\big(u-\ell_{x_0,r};B_{r}(x_0)\big)\leq c\left(\frac{r}{R} \right)^{2s}\tail\big(u-\ell_{x_0,R};B_R(x_0)\big)\\
    +c\,r\int_r^R{\Big(\frac{r}{\mu}\Big)}^{2s-1}\mint_{B_{\mu}(x_0)} \big|Du-{(Du)}_{B_\mu(x_0)}\big|\dx\,\frac{d\mu}{\mu}+c\left(\frac{r}{R}\right)^{2s}R\mint_{B_{R}(x_0)} \big|Du-{(Du)}_{B_R(x_0)}\big|\dx\,.
\end{multline}}
\end{lemma}

\begin{proof}
By translation, we may assume that $x_0=0$ and set $\ell_\vrho=\ell_{0,\vrho}$. The proof follows the argument of \cite[Lemma~3.2]{dmn1}; since the replacement of the constant averages $(u)_{B_\rho}$ by the affine maps $\ell_\rho$ produces additional terms involving the oscillation of $Du$, we include the details. We split
\begin{align*}
    \tail\big(u-\ell_{r};B_{r}\big)\leq &\,r^{2s}\int_{\R^n\setminus B_R}|u-\ell_r|\,\frac{dx}{|x|^{n+2s}}+r^{2s}\int_{B_R\setminus B_r}|u-\ell_r|\,\frac{dx}{|x|^{n+2s}} 
    \\
    \leq & \left(\frac{r}{R} \right)^{2s}\tail\big(u-\ell_R;B_R\big)+r^{2s}\int_{\R^n\setminus B_R}|\ell_R-\ell_r|\,\frac{dx}{|x|^{n+2s}}+r^{2s}\int_{B_R\setminus B_r}|u-\ell_r|\,\frac{dx}{|x|^{n+2s}}.
\end{align*}
We call the last two integrals $T_1$ and $T_2$, respectively. We first study the case $r\leq R/4$. Hence, there exist $\gamma\in (1/4,1/2]$ and an integer $k\geq 2$ such that $r=\g^k R$. By telescoping  and exploiting that $\g\in (1/4,1/2]$ and  \eqref{diff:affini} multiple times, we get
\begin{align}\label{T1:tail}
    T_1&\leq  c\,r^{2s}\sum_{i=0}^{k-1} \int_{\R^n\setminus B_R}\big|\ell_{\g^{i}R}-\ell_{\g^{i+1} R}\big|\,\frac{dx}{|x|^{n+2s}}\notag
    \\
     &\leq  c\left(\frac{r}{R}\right)^{2s}\sum_{i=1}^{k-1}\mint_{B_{\g^{i}R}} |u-\ell_{\g^{i}R}|\dx+c\left(\frac{r}{R}\right)^{2s}R\,\sum_{i=1}^{k-1}\mint_{B_{\g^{i}R}} |Du-{(Du)}_{\g^{i}R}|\dx\notag
     \\
      &\pushright{+\,c\left(\frac{r}{R}\right)^{2s}\mint_{B_R} |u-\ell_{R}|\dx+c\,R\left(\frac{r}{R}\right)^{2s}\mint_{B_{R}} \big|Du-{(Du)}_{R}\big|\dx\,;}
\end{align}
we remark that above we used \eqref{tail.constant}, also replacing $2s$ with $2s-1$, which are admissible since $s>1/2$. Now we shorten
\begin{equation}\label{GH}
G(\rho):=\mint_{B_\rho}|u-\ell_\rho|\dx, \qquad H(\rho):=\mint_{B_\rho}\big|Du-{(Du)}_\rho\big|\dx  
\end{equation}
and we notice that for every $i\in\{1,\ldots,k-1\}$ and every $\mu\in[\gamma^iR,\gamma^{i-1}R]$, by \eqref{diff:affini}, \eqref{media} and the fact that $\gamma\in(1/4,1/2]$, we have $G(\gamma^iR)\leq c\,G(\mu)+c\,\mu H(\mu)$ and $H(\gamma^iR)\leq c\,H(\mu)$. Consequently,
\[
\Big(\frac rR\Big)^{2s} \sum_{i=1}^{k-1}G(\gamma^iR)\leq c\int_r^R \Big(\frac r\mu\Big)^{2s}G(\mu)\,\frac{d\mu}{\mu} + c\,r\int_r^R \Big(\frac r\mu\Big)^{2s-1}H(\mu)\,\frac{d\mu}{\mu}
\]
and, recalling that $2s>1$,
\[
\Big(\frac rR\Big)^{2s}R \sum_{i=1}^{k-1}H(\gamma^iR) \leq c\,r\int_r^R \Big(\frac r\mu\Big)^{2s-1}H(\mu)\,\frac{d\mu}{\mu}\,.
\]
Using the first estimate to bound the first term on the right-hand side in \eqref{T1:tail}, we have
\begin{multline}\label{T11:tail}
\left(\frac{r}{R}\right)^{2s}\sum_{i=1}^{k-1}\mint_{B_{\g^{i}R}} |u-\ell_{\g^{i}R}|\dx\\\leq  c\int_r^R\Big(\frac r\mu\Big)^{2s} \mint_{B_\mu}|u-\ell_\mu|\dx\,\frac{d\mu}{\mu} + c\,r\int_r^R\Big(\frac r\mu\Big)^{2s-1} \mint_{B_\mu}|Du-(Du)_\mu|\dx\,\frac{d\mu}{\mu}\,,
\end{multline}
while using the second one yields
\begin{equation}\label{T12:tail}
\left(\frac{r}{R}\right)^{2s}\,R\,\sum_{i=1}^{k-1}\mint_{B_{\g^{i}R}} \big|Du-{(Du)}_{\g^{i}R}\big|\dx\leq c\,r\int_r^R\Big(\frac r\mu\Big)^{2s-1}\mint_{B_\mu}\big|Du-{(Du)}_\mu\big|\dx\,\frac{d\mu}{\mu}
\end{equation}
with $c=c(n,N,s)$. Combining \eqref{T11:tail}--\eqref{T12:tail} with \eqref{T1:tail}, we get
\begin{multline*}
    T_1\leq  c\,\int_r^R\Big(\frac{r}{\mu}\Big)^{2s}\mint_{B_{\mu}} |u-\ell_{\mu}|\dx\,\frac{d\mu}{\mu}+c\,r\int_r^R\Big(\frac{r}{\mu}\Big)^{2s-1}\mint_{B_{\mu}}\big |Du-{(Du)}_{\mu}\big|\dx\,\frac{d\mu}{\mu}\\
    +c\,\Big(\frac{r}{R}\Big)^{2s}\mint_{B_R} |u-\ell_{R}|\dx+c\,\Big(\frac{r}{R}\Big)^{2s}R\mint_{B_{R}} \big|Du-{(Du)}_{R}\big|\dx\,,
\end{multline*}
  with $c=c(n,N,s)$.  Next we estimate $T_2$: by telescoping and using again that $\g\in (1/4,1/2]$, we find
\begin{align}\label{T21:tail}
    T_2\leq \, r^{2s}\sum_{i=0}^{k-1}\int_{B_{\g^{-i-1}r}\setminus B_{\g^{-i}r}}|u-\ell_r|\,\frac{dx}{|x|^{n+2s}} \leq c(n,s)\sum_{i=0}^{k}\g^{i2s}\mint_{B_{\g^{-i}r}}|u-\ell_r|\dx\,.
\end{align}
On the other hand, for $0\leq i\leq k$, by telescoping and using \eqref{diff:affini} and that $\g\in(1/4,1/2]$, we find
\begin{equation}\label{T22:tail}
    \mint_{B_{\g^{-i}r}}|u-\ell_r|\dx \leq  c\sum_{m=0}^i \mint_{B_{\g^{-m}r}}|u-\ell_{\g^{-m}r}|\dx+c\,\g^{-i}r\sum_{m=0}^i\mint_{B_{\g^{-m}r}}\big|Du-{(Du)}_{\g^{-m}r}\big|\dx\,,
\end{equation}
with $c=c(n,N)$. Merging the content of \eqref{T21:tail}--\eqref{T22:tail}, and using Fubini's theorem on discrete sums, we obtain
\begin{align}\label{t23:tail}
    T_2 & \leq c\sum_{i=0}^k\g^{i2s} \sum_{m=0}^i \mint_{B_{\g^{-m}r}}|u-\ell_{\g^{-m}r}|\dx+c\sum_{i=0}^k\g^{i(2s-1)}r \sum_{m=0}^i \mint_{B_{\g^{-m}r}}|Du-{(Du)}_{\g^{-m}r}|\dx\notag
    \\
    &\leq c\sum_{m=0}^k \mint_{B_{\g^{-m}r}}|u-\ell_{\g^{-m}r}|\dx \sum_{i=m}^\infty \g^{i2s}+ c\,r\,\sum_{m=0}^k \mint_{B_{\g^{-m}r}}\big|Du-{(Du)}_{\g^{-m}r}\big|\dx\sum_{i=m}^\infty \g^{i(2s-1)}\notag
    \\
    &\leq c\sum_{m=0}^k \g^{m2s} \mint_{B_{\g^{-m}r}}|u-\ell_{\g^{-m}r}|\dx+c\,r \sum_{m=0}^k\g^{m(2s-1)} \mint_{B_{\g^{-m}r}}\big|Du-{(Du)}_{\g^{-m}r}\big|\dx
\end{align}
with $c=c(n,N,s)$, since $\g\in (1/4,1/2]$ and $2s>1$, and thus we can majorize the geometric sums. Now we estimate as after \eqref{GH}: for $m=0,\ldots,k-1$, if $\mu\in[\gamma^{-m}r,\gamma^{-(m+1)}r]$ we again have $G(\gamma^{-m}r)\leq c\,G(\mu)+c\mu H(\mu)$ and $H(\gamma^{-m}r)\leq cH(\mu)$. Since $\gamma^{2sm}\approx (r/\mu)^{2s}$, $r\gamma^{(2s-1)m}\approx r (r/\mu)^{2s-1}$ for $\mu\in[\gamma^{-m}r,\gamma^{-(m+1)}r]$, we obtain
\begin{equation}\label{t24:tail}
\sum_{m=0}^k \g^{m2s}G(\gamma^{-m}r)\leq c\int_r^R \Big(\frac r\mu\Big)^{2s}G(\mu)\,\frac{d\mu}{\mu}+ c\,r\int_r^R\Big(\frac r\mu\Big)^{2s-1}H(\mu)\,\frac{d\mu}{\mu}+c\left(\frac rR\right)^{2s}G(R),
\end{equation}
and
\begin{equation}\label{T25:tail}
\sum_{m=0}^k\g^{m(2s-1)}H(\gamma^{-m}r) \leq c\int_r^R\Big(\frac r\mu\Big)^{2s-1}H(\mu)\,\frac{d\mu}{\mu}+c\left(\frac rR\right)^{2s-1}H(R).
\end{equation}
Inserting \eqref{t24:tail}--\eqref{T25:tail} into \eqref{t23:tail}, we obtain
\begin{multline*}
    T_2\leq  c\int_r^R\Big(\frac{r}{\mu} \Big)^{2s}\mint_{B_\mu}|u-\ell_\mu|\dx\,\frac{d\mu}{\mu}+c\,\left( \frac{r}{R}\right)^{2s}\mint_{B_R}|u-\ell_R|\dx\\
+c\,r\int_r^R\Big(\frac{r}{\mu} \Big)^{2s-1}\mint_{B_\mu}\big|Du-{(Du)}_\mu\big|\dx\,\frac{d\mu}{\mu}+c\,R\left(\frac{r}{R} \right)^{2s}\mint_{B_{R}}\big|Du-{(Du)}_{R}\big|\dx\,,
\end{multline*}
where $c=c(n,N,s)$. Merging the estimates for $T_1$ and $T_2$ proves \eqref{tailaffine:utile} when $r\leq R/4$.

\vs 

For $R/4\leq r\leq R$, using \eqref{diff:affini} we simply have
\begin{align*}
    T_1 & \leq c\mint_{B_R}|u-\ell_R|\dx+c\,r^{2s}\mint_{B_R}\big|Du-{(Du)}_R\big|\dx\int_{\R^n\setminus B_R}\frac{dx}{|x|^{n+2s-1}}
    \\
    &\leq c\left( \frac{r}{R}\right)^{2s}\mint_{B_R}|u-\ell_R|\dx+c\left( \frac{r}{R}\right)^{2s} R\mint_{B_R}\big|Du-{(Du)}_R\big|\dx\,,
\end{align*}
and
\begin{align*}
T_2&\leq c\, r^{-n}\int_{B_R\setminus B_r}|u-\ell_r|\dx\leq c\mint_{B_R}|u-\ell_R|\dx+c\mint_{B_R}|\ell_R-\ell_r|\dx\\
&\leq c\left(\frac rR\right)^{2s}\mint_{B_R}|u-\ell_R|\dx + c\left(\frac rR\right)^{2s}R\mint_{B_R}\big|Du-(Du)_R\big|\dx\,.
\end{align*}
\end{proof}

Now we move on to the proof of Theorem \ref{thm:c1bfull}.

\begin{proof}[Proof of Theorem~\ref{thm:c1bfull}]
Theorem~\ref{thm:c1b} already covers the case $s\leq 1/2$; we may therefore assume throughout the proof that $s>1/2$. The argument is a bootstrap on the differentiability exponent: Theorem~\ref{thm:c1b} indeed provides the initial regularity $Du\in C^{0,\beta}$ for every $\beta\leq 2(1-s)$. The special structure of the nonlocal term then allows, at each step, to use the already known H\"older regularity of $Du$ to improve the decay of the nonlocal tail by a further amount $2(1-s)$. This is due to the cancellation of affine functions in the frozen nonlocal operator, which allows the tail to be estimated in terms of the oscillation of $Du$, rather than of $u$ itself. All the estimates below are performed on the regular set $\Omega_u$ obtained in Theorem~\ref{thm:c1b}; the bootstrap does not change this set, but only improves the H\"older exponent of $Du$ on compact subsets of $\Omega_u$.

\vs 

We prove the theorem by induction; more precisely, we show that for every $k\in\N_0$, 
\begin{equation}\label{inindu}
\text{Theorem~\ref{thm:c1bfull} holds for every $\beta\in (0,\min\{2(k+1)(1-s),1\})$.} 
\end{equation}
 The statement for $k=0$ follows from Theorem~\ref{thm:c1b}, since it covers the range $0<\beta\leq 2(1-s)$.  Assume that \eqref{inindu} holds for some $k_0\in\N_0$ (such that $2(k_0+1)(1-s)<1$); let us prove it for $k_0+1$. Let $0<\beta<\min\{2(k_0+2)(1-s),1\}$; if $\beta<2(k_0+1)(1-s)$, then the conclusion is trivial by the inductive hypothesis; thus  we may assume that
\[
2(k_0+1)(1-s)\leq \beta<\min\{2(k_0+2)(1-s),1\}.
\]
Choose $\beta_0<2(k_0+1)(1-s)$ so close to $2(k_0+1)(1-s)$ that $\beta<\beta_0+2(1-s)$; we also choose $\beta_0\neq 2s-1$. Since the assumption $f\in\mathcal M^{n/(1-\beta)}_{\rm loc}$ implies $f\in\mathcal M^{n/(1-\beta_0)}_{\rm loc}$, we obtain $Du\in C^{0,\beta_0}_{\rm loc}(\Omega_u)$ and in particular, for every $x_0\in\Omega_u$, after setting $R_{x_0}=\bar r_{x_0}/32$,
\begin{equation}\label{temp:c1k0}
{[Du]}_{C^{0,\beta_0}(B_{4R_{x_0}}(x_0))}\leq C_{k_0} \,.
\end{equation}
Let now $\bar x\in B_{R_{x_0}/4}(x_0)$ and $0<r\leq R_{x_0}/32$ be fixed, and also let $\ell_{\bar x,2r}=\ell_{2r}$ be the affine function defined in \eqref{def:ellx}.

\vs 

We go back to estimate \eqref{st:comp0} and we improve the estimate for the term $(III)={\llangle Q^\mathrm{nl}_a u, u-h\rrangle}_{\mathrm{nl}}$; by using \eqref{aff:zero}, the linearity of $Q^\mathrm{nl}_a$, \eqref{ass:phi1}--\eqref{ass:phi2} and \eqref{v:extension}, we get
\begin{align}\label{new:nonloc0}
{\llangle Q^\mathrm{nl}_a u, u-h\rrangle}_{\mathrm{nl}} &= {\llangle Q^\mathrm{nl}_a (u-\ell_{2r}), u-h\rrangle}_{\mathrm{nl}}\notag\\
&= \int_{B_{2r}(\bar x)}\int_{B_{2r}(\bar x)}a(y-z)\big[(u-\ell_{2r})(y)-(u-\ell_{2r})(z) \big]\cdot \big[(u-h)(y)-(u-h)(z) \big]\,\frac{dy\,dz}{|y-z|^{n+2s}}\notag\\
&\qquad -2\int_{B_{r}(\bar x)}\int_{\R^n\setminus B_{2r}(\bar x)}a(y-z)\big[(u-\ell_{2r})(y)-(u-\ell_{2r})(z) \big]\cdot (u-h)(z)\,\frac{dy\,dz}{|y-z|^{n+2s}}\notag\\
&\eqqcolon (IV)_1+(IV)_2\,.
\end{align}
Using \eqref{ass:phi1}, \eqref{frac:emb} and \eqref{v:extension}, we find
\begin{align*}
    |(IV)_1| & \leq c\,[u-\ell_{2r}]_{W^{s,2}(B_{2r}(\bar x))}\,[u-h]_{W^{s,2}(B_{2r}(\bar x))}\\
&= c\,r^{2(1-s)}\bigg( \int_{B_{2r}(\bar x)}{\big|Du-{(Du)}_{B_{2r}(\bar x)}\big|}^2\dx\bigg)^{1/2}\bigg( \int_{B_{2r}(\bar x)}|Du-Dh|^2\dx\bigg)^{1/2}\notag
\\
&\leq c\,C_{k_0}\,r^{n/2+\beta_0+2(1-s)}\bigg( \int_{B_{r}(\bar x)}|Du-Dh|^2\dx\bigg)^{1/2}\,,
\end{align*}
with $c=c(n,N,s,\max\{\nu,1\})$, where in the last inequality we used \eqref{v:extension} and \eqref{temp:c1k0}. Then, by \eqref{ass:phi1}, \eqref{tail.constant}, \eqref{v:extension}, \eqref{zbarx}, H\"older's and Poincar\'e's inequalities (which are applicable as $(\ell_{2r})_{B_{2r}(\bar x)}=(u)_{B_{2r}(\bar x)}$) we get 
\begin{align*}
    |(IV)_2| \leq & \,c\,\int_{B_r(\bar x)}|(u-h)(z)|\,\bigg\{\int_{\R^n\setminus B_{2r}(\bar x)}\big(|u(y)-\ell_{2r}(y)|+|u(z)-\ell_{2r}(z)|\big)\,\frac{dy}{|y-\bar x|^{n+2s}} \bigg\}\,dz\notag
\\
    &=  c\,r^{-2s}\tail\big(u-\ell_{2r};B_{2r}(\bar x)\big) \int_{B_r(\bar x)}|(u-h)(z)|\dx+c\,r^{-2s}\int_{B_r(\bar x)}|u-h||u-\ell_{2r}|\dx\notag\\
    &\leq c\,r^{n/2-2s}\tail\big(u-\ell_{2r};B_{2r}(\bar x)\big)\,\bigg(\int_{B_r(\bar x)}|u-h|^2\dx\bigg)^{1/2}\notag\\
    &\pushright{+\,c\,r^{-2s}\bigg(\int_{B_r(\bar x)}|u-h|^2\dx\bigg)^{1/2} \bigg(\int_{B_{2r}(\bar x)}|u-\ell_{2r}|^2\dx\bigg)^{1/2}}\notag\\
    &\leq c\,r^{n/2+1-2s}\tail\big(u-\ell_{2r};B_{2r}(\bar x)\big)\bigg(\int_{B_r(\bar x)}|Du-Dh|^2\dx\bigg)^{1/2} \notag\\
    &\pushright{+\, c\,r^{n/2+2(1-s)}\,\bigg(\int_{B_r(\bar x)}|Du-Dh|^2\dx\bigg)^{1/2}\,\bigg(\mint_{B_{2r}(\bar x)}{\big|Du-{(Du)}_{B_{2r}(\bar x)}\big|}^2\dx\bigg)^{1/2}}\notag\\
    &\leq c\,r^{n/2+1-2s}\tail\big(u-\ell_{2r};B_{2r}(\bar x)\big)\,\bigg(\int_{B_r(\bar x)}|Du-Dh|^2\dx\bigg)^{1/2}\notag\\
    &\pushright{+\, c\,C_{k_0}\,r^{n/2+\beta_0+2(1-s)}\,\bigg(\int_{B_r(\bar x)}|Du-Dh|^2\dx\bigg)^{1/2}\,,}
\end{align*}
with $c=c(\data,\max\{\nu,1\})$, where in the last inequality we used \eqref{temp:c1k0}. Then, by using the tail property \eqref{tailaffine1:utile} and \eqref{temp:c1k0},  we get
\begin{align*}
    r^{-2s}&\tail\big(u-\ell_{2r};B_{2r}(\bar x)\big)\\
    & \leq c\,R_{x_0}^{-2s}\tail\big(u-\ell_{R_{x_0}};B_{R_{x_0}}(\bar x)\big)+c\int_r^{R_{x_0}}\mu^{1-2s}\mint_{B_\mu(\bar x)}\big|Du-{(Du)}_{B_\mu(\bar x)}\big|\dx\,\frac{d\mu}{\mu}\\
    &\pushright{+\,c\,R_{x_0}^{1-2s}\mint_{B_{R_{x_0}}(\bar x)}\big|Du-{(Du)}_{B_{R_{x_0}}(\bar x)}\big| \dx}\notag\\
    &\leq c\,R_{x_0}^{-2s}\tail\big(u-\ell_{R_{x_0}};B_{R_{x_0}}(\bar x)\big)+c\,C_{k_0}\int_r^{R_{x_0}}\mu^{1-2s+\beta_0}\,\frac{d\mu}{\mu}+c\,C_{k_0}R_{x_0}^{1-2s+\beta_0}\notag
     \\
    &\leq c\,R_{x_0}^{-2s}\tail\big(u-\ell_{R_{x_0}};B_{R_{x_0}}(\bar x)\big)+c\,C_{k_0}\big(R_{x_0}^{1-2s+\beta_0}+r^{1-2s+\beta_0}\big)\,,
\end{align*}
where $c=c(n,N,s)$. Moreover, by \eqref{tail:3} and since $s>1/2$, we estimate
\begin{align}\label{tttttail:4}
    R_{x_0}^{-2s}&\tail\big(u-\ell_{R_{x_0}};B_{R_{x_0}}(\bar x)\big)\notag\\
    &\leq  c\,R_{x_0}^{-2s}\tail\big(u-(u)_{B_{R_{x_0}}(\bar x)};B_{R_{x_0}}(\bar x)\big)+c\,\big|{(Du)}_{B_{R_{x_0}}(\bar x)}\big|\int_{\R^n\setminus B_{R_{x_0}}(\bar x)}\frac{dx}{|x-\bar x|^{n+2s-1}}\notag
    \\
    &\leq c\,R_{x_0}^{-2s}\tail\big(u-(u)_{B_{4R_{x_0}}(x_0)};B_{4R_{x_0}}(x_0)\big)+c\,R_{x_0}^{1-2s}{\|Du\|}_{L^\infty(B_{\bar r_{x_0}}(x_0))}
\end{align}
for all $\bar x\in B_{R_{x_0}}(x_0)$, where $c=c(n,N,s)$.
We merge the content of \eqref{new:nonloc0}--\eqref{tttttail:4}, \eqref{temp:liptail}, and \eqref{temp:c1k0}, obtaining
\begin{align}\label{st:comp100}
    \big| \llangle Q^\mathrm{nl}_a u, u-h\rrangle\big|  &\leq c\,r^{n/2+1}\bigg\{R_{x_0}^{-2s}\tail\big(u-(u)_{B_{4R_{x_0}}(x_0)};B_{4R_{x_0}}(x_0)\big)+R_{x_0}^{1-2s}{\|Du\|}_{L^\infty(B_{\bar r_{x_0}}(x_0))}\notag
    \\
    &\qquad\qquad\qquad\qquad{+\,c\,C_{k_0}\big(R_{x_0}^{1-2s+\beta_0}+r^{1-2s+\beta_0}\big)\bigg\}\bigg(\int_{B_r(\bar x)}|Du-Dh|^2\dx\bigg)^{1/2}}\notag\\
    &\leq  \widetilde{C}_{k_0}r^{n/2+\min\{1,\beta_0+2(1-s)\}}\bigg(\int_{B_r(\bar x)}|Du-Dh|^2\dx\bigg)^{1/2}
\end{align}
for all $\bar x\in B_{R_{x_0}}(x_0)$, and $0<r\leq R_{x_0}/32$, with the obvious definition for $\widetilde{C}_{k_0}$. We now combine  \eqref{st:comp0}--\eqref{st:comp2}, \eqref{st:comp100}, \eqref{f:mars} and performing the same algebraic manipulations we obtain the improved comparison estimate
\begin{align}\label{nuovo:comparison}
   \bigg(\int_{B_r(\bar x)}|Du-Dh|^2&\dx\bigg)^{1/2} \notag\\
   &\leq c\,r^{n/2}\Big\{{\|A(\cdot,u(\,\cdot\,))-\bar A\|}_{L^\infty(B_r(\bar x))}+\widetilde{C}_{k_0}r^{\min\{1,\beta_0+2(1-s)\}}\Big\}\bigg(\mint_{B_{2r}(\bar x)}|Du|^2\dx\bigg)^{1/2}\notag\\
   &\pushright{+\,c\,r^{n/2+\beta}{\|f\|}_{\mathcal{M}^{n/(1-\beta)}(B_{\vrho_{x_0}}(x_0))}}\notag\\
   &\leq c\,r^{n/2}\Big\{\big({\|Du\|}_{L^\infty(B_{\bar r_{x_0}}(x_0))}^\beta+1\big)r^\beta+\widetilde{C}_{k_0}r^{\min\{1,\beta_0+2(1-s)\}}\Big\}\,{\|Du\|}_{L^\infty(B_{R_{x_0}}(x_0))}\notag\\
   &\pushright{+\,c\,r^{n/2+\beta}{\|f\|}_{\mathcal{M}^{n/(1-\beta)}(B_{\vrho_{x_0}}(x_0))}}\notag\\
   &\leq  \widehat{C}_{k_0}r^{n/2+\beta}
\end{align}
for all $0< r\leq R_{x_0}/32$, where in the last line we used \eqref{AAx:3}, \eqref{omega:rb}, and the choice $\beta<\min\{1,\beta_0+2(1-s)\}$, with the obvious definition for $\widehat{C}_{k_0}$. Equation~\eqref{nuovo:comparison} together with \eqref{st:comp6} gives
\begin{equation}\label{ult:k0camp}
    \bigg(\int_{B_\vrho(\bar x)}|Du-{(Du)}_{B_{\vrho}(\bar x)}|^2\dx\bigg)^{1/2}\leq c\,\Big(\frac{\vrho}{r} \Big)^{n/2+1}\bigg(\int_{B_r(\bar x)}|Du-{(Du)}_{B_r(\bar x)}|^2\dx\bigg)^{1/2}+\widehat{C}_{k_0}r^{n/2+\beta}\,;
\end{equation}
this holds for all $\bar x\in B_{R_{x_0}}(x_0)$ and $0<\vrho\leq r\leq R_{x_0}/32$. Equation~\eqref{ult:k0camp} together with the standard iteration argument already used after \eqref{camp:carat1} imply
\begin{equation*}
    E\big(Du;B_r(\bar x)\big) \leq c\bigg[\frac{1}{R_{x_0}^\beta}E\big(Du;B_{R_{x_0}}(\bar x)\big)+\widehat{C}_{k_0} \bigg]r^{\beta}\leq c\bigg[\frac{1}{R_{x_0}^\beta}{\|Du\|}_{L^\infty(B_{\bar r_{x_0}}(x_0))}+\widehat{C}_{k_0}\bigg]r^{\beta}
\end{equation*}
for all $r\leq R_{x_0}$ and $\bar x\in B_{R_{x_0}}(x_0)$,  with $c=c(\data,\max\{\nu,1\},\beta)$. This estimate and Campanato's characterization of H\"older continuity prove that $Du\in C^{0,\beta}_{\rm loc}(B_{R_{x_0}}(x_0))$ with quantitative estimate \eqref{c1bfull}, and this completes the inductive step, and hence the proof.
\end{proof}

\subsection{Hierarchy of radii and admissible balls.} \label{hier} 
For the reader's convenience, we summarize the hierarchy of radii used throughout the preceding proofs: several local scales have been introduced at different stages of the argument, each with a specific role. The following recap wants to collect their mutual relations and the corresponding ball inclusions, so that the localization arguments above can be read without ambiguity. For an exhaustion of compact sets $K_i\Subset\Omega$ as defined at the beginning of the Chapter, we define the threshold
\[
\bar r_\beta(K_i) := \min\left\{ r_0,r_1,r_f(K_i),r_A(K_i), \frac14\operatorname{dist}(K_i,\partial\Omega) \right\}\,:
\]
this is the maximal scale on the compact set $K_i$ at which the smallness assumptions on the coefficients and on the datum are valid, and at which the structural restrictions involving $r_0$ and $r_1$ can be used. If $x_0\in\Omega_{u,\beta}$, then, by definition, there are $i=i(x_0)\in\mathbb N$ and $0<\varrho_{x_0}<\bar r_\beta(K_i)$ such that
\[
x_0\in K_i,\qquad B_{\varrho_{x_0}}(x_0)\Subset B_{2\varrho_{x_0}}(x_0)\Subset\Omega, \qquad \mathcal E_\gamma(u;B_{\varrho_{x_0}}(x_0))<\frac{\varepsilon_s}{2}\,.
\]
The radius $\varrho_{x_0}$ is the initial scale at which Proposition \ref{dec:one} is applied.  Next, by the continuity of $x\mapsto \mathcal E_\gamma(u;B_{\varrho_{x_0}}(x))$, we choose a radius $0<r_{x_0}\leq \varrho_{x_0}/8$ so small that
\[
B_{r_{x_0}}(x_0)\Subset \operatorname{int}(K_i),
\qquad
B_{\varrho_{x_0}}(\bar x)\Subset\Omega
\qquad\text{and}\qquad \mathcal E_\gamma(u;B_{\varrho_{x_0}}(\bar x)) < \frac{\varepsilon_s}{2}
\]
for every $\bar x\in B_{r_{x_0}}(x_0)$: this is the localization radius, as it allows us to apply Proposition \ref{dec:one} uniformly with center $\bar x\in B_{r_{x_0}}(x_0)$ and initial scale $\varrho_{x_0}$. In particular, for such centers,  $B_{\varrho_{x_0}}(\bar x)\subset B_{2\varrho_{x_0}}(x_0)\Subset\Omega$.

\vs

In the gradient estimates we shall further reduce the working scale. We choose $0<\bar r_{x_0}\leq r_{x_0}/4$, and this radius is the scale on which the Lipschitz estimate is obtained. It is chosen small enough to satisfy the smallness conditions required in the iteration/summation procedures, and it guarantees the inclusions $B_r(\bar x)\subset B_{2\bar r_{x_0}}(x_0)\subset B_{r_{x_0}}(x_0)$ if $\bar x\in B_{\bar r_{x_0}}(x_0)$ and $0<r\leq \bar r_{x_0}$, and $B_{2r}(\bar x)\subset B_{\bar r_{x_0}}(x_0)$ if $0<r\leq \bar r_{x_0}/4$ and  $\bar x\in B_{\bar r_{x_0}/2}(x_0)$. Finally, in the bootstrap arguments for higher gradient regularity Theorem \ref{thm:c1bfull}, we have to reduce the radius once more and choose $0<R_{x_0}\leq \bar r_{x_0}$. This is the scale on which the already established $C^{1,\beta_0}$ estimate is used as an input for the next step of the bootstrap. If $\bar x\in B_{R_{x_0}/4}(x_0)$ and  $0<r\leq R_{x_0}/8$, then $B_{2r}(\bar x)\subset B_{R_{x_0}/2}(x_0)\subset B_{R_{x_0}}(x_0) \subset B_{\bar r_{x_0}}(x_0)$. Thus the complete hierarchy of radii is
\[
0<R_{x_0}=\frac{\bar r_{x_0}}{32}\leq \bar r_{x_0}\leq \frac{r_{x_0}}4 \leq \frac{\varrho_{x_0}}{32} <\varrho_{x_0} <\bar r_\beta(K_i),
\]
and all balls appearing in the local and nonlocal estimates are contained in the region where the previously established estimates are available.

\section{Full regularity for \texorpdfstring{$u$-}{u-}independent local operators}\label{sec:everywhere}
Suppose that $A(x,u)\equiv A(x)$, where $A(\,\cdot\,)$ satisfies the assumptions of the main theorems. In this case, the modulus of continuity in the $u$-variable \eqref{mod:cont.u} is trivial, that is, $\omegau(\,\cdot\,) \equiv 0$. Returning to the approximate harmonicity result in Proposition~\ref{prop:a}, we note that the point $x_0$ and the radius $\vrho$ no longer need to satisfy the smallness assumption \eqref{excess:small}: this is because \eqref{excess:small} only enters through condition \eqref{chiamoomegau}, which in the current setting is automatically satisfied for any choice of $\e_0$. Consequently, for the remainder of the proofs, we may take  $\Omega_u \equiv \Omega$.

\vs

Let us now track dependence of the radii. First, to verify the assumptions of Theorems \ref{part.BMO} to \ref{thm:hold} and \ref{partial.Holder},  the radii have to satisfy the requirements of Paragraph \ref{general.construction}, except for the smallness condition \eqref{non:cancellare} and its uniform version \eqref{local.Egamma.bdd}. More precisely, if $x_0\in \Omega$, we may set $d={\rm dist}(x_0,\partial \Omega)$ and choose 
\[
\varrho_{x_0} = \min\left\{ r_0,r_1,r_f(\overline{B_{d/2}(x_0)}),r_A(\overline{B_{d/2}(x_0)}), \frac14\,d,\frac12 \right\}\,;
\]
we then set $r_{x_0}:=\varrho_{x_0}/4$. With this choice, for every $\bar x\in B_{r_{x_0}}(x_0)$, we have $B_{\varrho_{x_0}}(\bar x) \subset B_{2\varrho_{x_0}}(x_0) \Subset\Omega$. Thus the balls $B_{\varrho_{x_0}}(x)$, with
$x\in B_{r_{x_0}}(x_0)$, are all admissible, compare with Paragraph \ref{hier}. In this case the choice of $\varrho_{x_0}$ and $r_{x_0}$ depends only on ${\rm dist}(x_0,\partial \Omega)$, the structural radii $r_0,r_1$, and the local radii $r_f(\,\cdot\,),r_A(\,\cdot\,)$, but not on $u$. Additionally, in the case of Theorems \ref{thm:hold} and \ref{partial.Holder}, we can actually quantify $r_f(\,\cdot\,)$, as by \eqref{maximal:f} and \eqref{density.Marc}. Indeed, $r_f$ has to satisfy \eqref{ass.BMO.f0}, and in our case a sufficient condition can be obtained from the maximal quantity appearing in \eqref{maximal:f} in the case of Theorem \ref{partial.Holder}: we have, since $B_{r_{x_0}}(x_0)\subset B_{d/2}(x_0)$ and $\vrho_{x_0}\leq r_f(\overline{B_{d/2}(x_0)})=r_f$, if $r_f$ is smaller than $d/4$,
\[
\sup_{\substack{\bar x\in B_{r_{x_0}}(x_0)\\0<t\leq \vrho_{x_0}}} t^2\bigg(\mint_{B_t(\bar x)}|f|^\chi\dx \bigg)^{1/\chi}\leq \sup_{\substack{\bar x\in B_{d/2}(x_0)\\ 0<t\leq r_f}} t^2 \bigg(\mint_{B_t(\bar x)} |f|^\chi\,dx\bigg)^{1/\chi} \leq r_f^{\beta_0} \sup_{\substack{\bar x\in B_{d/2}(x_0)\\ 0<t\leq d/4}} \, t^{2-\beta_0} \bigg(\mint_{B_t(\bar x)} |f|^\chi\,dx\bigg)^{1/\chi}
\]
and the latter quantity is smaller than $\frac12\varepsilon_s(\data,\beta_0)$ -- and thus \eqref{small.nou.f} is ensured -- as soon as $r_f$ is small enough, with the dependences stated in Theorem \ref{thm:full-u-independent}. Similarly for Theorem \ref{thm:hold},
\[
\sup_{\substack{\bar x\in B_{r_{x_0}}(x_0)\\0<t\leq \vrho_{x_0}}} t^2\bigg(\mint_{B_t(\bar x)}|f|^\chi\dx \bigg)^{1/\chi}\leq c\,r_f^{\beta_0}{\|f\|}_{\M^{n/(2-\beta_0)}(B_{d/2}(x_0))}\,.
\]
For Theorems \ref{thm:lip}--\ref{thm:c1bfull}, the coefficient radius $r_A$ can also be quantified. Here the local $BMO$-smallness condition \eqref{hp:A} is no longer an additional restriction: it follows directly from the global modulus of continuity \eqref{mod:cont.x}. More precisely, if $0<r_A\leq d/4$, then
\[
\sup_{x\in B_{d/2}(x_0)} E_{r_A}(A;x) \leq \Lambda\,\omegax(r_A)<\delta_0(\data,\beta_0(s))\,,
\]
where $\delta_0(\data,\beta_0(s))$ comes from Proposition \ref{dec:one} and it corresponds to the choice of $\beta_0(s)$ in \eqref{cond:beta}, provided that $r_A\leq d/4$ is sufficiently small; in particular, $r_A$ depends only on ${\rm dist}(x_0,\partial\Omega)$, $\data$ and $\omegax(\,\cdot\,)$. Finally, we quantify $r_f$ in the settings of Theorems \ref{thm:lip} and \ref{thm:c1b}; again taking $r_f\leq d/4$, all balls $B_\varrho(x)$ with $x\in B_{d/2}(x_0)$ and $0<\varrho\leq r_f$ are contained in
$B_{3d/4}(x_0)\Subset\Omega$; thus
\[
\sup_{\substack{x\in B_{d/2}(x_0)\\0<t\leq r_f}} t^2\bigg(\mint_{B_t(x)} |f|^\chi\,dx\bigg)^{1/\chi} \leq c\,r_f \big\|{\bf I}^f_{1,\chi}(\,\cdot\,,d/4)\big\|_{L^\infty(B_{d/2}(x_0))}
\]
by \eqref{est.potential.onescale}, for Theorem \ref{thm:lip}. In order to have the latter quantity smaller than $\varepsilon_b(\data,\beta_0(s))$ and, in this way, have the datum smallness required in the preliminary
construction \eqref{ass.BMO.f0} satisfied (this is employed after \eqref{small.ass61}), it is sufficient to choose $r_f$ smaller than a constant depending on $d$, $\data$ and a bound on  $\|{\bf I}^f_{1,\chi}(\,\cdot\,,d/4)\|_{L^\infty(B_{d/2}(x_0))}$. In the setting of Theorems \ref{thm:c1b}--\ref{thm:c1bfull}, by \eqref{density.Marc}, we have
\[
\sup_{\substack{x\in B_{d/2}(x_0)\\0<t\leq r_f}} t^2\bigg(\mint_{B_t(x)} |f|^\chi\,dx\bigg)^{1/\chi} \leq c(\beta)\,r_f^{1+\beta} \|f\|_{\mathcal M^{n/(1-\beta)}(B_{3d/4}(x_0))}\,.
\]
Thus it is enough to choose $r_f\leq d/4$ so small, depending on $d$, $\data$, $\beta$ and any bound on $\|f\|_{\mathcal M^{n/(1-\beta)}(B_{3d/4}(x_0))}$, such that
\[
c(\beta)\,r_f^{1+\beta} \|f\|_{\mathcal M^{n/(1-\beta)}(B_{3d/4}(x_0))}<\varepsilon_b(\data,\beta_0(s))\,,
\]
for the same reason just explained. Finally, the radius $\bar r=\bar r_{x_0}$ has to satisfy the conditions \eqref{lorichiamo} and \eqref{radius:chiamo};  recalling that $\omegau(\,\cdot\,)\equiv 0$, it is then easy to check that $\bar r_{x_0}$ will depend only on the quantities stated in Theorem  \ref{thm:full-u-independent}; the very same reasoning applies to the radius $R_{x_0}=\bar r_{x_0}/8$ appearing in Theorem \ref{thm:c1bfull}.

\bigskip

 \par\noindent {\bf Acknowledgments.} The authors thank L. Brasco for several helpful suggestions on a preliminary version of the paper, which helped improve and clarify the presentation. Part of this work was carried out while C.A. Antonini held a postdoctoral fellowship of the National Institute for Advanced Mathematics (INdAM) at the University of Florence, and previously a postdoctoral position at the University of Parma. His work (and so this research) has been there supported by the University of Parma via the project ``Bando di Ateneo per la Ricerca 2024 -- Azione D''. C.A. Antonini's work was supported by the Italian Ministry of University and Research (MUR) through the FIS 2 project ``{\em SiGmA - Singularities in Geometric Analysis: Minimal Surfaces and Mean Curvature Flows}'', project code FIS-2023-02962 (CUP G53C25000120001). P. Baroni is supported by the European Research Council, through the ERC StG project NEW, nr.~101220121.

\bigskip

%

\end{document}